\documentclass{amsproc}

\address{Simons Center for Geometry and Physics,
State University of New York, Stony Brook, NY 11794-3636 U.S.A.} \email{kfukaya@scgp.stonybrook.edu}
\address{Center for Geometry and Physics, Institute for Basic Sciences (IBS), Pohang, Korea \& Department of Mathematics,
POSTECH, Pohang, Korea} \email{yongoh1@postech.ac.kr}
\address{Graduate School of Mathematics,
Nagoya University, Nagoya, Japan} \email{ohta@math.nagoya-u.ac.jp}
\address{Research Institute for Mathematical Sciences, Kyoto University, Kyoto, Japan}
\email{ono@kurims.kyoto-u.ac.jp}

\usepackage[dvipdfmx]{graphicx}
\usepackage{amsmath}
\usepackage{amscd}
\usepackage{amssymb}
\usepackage{amstext}
\usepackage{amsmath}
\usepackage[all]{xy}
\usepackage{yhmath}
\usepackage{mathrsfs}
\usepackage{bbding}
\usepackage[dvipdfmx]{color}
\usepackage[dvipdfmx]{hyperref}
\hypersetup{
 setpagesize=false,
 bookmarksnumbered=true,%
 bookmarksopen=true,%
 colorlinks=true,%
 linkcolor=blue,
 citecolor=green,
}
\usepackage{color}

\makeatletter
\def\l@section{\@tocline{1}{0pt}{3mm}{8mm}{}}
\def\l@subsection{\@tocline{2}{0pt}{6mm}{10mm}{}}
\def\l@subsubsection{\@tocline{3}{0pt}{9mm}{11mm}{}}
\makeatother

\def\E{\ifmmode{\mathbb E}\else{$\mathbb E$}\fi} %natural numbers
\def\N{\ifmmode{\mathbb N}\else{$\mathbb N$}\fi} %natural numbers%
\def\R{\ifmmode{\mathbb R}\else{$\mathbb R$}\fi} %real numbers
\def\Q{\ifmmode{\mathbb Q}\else{$\mathbb Q$}\fi} %rational numbers
\def\C{\ifmmode{\mathbb C}\else{$\mathbb C$}\fi} %complex numbers
\def\H{\ifmmode{\mathbb H}\else{$\mathbb H$}\fi} %complex numbers
\def\Z{\ifmmode{\mathbb Z}\else{$\mathbb Z$}\fi} %integers
\def\P{\ifmmode{\mathbb P}\else{$\mathbb P$}\fi} %real numbers
\def\T{\ifmmode{\mathbb T}\else{$\mathbb T$}\fi} %real numbers
\def\SS{\ifmmode{\mathbb S}\else{$\mathbb S$}\fi} %real numbers
\def\DD{\ifmmode{\mathbb D}\else{$\mathbb D$}\fi} %real numbers
\def\K{\ifmmode{\mathbb K}\else{$\mathbb K$}\fi}

\theoremstyle{theorem}
\newtheorem{thm}{Theorem}[section]
\newtheorem{cor}[thm]{Corollary}
\newtheorem{lem}[thm]{Lemma}
\newtheorem{sublem}[thm]{Sublemma}

\newtheorem{prop}[thm]{Proposition}

\theoremstyle{definition}
\newtheorem{defn}[thm]{Definition}
\newtheorem{rem}[thm]{Remark}

\newtheorem{conds}[thm]{Condition}

\newtheorem{shitu}[thm]{Situation}

\numberwithin{equation}{section}
\makeindex
\begin{document}

\title[Moduli spaces of pseudoholomorphic disks]{
Construction of Kuranishi structures 
on the moduli spaces of pseudo holomorphic disks: I}
\author{Kenji Fukaya, Yong-Geun Oh, Hiroshi Ohta, Kaoru Ono}

\thanks{Kenji Fukaya is supported partially by NSF Grant No. 1406423
and Simons Collaboration on homological Mirror symmetry
,
Yong-Geun Oh by the IBS project IBS-R003-D1, Hiroshi Ohta by JSPS Grant-in-Aid
for Scientific Research No. 15H02054 and Kaoru Ono by JSPS Grant-in-Aid for
Scientific Research, Nos. 21244002, 26247006.
}

\begin{abstract}
This is the first of two articles in which we provide detailed and 
self-contained account of the construction of a system 
of Kuranishi structures on the moduli spaces of pseudo holomorphic 
disks, using 
the exponential decay estimate given in \cite{foooexp}.
This  article completes the construction 
of a Kuranishi structure of a single moduli space.
This article is an improved version of 
\cite[Part 4]{foootech} and its mathematical content is taken from our
earlier writing \cite{FO,fooobook2,foootech,foooexp}.
\end{abstract}
\maketitle

\date{Oct. 03st, 2017}

%\keywords{pseudo holomorphic curve, Kuranishi structure, symplectic geometry}

\maketitle

%\tableofcontents
%\newpage

\section{Statement of the results}
\label{subsec;discsstatement}

This is the first of two articles which provide detail of the construction of 
a system of Kuranishi structures
on the moduli spaces of pseudo holomorphic disks.
\par
The construction of Kuranishi structure on the 
moduli spaces of pseudo holomorphic curves is a part of the virtual fundamental chain/cycle 
technique which was discovered in 
the year 1996 (\cite{FO,LiTi98,LiuTi98,Rua99,Siebert}).
The case of pseudo holomorphic disks was established and used in \cite{fooobook,fooobook2}.

Let 
$(X,\omega)$ be a symplectic manifold that is tame at infinity and
$L$ a compact Lagrangian submanifold without boundary.
Take an almost complex structure $J$ on $X$ which is tamed by $\omega$ and
let $\beta \in H_2(X,L;\Z)$.
\par
We denote by $\mathcal M_{k+1,\ell}(X,L,J;\beta)$ the compactified
moduli space of stable maps with boundary condition given by $L$ 
and of homology class $\beta$, from marked disks
with $k+1$ boundary and $\ell$ interior marked points.
We require that the enumeration of the boundary marked points 
respects the cyclic order of the 
boundary. (See Definition  \ref{defn24222} for the detail of this definition.)
We can define a topology on 
$\mathcal M_{k+1,\ell}(X,L,J;\beta)$
which is Hausdorff and compact.
(See \cite[Definition 10.3]{FO}, 
\cite[Definition 7.1.42]{fooobook2}, and Definition \ref{defn411}.)
The main result we prove in this article is as follows.
\begin{thm}\label{thm11}
$\mathcal M_{k+1,\ell}(X,L,J;\beta)$ carries a Kuranishi structure 
with corners.
\end{thm}
See Section \ref{sec;Kurareview} for the definition of Kuranishi structure.

This article is not an original research paper but
is a revised version of \cite[Part 4]{foootech}.
Most of the material of this article is taken 
from our previous writing such as 
\cite{FO,fooobook2,foootech,foootech2,foootech21,foooexp,foootoric32,FuFu5}.
The novel points of this article are on its presentation and 
simplifications of the proofs  especially in the following two points.
\par
Firstly we clarify a sufficient condition of the way to take a family of `obstruction spaces' 
so that it produces Kuranishi structure. In other word, we define the notion of 
obstruction bundle data (Definition \ref{defn51}) and show that we can associate a Kuranishi structure 
to given obstruction bundle data 
in a canonical way (Theorem \ref{constthm}).
We also prove the existence of such obstruction bundle data (Theorem \ref{them111}).\footnote{
Certain minor adjustment of the proof becomes necessary for this purpose.
For example, compared to \cite{foootech}, we changed the order 
of the following two process: Solving modified Cauchy-Riemann equation to obtain 
a finite dimensional reduction: Cutting the space of maps by using local transversal.
In \cite{foootech} these two process are performed in this order.
In this article we do it in the opposite order.
Both proofs are correct.}
\par
Secondly we use an `ambient {\it set}' to simplify the construction 
of coordinate change and the proof of its compatibility.
(See Remark \ref{lem710} (2).)
\par
This article studies a single moduli space and constructs its 
Kuranishi structure.
We provide the detail of the proof using the exponential decay estimate in 
\cite{foooexp}.
\par
In the second of this series of articles, we will provide detail 
of the construction of a system of Kuranishi structures 
of the moduli spaces of holomorphic disks so that they are compatible.
More precisely we will construct a tree like K-system as defined in \cite[Definition 21.9]{foootech21}.
\par
We conclude the introduction by reviewing the definition of the 
moduli space $\mathcal M_{k+1,\ell}(X,L,J;\beta)$.

\begin{defn}\label{defn24222}
Let $k,\ell \in \Z_{\ge 0}$.
We denote by $\mathcal M_{k+1,\ell}(X,L,J;\beta)$
\index{00M3{k+1,l}(X,L,J@$\mathcal M_{k+1,\ell}(X,L,J;\beta)$} 
the set of all
$\sim$ equivalence classes of
$((\Sigma,\vec z,\vec {\frak z}),u)$ with the following properties.
\begin{enumerate}
\item
$\Sigma$ is a genus $0$ bordered curve with one boundary component
which has only (boundary or interior) nodal singularities.
\item
$\vec z = (z_0,z_1,\dots,z_k)$ is a $(k+1)$-tuple of boundary marked points.
We assume that they are distinct and are not nodal points.
Moreover we assume that the enumeration respects the counter clockwise cyclic
ordering of the boundary.
\item
$\vec {\frak z} = ({\frak z}_1,\dots,{\frak z}_{\ell})$ is an $\ell$-tuple of interior
marked points. We assume that they are distinct and are not nodal.
\item
$u : (\Sigma,\partial\Sigma) \to (X,L)$ is a continuous map
which is pseudo holomorphic on each irreducible component.
The homology class  $u_*([\Sigma,\partial\Sigma])$  is
$\beta$.
\item
$((\Sigma,\vec z,\vec {\frak z}),u)$ is stable in the sense of
Definition \ref{stabilitydefn27} below.
\end{enumerate}
\end{defn}
We define an equivalence relation $\sim$ in Definition \ref{stabilitydefn27} below.
\begin{defn}\label{stabilitydefn27}
Suppose $((\Sigma,\vec z,\vec {\frak z}),u)$
and $((\Sigma',\vec z^{\,\prime},\vec {\frak z}^{\,\prime}),u')$ satisfy
Definition \ref{defn24222} (1)(2)(3)(4).
We call a homeomorphism $v : \Sigma \to \Sigma'$ an {\it extended isomorphism}
if the following holds.
\begin{enumerate}
\item[(i)]
$v$ is biholomorphic on each irreducible component.
\item[(ii)]
$u' \circ v = u$.
\item[(iii)]
$v(z_j) = z'_j$ and there exists a 
permutation $\sigma : \{1,\dots,\ell\}\to \{1,\dots,\ell\}$
such that $(v({\frak z}_1),\dots,v({\frak z}_{\ell}))$
coincides with $({\frak z}_{\sigma(1)},\dots,{\frak z}_{\sigma(\ell)})$.
\end{enumerate}
We call $v$ an {\it isomorphism} if $\sigma = {\rm id}$ in addition
and $((\Sigma,\vec z,\vec {\frak z}),u) \sim ((\Sigma',\vec z^{\,\prime},\vec {\frak z}^{\,\prime}),u')$
if there exists an isomorphism between them.
\par
The group ${\rm Aut}^+((\Sigma,\vec z),\vec {\frak z}),u)$ of 
{\it extended automorphisms} (resp. ${\rm Aut}((\Sigma,\vec z,\vec {\frak z}),u)$ of
{\it automorphisms}) \index{extended automorphisms} 
consists
of extended isomorphisms (resp. isomorphisms) from  $((\Sigma,\vec z),\vec {\frak z}),u)$ to itself.
\par
The object $((\Sigma,\vec z,\vec {\frak z}),u)$ is said
to be {\it stable} if ${\rm Aut}^+((\Sigma,\vec z,\vec {\frak z}),u)$ is a finite group.
\end{defn}
The whole construction of this article is invariant under the group of 
extended automorphisms. Therefore the Kuranishi structure in Theorem \ref{thm11} is invariant 
under the permutation of the interior marked points.

\section{Universal family of marked disks and spheres}
\label{subsec;universal}

In this section we review well-known facts about the 
moduli spaces of marked spheres and disks.
See \cite{delignemumford,alcorGri,keel} for the detail of the sphere case 
and \cite[Subsection 7.1.5]{fooobook2} for the detail of the disk case.

\begin{prop}\label{prop21}
Let $\ell \ge 3$. There exist complex manifolds $\mathcal M^{\rm s,reg}_{\ell}$, 
$\mathcal C^{\rm s,reg}_{\ell}$ 
\index{00M^{s,reg}_{l}@$\mathcal M^{\rm s,reg}_{\ell}$}
\index{00C^{s,reg}_{l}@$\mathcal C^{\rm s,reg}_{\ell}$}
\footnote{Here $\rm s$ stands for `spheres'.} and holomorphic maps
$$
\pi : \mathcal C^{\rm s,reg}_{\ell} \to \mathcal M^{\rm s,reg}_{\ell}
, \qquad
\frak s_i : \mathcal M^{\rm s,reg}_{\ell} \to \mathcal C^{\rm s,reg}_{\ell}
$$
$i=1,\dots,\ell$, 
\index{00S4_i@$\frak s_i$}
with the following properties.
\begin{enumerate}
\item
$\pi$ is a proper submersion and its fiber $\pi^{-1}({\bf p})$ is biholomorphic to 
Riemann sphere $S^2$.
\item
$\pi \circ \frak s_i$ is the identity.
\item
$\frak s_i({\bf p}) \ne \frak s_j({\bf p})$ for $i \ne j$. 
\item
Let ${\frak z}_1,\dots,{\frak z}_{\ell} \in S^2$
be mutually distinct points. Then there exists 
uniquely a point ${\bf p} \in \mathcal M^{\rm s,reg}_{\ell}$ 
and a biholomorphic map $S^2\to \pi^{-1}({\bf p})$ which 
sends ${\frak z}_i$ to $\frak s_i({\bf p})$.
\item
There exist holomorphic actions of symmetric group ${\rm Perm}(\ell)$ 
of order $\ell!$ on $\mathcal M^{\rm s,reg}_{\ell}$, 
$\mathcal C^{\rm s,reg}_{\ell}$,
which commute with $\pi$ and 
$$
\frak s_{\sigma(i)}(\sigma({\bf p})) =\sigma(\frak s_i({\bf p})).
$$
\item
There exist anti-holomorphic involutions $\tau$ on 
$\mathcal M^{\rm s,reg}_{\ell}$, 
$\mathcal C^{\rm s,reg}_{\ell}$
such that $\pi$ and $\frak s_i$ commute with $\tau$.
The involution $\tau$ commutes with the action of ${\rm Perm}(\ell)$.
\end{enumerate}
\end{prop}
This is well-known and is easy to show.
\par
We can compactify the universal family given in Proposition \ref{prop21} as follows.
\begin{thm}
There exist compact complex manifolds $\mathcal M^{\rm s}_{\ell}$, 
$\mathcal C^{\rm s}_{\ell}$ 
\index{00M^{s}_{l}@$\mathcal M^{\rm s}_{\ell}$}
\index{00C1^{\rm s}_{\ell}@
$\mathcal C^{\rm s}_{\ell}$} containing $\mathcal M^{\rm s,reg}_{\ell}$, 
$\mathcal C^{\rm s,reg}_{\ell}$ as dense subspaces, respectively.
The maps $\pi$ and $\frak s_i$ extend to 
$$
\pi : \mathcal C^{\rm s}_{\ell} \to \mathcal M^{\rm s}_{\ell}
, \qquad
\frak s_i : \mathcal M^{\rm s}_{\ell} \to \mathcal C^{\rm s}_{\ell}
$$
and the following holds.
\begin{enumerate}
\item[(1)']
$\pi$ is proper and holomorphic. For each point $x \in \mathcal C^{\rm s}_{\ell}$ 
at which $\pi$ is not a submersion, we may choose local coordinates so that $\pi$ is given 
locally by $(u_1,\dots,u_m,w_1,w_2) \to (u_1,\dots,u_m,w_1w_2)$ where 
$m = \dim_{\C} \mathcal M^{\rm s}_{\ell} = \ell - 3$.
\item[(2)]
$\pi \circ \frak s_i$ is the identity.
$\pi$ is a submersion on the image of $\frak s_i$.
\item[(4)]
There exist holomorphic actions of symmetric group ${\rm Perm}(\ell)$ 
of order $\ell!$ on $\mathcal M^{\rm s}_{\ell}$, 
$\mathcal C^{\rm s}_{\ell}$,
which commute with $\pi$ and 
$$
\frak s_{\sigma(i)}(\sigma({\bf p})) =\sigma(\frak s_i({\bf p})).
$$
\item[(5)]
There exist anti-holomorphic involutions $\tau$ \index{00T5au@$\tau$} on 
$\mathcal M^{\rm s}_{\ell}$, 
$\mathcal C^{\rm s}_{\ell}$
such that $\pi$ and $\frak s_i$ commute with $\tau$.
$\tau$ also commutes with the action of ${\rm Perm}(\ell)$.
\end{enumerate}
\end{thm}
This is a special case of the marked version of Deligne-Mumford's compactification 
of the moduli space of stable curves (\cite{delignemumford}). 
We can make a similar statement as Proposition \ref{prop21} (4), where we replace 
$(S^2,(\frak z_1,\dots,\frak z_{\ell}))$ by a stable marked curve 
$(\Sigma,\vec{\frak z})$ of genus $0$ with $\ell$ marked points.
\par
We next define the moduli space of 
marked disks.
Let $k,\ell \in \Z_{\ge 0}$.
We define 
$$
\rho_0 : \{0,1,\dots,k+2\ell\} \to \{0,1,\dots,k+2\ell\}
$$
as follows:
$$
\aligned
\rho_0(i) = i, \qquad &\text{$i=0,\dots,k$}, \\
\rho_0(k+2j-1) = k+2j, \qquad &\text{$j=1,\dots,\ell$}, \\
\rho_0(k+2j) = k+2j-1, \qquad &\text{$j=1,\dots,\ell$}.
\endaligned
$$
$\rho_0$ defines a holomorphic involution on ${\mathcal M}^{\rm s}_{k+2\ell+1}$.
We compose it with $\tau$ and obtain an anti-holomorphic involution on 
${\mathcal M}^{\rm s}_{k+2\ell+1}$, which we denote by $\tilde\tau$.\index{00T5autilde@$\tilde\tau$}
We denote an element 
${\bf p} \in {\mathcal M}^{\rm s}_{k+2\ell+1}$
by
$(\pi^{-1}({\bf p}),\vec {\frak z},\vec {\frak z}^{\,+}({\bf p}),\vec {\frak z}^{\,-}({\bf p}))$.
Here
$$
\aligned
\vec {\frak z}({\bf p})
&= (\frak s_0({\bf p}),\dots,\frak s_k({\bf p})),
\\
\vec {\frak z}^{\,+}({\bf p})
&= (\frak s_{k+1}({\bf p}),\dots,\frak s_{k+2j-1}({\bf p}),\dots,\frak s_{k+2\ell-1}({\bf p}))
\\
\vec {\frak z}^{\,-}({\bf p})
&= (\frak s_{k+2}({\bf p}),\dots,\frak s_{k+2j}({\bf p}),\dots,\frak s_{k+2\ell}({\bf p}))
\endaligned
$$
(Here we enumerate $\frak s_i : {\mathcal M}^{\rm s}_{k+2\ell+1}
\to {\mathcal C}^{\rm s}_{k+2\ell+1}$ by $i=0,\dots,k+2\ell$ 
in place of $i=1,\dots,k+2\ell+1$.)
\par
We lift $\tilde{\tau}$ to ${\mathcal C}^{\rm s}_{k+2\ell+1}$ 
as follows. 
Note ${\mathcal C}^{\rm s}_{k+2\ell+1}$ is identified with 
${\mathcal M}^{\rm s}_{k+2\ell+2}$, where 
the projection ${\mathcal C}^{\rm s}_{k+2\ell+1} \to {\mathcal M}^{\rm s}_{k+2\ell+1}$
is identified with the map ${\mathcal M}^{\rm s}_{k+2\ell+2} \to {\mathcal M}^{\rm s}_{k+2\ell+1}$
which forgets the last marked point.
We extend $\rho_0$ to  $\rho_1 : \{0,1,\dots,k+2\ell+1\} \to \{0,1,\dots,k+2\ell+1\}$
by $\rho_1(k+2\ell+1) = k+2\ell+1$. The composition of $\tau : {\mathcal M}^{\rm s}_{k+2\ell+2}
\to {\mathcal M}^{\rm s}_{k+2\ell+2}$ and $\rho_1$ is an anti-holomorphic 
involution $\tilde{\tau}$ on ${\mathcal C}^{\rm s}_{k+2\ell+1}$ which is a lift of the 
involution $\tilde{\tau}$ on 
${\mathcal M}^{\rm s}_{k+2\ell+1}$
\par
Suppose $\tilde{\tau}{\bf p} = {\bf p}$, ${\bf p} \in {\mathcal M}^{\rm s,\rm{reg}}_{k+2\ell+1}$. 
Put $S^2_{\bf p} = \pi^{-1}({\bf p})$.
The restriction of $\tilde{\tau}$, still denoted by $\tilde{\tau}$, 
becomes an anti-holomorphic involution $\tilde{\tau}$ on $S^2_{\bf p}$.
\par
Note 
$\tilde\tau({\frak z}_0) = {\frak z}_0$ by definition. Therefore the fixed point set of 
the anti-holomorphic involution $\tilde\tau : 
S^2_{\bf p} \to S^2_{\bf p}$ is nonempty.
We put $C_{\bf p} = \{ z \in S^2_{\bf p} \mid \tilde\tau({\bf p}) = {\bf p}\}$.
Using the fact $C_{\bf p}$ is nonempty we can show that 
$C_{\bf p}$ is a circle.
\begin{defn}\label{defn23}
We denote by ${\mathcal M}^{\rm d,reg}_{k+1,\ell}$ the set of all 
\index{00M^{s,reg}_{l}M^@$\mathcal M^{\rm s,reg}_{\ell}$}
${\bf p} \in {\mathcal M}^{\rm s,reg}_{k+2\ell+1}$ with the following 
properties.\footnote{Here $\rm d$ stands for `disks'.}
\begin{enumerate}
\item
$\tilde{\tau}{\bf p} = {\bf p}$.
\item
Let $C_{\bf p}$ be as above. We can decompose
$S^2 \setminus C_{\bf p} = {\rm Int}D_+ \cup {\rm Int}D_-$,\footnote{${\rm Int}D_+
= \{z \in {\rm Int} D^2 \mid {\rm Im}(z) \ge 0\}$.
\index{00IntD_+@${\rm Int}D_+$}} where $D_{\pm} 
= {\rm Int}D_{\pm} \cup 
C_{\bf x}$ are disks. 
\item We require that elements of $\vec {\frak z}^+$ are all in 
${\rm Int}D_+$. (It implies that elements of $\vec {\frak z}^-$ are all in 
${\rm Int}D_-$.)
\item
We orient $C_{\bf p}$ by using the (complex) orientation of ${\rm Int}D_+$.
Note ${\frak z}_0,\dots,{\frak z}_k \in C_{\bf p}$. We require the enumeration ${\frak z}_0,\dots,{\frak z}_k$ respects 
the orientation of $C_{\bf p}$.
\end{enumerate}
\par
We denote by ${\mathcal M}^{\rm d}_{k+1,\ell}$ the 
closure of ${\mathcal M}^{\rm d,reg}_{k+1,\ell}$ in ${\mathcal M}^{\rm s}_{k+1+2\ell}$.
\end{defn}
We remark that by definition  ${\mathcal M}^{\rm d,reg}_{k+1,\ell}$ 
is a connected component of the fixed point set of the $\tilde{\tau}$ 
action of ${\mathcal M}^{\rm s,reg}_{k+2\ell+1}$.
We also remark that 
${\mathcal M}^{\rm d,reg}_{k+1,\ell}$ is identified 
with the set of isomorphism classes of $(D^2,\vec z,\vec{\frak z})$
where:
\begin{enumerate}
\item $\vec z = (z_0,\dots,z_{k+1})$, $z_j \in \partial D^2$
are mutually distinct and the 
enumeration respects the orientation.
\item
$\vec {\frak z} = ({\frak z}_1,\dots,{\frak z}_{\ell})$, ${\frak z}_i \in {\rm Int} D^2$
are mutually distinct.
\end{enumerate}
We say $(D^2,\vec z,\vec {\frak z})$ is isomorphic to 
$(D^2,\vec z^{\,\prime},\vec {\frak z}^{\,\prime})$
if there exists a biholomorphic map $v : D^2 \to D^2$  such that
$v(z_i) = z'_i$ and $v({\frak z}_i) = {\frak z}'_i$.
\par
We can use this remark to show the identification:
$$
{\mathcal M}^{\rm d}_{k+1,\ell} \cong \mathcal M_{k+1,\ell}({\rm pt},{\rm pt},J;0).
$$
Here the right hand side is the case of the moduli space $\mathcal M_{k+1,\ell}(X,L,J;\beta)$
when $X$ is a point. ($L$ then is necessarily a point and the homology class $\beta$ is 
$0$.)
Therefore an element of ${\mathcal M}^{\rm d}_{k+1,\ell}$ is an 
equivalence class of an object $(\Sigma,\vec z,\vec {\frak z})$ as in
Definition \ref{defn24222}. (We do not include $u$ here in the notation 
since it is the constant map to the point = $X$.)

\begin{defn} 
We define 
$\partial\mathcal C^{\rm d}_{k+1,\ell}$ as the subspace of ${\mathcal C}^{\rm s}_{k+1+2\ell}$
which consists of the element $x$ such that
\begin{enumerate}
\item $\pi(x) \in {\mathcal M}^{\rm d}_{k+1,\ell}$.
\item
$\tilde\tau(x) = x$.
\end{enumerate}
By construction it is easy to see that there exists an open 
subset $\overset{\circ}{\mathcal C^{\rm d}_{k+1,\ell}}$
of $\pi^{-1}({\mathcal M^{\rm d}_{k+1,\ell}})$ such that, for ${\bf p} \in \mathcal M^{\rm d}_{k+1,\ell}$,
$\pi^{-1}({\mathcal M^{\rm d}_{k+1,\ell}})$ is the disjoint union
$
\overset{\circ}{\mathcal C^{\rm d}_{k+1,\ell}} \cup 
\tilde\tau(\overset{\circ}{\mathcal C^{\rm d}_{k+1,\ell}})
\cup \partial\mathcal C^{\rm d}_{k+1,\ell}
$,
$
\vec{\frak z}^{\,+}({\bf p}) \subset \overset{\circ}{\mathcal C^{\rm d}_{k+1,\ell}}
$,
and that
the enumeration of $\vec z$ respects the boundary orientation of $\partial\overset{\circ}{\mathcal C^{\rm d}_{k+1,\ell}}
\cap \pi^{-1}({\bf p})$.
Such a choice of $\overset{\circ}{\mathcal C^{\rm d}_{k+1,\ell}}$ is unique.
We define 
\index{00C^{d}_{k+1,l}@${\mathcal C^{\rm d}_{k+1,\ell}}$}
\begin{equation}
{\mathcal C^{\rm d}_{k+1,\ell}}
=
\overset{\circ}{\mathcal C^{\rm d}_{k+1,\ell}} \cup 
\partial{\mathcal C^{\rm d}_{k+1,\ell}}.
\end{equation}
\end{defn}
The restrictions of the maps $\pi,\frak s^{\rm d}_j, \frak s^{\rm s}_{i}$ above  define maps
$$
\pi : \mathcal C^{\rm d}_{k+1,\ell} \to \mathcal M^{\rm d}_{k+1,\ell}
, \quad
\frak s^{\rm d}_j : \mathcal M^{\rm d}_{k+1,\ell} \to \partial\mathcal C^{\rm d}_{k+1,\ell},
\quad \frak s^{\rm s}_{i} : \mathcal M^{\rm d}_{k+1,\ell} \to \overset{\circ}{\mathcal C^{\rm d}_{k+1,\ell}}
$$
for
$j=0,\dots,k$, $i=1,\dots,\ell$.
\par
If ${\bf p} \in \mathcal M^{\rm d}_{k+1,\ell}$ is represented by $(\Sigma_{\bf p},\vec z_{\bf p},\vec {\frak z}_{\bf p})$
\index{00Sigma_{p},z_{\bf p},\vec {\frak z}_{\bf p})@$(\Sigma_{\bf p},\vec z_{\bf p},\vec {\frak z}_{\bf p})$}
then the fiber $\pi^{-1}({\bf p})$ is canonically identified with $\Sigma_{\bf p}$.
Moreover $\frak s^{\rm d}_j({\bf p}) = z_{{\bf p},j}$, $\frak s^{\rm s}_i({\bf p}) = \frak z_{{\bf p},i}$, 
via this 
identification. 
\par
We denote by $\frak S^{\rm d}_{k+1,\ell}$ the set of all points $x \in \mathcal C^{\rm d}_{k+1,\ell}$ such that 
it corresponds to a boundary or interior node of $\Sigma_{\bf p}$ by the identification of $\Sigma_{\bf p} \cong \pi^{-1}(\pi(x))$.
\begin{prop}
\begin{enumerate}
\item
$\mathcal C^{\rm d}_{k+1,\ell} \setminus \frak S^{\rm d}_{k+1,\ell}$ is a smooth manifold 
with corner. 
\item
$\pi$ is proper. The restriction of $\pi$ to $\mathcal C^{\rm d}_{k+1,\ell} \setminus \frak S^{\rm d}_{k+1,\ell}$
is a submersion.
\item
$\pi \circ \frak s^{\rm d}_j$, $\pi \circ \frak s^{\rm s}_i$ are the identity maps.
The images of $\frak s^{\rm d}_j$, $\frak s^{\rm s}_i$ do not intersect with $\frak S^{\rm d}_{k+1,\ell}$.
\item
$\frak s^{\rm d}_i({\bf p}) \ne \frak s^{\rm d}_j({\bf p})$, $\frak s^{\rm s}_i({\bf p}) \ne \frak s^{\rm s}_j({\bf p})$ for $i \ne j$. 
\item
There exist smooth actions of the symmetric group ${\rm Perm}(\ell)$ 
of order $\ell!$ on $\mathcal M^{\rm d,reg}_{k+1,\ell}$, 
$\mathcal C^{\rm d,reg}_{k+1,\ell}$,
which commute with $\pi$ and satisfy
$$
\frak s^{\rm s}_{\sigma(i)}(\sigma({\bf p})) =\sigma(\frak s^{\rm s}_i({\bf p})).
$$
\end{enumerate}
\end{prop}
Construction of a smooth structure on 
$\mathcal M^{\rm d}_{k+1,\ell}$ 
is explained in Subsection \ref{subsec:complexstrumod}. The other part of the proof is easy and is omitted.

\section{Analytic family of coordinates at the marked points and 
local trivialization of the universal family}
\label{subsec;markcoordinate}

\subsection{Analytic family of coordinates at the marked points}
\label{subsec:coordinates}

We first recall the notion of an analytic family of 
coordinates introduced in \cite[Section 8]{foooexp}.
Let a stable marked curve 
$(\Sigma_{\bf q},\vec{\frak z}_{\bf q})$ of genus $0$ with $\ell$ marked points represent an element 
$\bf q$ of $\mathcal M^{\rm s}_{\ell}$
and $(\Sigma_{\bf p},\vec z_{\bf p},\vec{\frak z}_{\bf p})$ represent an element 
$\bf p$ of $\mathcal M^{\rm d}_{k+1,\ell}$.
We put
\index{00D1^2_{\circ}@$D^2_{\circ}$} 
\index{00D1_{\circ}^2(+)@$D_{\circ,+}^2$}
$$
D^2_{\circ} = \{ z \in \C \mid \vert z\vert < 1\}, \quad
D_{\circ,+}^2 = \{ z \in \C \mid \vert z\vert < 1, {\rm Im} z \ge 0\}.
$$
\begin{defn}\label{defn31}(\cite[Definition 8.1]{foooexp})
An {\it analytic family of 
coordinates of $\bf q$ (resp. $\bf p$) at the $i$-th interior marked point} 
\index{analytic family of 
coordinates}
is by definition a holomorphic map
$$
\varphi : \mathcal V \times D_{\circ}^2 \to \mathcal C^{\rm s}_{\ell}
\qquad
\text{(resp.
$
\varphi : \mathcal V \times D_{\circ}^2 \to \mathcal C^{\rm d}_{k_+1,\ell}$).}
$$
Here $\mathcal V$ is  a neighborhood of $\bf q$ in $\mathcal M^{\rm s}_{\ell}$
(resp. a neighborhood $\mathcal V$ of $\bf p$ in $\mathcal M^{\rm d}_{k+1,\ell}$).
We require that it has the following properties.
\begin{enumerate}
\item
$\pi \circ \varphi$ coincides with the projection 
$\mathcal V \times D_{\circ}^2 \to \mathcal V$.
\item
$\varphi({\bf x},0) = \frak s_{i}({\bf x})$
(resp. $\varphi({\bf x},0) = \frak s^{\rm s}_{i}({\bf x})$)
for ${\bf x} \in \mathcal V$.
\item
For ${\bf x} \in \mathcal V$ the restriction of $\varphi$ to 
$\{{\bf x}\} \times D_{\circ}^2$ defines a biholomorphic map 
to a neighborhood of $\frak s_{i}({\bf x})$ in $\pi^{-1}({\bf x})$.
(resp. $\frak s^{\rm s}_{i}({\bf x})$ in $\pi^{-1}({\bf x})$).
\end{enumerate}
\end{defn}
We next define an analytic family of 
coordinates at a boundary marked point.
Let $(\Sigma_{\bf p},\vec z_{\bf p},\vec{\frak z}_{\bf p})$ represent an element 
$\bf p$ of $\mathcal M^{\rm d}_{k+1,\ell}$.
By Definition \ref{defn23}, 
$\mathcal M^{\rm d}_{k+1,\ell}$ is a subset of 
$\mathcal M^{\rm s}_{k+1+2\ell}$.
Let ${\bf p}^{\rm s} = (\Sigma^{\rm s},\vec z\cup \vec{\frak z}_{\bf p} \cup \vec{\frak z}^{\,\prime}_{\bf p})$
be a representative of the corresponding element of $\mathcal M^{\rm s}_{k+1+2\ell}$.
In other words, $\Sigma^{\rm s}_{\bf p}$ admits an anti-holomorphic involution 
$\tilde\tau : \Sigma^{\rm s}_{\bf p} \to \Sigma^{\rm s}_{\bf p}$ and
$\Sigma_{\bf p}$ is identified with a subset of $\Sigma^{\rm s}_{\bf p}$,
such that
$\Sigma^{\rm s}_{\bf p} = \Sigma_{\bf p} \cup \tilde\tau(\Sigma_{\bf p})
$.
Moreover
$\partial\Sigma_{\bf p} = \Sigma_{\bf p} \cap \tilde\tau(\Sigma_{\bf p})
$
and 
$\vec{\frak z}^{\,\prime}_{\bf p} = \tilde\tau(\vec{\frak z}_{\bf p})$.

\begin{defn}\label{defn32}(\cite[Definition 8.5]{foooexp})
An {\it analytic family of 
coordinates of $\bf p$ at the $j$-th (boundary) marked point} 
is by definition a holomorphic map
$$
\varphi^{\rm s} : \mathcal V^{\rm s} \times D_{\circ}^2 \to \mathcal C^{\rm s}_{k+1+2\ell}
$$
with the following properties.
\begin{enumerate}
\item
$\mathcal V^{\rm s}$ is a neighborhood of $\bf p^{\rm s}$ in $\mathcal M^{\rm s}_{k+1+2\ell}$
and is $\tilde\tau$ invariant.
\item
$\varphi^{\rm s}$ is an analytic family of 
coordinates at $\frak p^{\rm s}$ of the $j$-th marked point in the sense of Definition \ref{defn31}.
\item
$
\varphi^{\rm s}(\tilde\tau({\bf x}),\overline z) = \tilde\tau(\varphi^{\rm s}({\bf x},z)).
$
\end{enumerate}
\end{defn}

We put $\mathcal V = \mathcal V^{\rm s} \cap \mathcal M^{\rm d}_{k+1,\ell}$.
In the situation of Definition \ref{defn32}
we may replace $\varphi^{\rm s}(\frak v,z)$ by $\varphi^{\rm s}(\frak v,-z)$ if necessary 
and may assume
$
\varphi^{\rm s}(\mathcal V \times D_{\circ,+}^2) \subset \mathcal C^{\rm d}_{k+1,\ell}.
$
We put
\begin{equation}\label{form31}
\varphi = \varphi^{\rm s}\vert_{\mathcal V \times D_{\circ,+}^2}.
\end{equation}
Then for each ${\bf x} \in \mathcal V$, the restriction of 
$\varphi$ to $\{{\bf x}\}\times D_{\circ,+}^2$ defines a 
coordinate of $\pi^{-1}({\bf x})$ at $j$-th boundary coordinate.
The existence of an analytic family of 
coordinates is proved in \cite[Lemma 8.3]{foooexp}.

\subsection{Analytic families of coordinates and 
complex/smooth structure of the moduli space}
\label{subsec:complexstrumod}

In this subsection we use analytic families of coordinates
to describe the complex and/or smooth structures of the moduli space 
of stable marked curves of genus $0$.

Let a stable marked curve 
$(\Sigma_{\bf q},\vec{\frak z}_{\bf q})$ of genus $0$ with $\ell$ marked points 
represent an element 
${\bf q}$ of $\mathcal M^{\rm s}_{\ell}$
and $(\Sigma_{\bf p},\vec z_{\bf p},\vec{\frak z}_{\bf p})$ represent an element 
$\bf p$ of $\mathcal M^{\rm d}_{k+1,\ell}$.
We decompose $\Sigma_{\bf q}$, $\Sigma_{\bf p}$ into irreducible
components as
\begin{equation}\label{form3232}
\Sigma_{\bf q}
= \bigcup_{a \in \mathcal A_{\bf q}} \Sigma_{\bf q}(a),
\qquad
\Sigma_{\bf p}
= \bigcup_{a \in \mathcal A_{\bf p}^{\rm s}} \Sigma_{\bf p}(a)
\cup \bigcup_{a \in \mathcal A_{\bf p}^{\rm d}} \Sigma_{\bf p}(a).
\end{equation}
Here $\Sigma_{\bf q}(a)$ and $\Sigma_{\bf p}(a)$ for 
$a \in \mathcal A_{\bf p}^{\rm s}$ are $S^2$ and 
$\Sigma_{\bf p}(a)$ for 
\index{00Sigma_{\p}(a)@$\Sigma_{\bf p}(a)$}
$a \in \mathcal A_{\bf p}^{\rm d}$ is $D^2$.
\footnote{$\mathcal A_{\bf q}$ etc. are certain index sets.}
\index{00A3_{q}@$\mathcal A_{\bf q}$}
\index{00A3_{p}^{\rm s}@$\mathcal A_{\bf p}^{\rm s}$}
\index{00A3_{p}^{\rm d}@$\mathcal A_{\bf p}^{\rm d}$}
\par
We regard the nodal points and marked points on each irreducible 
component as the marked points on the component. Together with the marked points 
of $\bf p$, $\bf q$, they determine elements
\begin{eqnarray}
{\bf q}_{a} &=& (\Sigma_{\bf q}(a),\vec z_{\bf q}(a)) \in \mathcal M^{\rm s,reg}_{\ell(a)}
\nonumber
\\
{\bf p}_{a} &=& (\Sigma_{\bf p}(a),\vec z_{\bf p}(a)) \in \mathcal M^{\rm s,reg}_{\ell(a)}
\quad (a \in \mathcal A_{\bf p}^{\rm s})\label{formpa1}
\\
{\bf p}_{a} &=& (\Sigma_{\bf p}(a),\vec z_{\bf p}(a),\vec{\frak z}_{\bf p}(a)) \in \mathcal M^{\rm d,reg}_{k(a)+1,\ell(a)}
\quad (a \in \mathcal A_{\bf p}^{\rm d}).\label{formpa2}
\end{eqnarray}

Let $\mathcal V_a$ 
\index{00V3_a@$\mathcal V_a$}  
be a neighborhood of ${\bf q}_{a}$ 
in $ \mathcal M^{\rm s,reg}_{\ell(a)}$ or  ${\bf p}_{a}$ in $\mathcal M^{\rm d,reg}_{k(a)+1,\ell(a)}$.

\begin{defn}\label{defn33}
{\it Analytic families of coordinates at the nodes of ${\bf q}$} are data which assign 
an analytic family of coordinates at each marked point of ${\bf q}_{a}$ 
corresponding to a nodal point of ${\bf q}$ for each $a$.
We require them to be invariant under the extended automorphisms of ${\bf q}$
in the obvious sense.\footnote{In our genus $0$ situation the automorphism group of 
${\bf q}$ is trivial. However there may be a nontrivial {\it extended} automorphism,
which is a biholomorphic map exchanging the marked points.}
Analytic families of coordinates at the nodes of ${\bf p}$ are defined 
in the same way.
\end{defn}

\begin{lem}{\rm(See \cite[Definition-Lemma 8.7]{foooexp})}\label{lem34}
Analytic families of coordinates at the nodes of ${\bf p}$
determine a smooth open embedding
\index{00FPhi@$\Phi$}
\begin{equation}\label{form33}
\Phi : 
\prod_{a\in \mathcal A_{\bf p}^{\rm s} \cup \mathcal A_{\bf p}^{\rm d}} \mathcal V_a 
\times [0,c)^{m_{\rm d}} \times (D^2_{\circ}(c))^{m_{\rm s}}
\to \mathcal M^{\rm d}_{k+1,\ell}, \qquad c < 1/10
\end{equation}
where $m_{\rm d}$ (resp. $m_{\rm s}$) is the number of boundary (resp. interior) 
nodes of $\Sigma_{\bf p}$.
\par
Analytic families of coordinates at the nodes of ${\bf q}$
determine a smooth open embedding
\begin{equation}\label{form34}
\Phi : \prod_{a\in \mathcal A_{\bf q}} \mathcal V_a 
\times (D^2_{\circ}(c))^{m}
\to \mathcal M^{\rm s}_{\ell},
\qquad c < 1/10
\end{equation}
where $m$ is the number of
nodes of $\Sigma_{\bf q}$.
\par
(\ref{form33}) is a diffeomorphism onto a neighborhood of ${\bf p}$.
(\ref{form34}) is a biholomorphic map onto a neighborhood of ${\bf q}$.
(\ref{form33}), (\ref{form34})  are invariant under the extended automorphisms 
of ${\bf p}$, ${\bf q}$, in the obvious sense.
\end{lem}
\begin{rem}
In other words, we specify the smooth and complex structures of $\mathcal M^{\rm s}_{\ell}$
by requiring (\ref{form34}) to be biholomorphic to the image, and specify the smooth structure of 
$\mathcal M^{\rm d}_{k+1,\ell}$ by requiring (\ref{form33}) to be a diffeomorphism 
onto the image .
\end{rem}
\begin{proof}
Below we define the map (\ref{form33}).
See \cite[Section 8]{foooexp} for the definition of (\ref{form34})
and the proof of its holomorphicity.
(We do not use (\ref{form34}) in this article.)
Let $\frak n^{\rm s}_i$ ($i=1,\dots,m_{\rm s}$) be the interior nodes 
of $\Sigma_{\bf p}$ and 
$\frak n^{\rm d}_j$ ($j=1,\dots,m_{\rm d}$)  the boundary nodes 
of $\Sigma_{\bf p}$.
We take $a^{\rm s}_{i,1}, a^{\rm s}_{i,2} \in \mathcal A_{\bf p}^{\rm s} \cup \mathcal A_{\bf p}^{\rm d}$
and $a^{\rm d}_{j,1}, a^{\rm d}_{j,2} \in \mathcal A_{\bf p}^{\rm d}$
such that
$$
\{\frak n^{\rm s}_i\} = \Sigma_{\bf p}(a^{\rm s}_{i,1}) \cap \Sigma_{\bf p}(a^{\rm s}_{i,2}),
\quad
\{\frak n^{\rm d}_j\} = \Sigma_{\bf p}(a^{\rm d}_{j,1}) \cap \Sigma_{\bf p}(a^{\rm d}_{j,2}).
$$
Let $\varphi^{\rm s}_{i,1}$, $\varphi^{\rm s}_{i,2}$, 
$\varphi^{\rm d}_{j,1}$, $\varphi^{\rm d}_{j,2}$ be analytic families of coordinates 
\index{00FPh{\rm s}_{i,1}@$\varphi^{\rm s}_{i,1}$, $\varphi^{\rm s}_{i,2}$, 
$\varphi^{\rm d}_{j,1}$, $\varphi^{\rm d}_{j,2}$}
at those nodal points which we take by assumption.
Suppose
$$
(({\bf x}_a)_{a\in \mathcal A_{\bf q}},(r_j)_{j=1}^{m_{\rm d}},(\sigma_i)_{i=1}^{m_{\rm s}})
\in
\prod_{a\in \mathcal A_{\bf p}^{\rm s} \cup \mathcal A_{\bf p}^{\rm d}} \mathcal V_a 
\times [0,c)^{m_{\rm d}} \times (D^2_{\circ}(c))^{m_{\rm s}}.
$$
We denote 
${\bf x}_a = (\Sigma_{\bf x}(a),\vec z_{\bf x}(a),\vec{\frak z}_{\bf x}(a))$
or 
${\bf x}_a = (\Sigma_{\bf x}(a),\vec z_{\bf x}(a))$.
\par
We consider the disjoint union
\begin{equation}\label{form3535}
\bigcup_{a \in \mathcal A_{\bf p}^{\rm s} \cup \mathcal A_{\bf p}^{\rm d}}\Sigma_{\bf x}(a).
\end{equation}
We remove the (disjoint) union 
\begin{equation}\label{form3636}
\aligned
&\left(\bigcup_{i=1,\dots,m_{\rm s}} (\varphi_{a^{\rm s}_{i,1}}(D_{\circ}^2(\vert\sigma_i\vert)) \cup \varphi_{a^{\rm s}_{i,2}}(D_{\circ}^2(\vert\sigma_i\vert)))\right)
\\
&\cup
\left(\bigcup_{j=1,\dots,m_{\rm d}} (\varphi_{a^{\rm d}_{j,1}}(D_{\circ,+}^2(r_j)) \cup 
\varphi_{a^{\rm d}_{j,2}}(D_{\circ,+}^2(r_j)))\right)
\endaligned
\end{equation}
from (\ref{form3535}).
Here
$$
D^2_{\circ}(c) = \{ z \in \C \mid \vert z\vert < c\}, \quad
D_{\circ,+}^2(c) = \{ z \in \C \mid \vert z\vert < c, {\rm Im} z \ge 0\}.
$$
In case $r_j = 0$ or $\sigma_i = 0$, certain summand of (\ref{form3636}) may be an empty set.
Let 
$$
\Sigma' = (\ref{form3535}) \setminus {(\ref{form3636})}.
$$
When $z_1,z_2 \in D^2 \setminus D^2(\vert\sigma_i\vert))$, we identify
$$
\varphi_{a^{\rm s}_{i,1}}(z_1) \in \Sigma_{\bf x}(a^{\rm s}_{i,1}) 
\quad \text{and} \quad
\varphi_{a^{\rm s}_{i,2}}(z_2) \in \Sigma_{\bf x}(a^{\rm s}_{i,2})
$$
if and only if
$$
z_1z_2 = \sigma_i.
$$
When $z_1,z_2 \in D^2_{+} \setminus D^2_{+}(\vert\sigma_j\vert))$, we identify
$$
\varphi_{a^{\rm d}_{j,1}}(z_1) \in \Sigma_{\bf x}(a^{\rm d}_{j,1}) 
\quad \text{and} \quad
\varphi_{a^{\rm d}_{j,2}}(z_2) \in \Sigma_{\bf x}(a^{\rm d}_{j,2})
$$
if and only if
$$
z_1z_2 = r_j.
$$
In case $r_j = 0$ or $\sigma_i = 0$, we identify the corresponding marked points and obtain a 
nodal point.
Under these identifications, we obtain $\Sigma$ from $\Sigma'$.
\par
The marked points of ${\bf x}_a = (\Sigma_{\bf x}(a),\vec z_{\bf x}(a),\vec{\frak z}_{\bf x}(a))$
or 
${\bf x}_a = (\Sigma_{\bf x}(a),\vec z_{\bf x}(a))$
determine the corresponding marked points on $\Sigma$ in the obvious way. 
We thus obtain an element 
$(\Sigma,\vec z,\vec{\frak z})$ which is 
by definition a representative of the stable marked curve
$\Phi(({\bf x}_a)_{a\in \mathcal A_{\bf q}},(r_j)_{j=1}^{m_{\rm d}},(\sigma_i)_{i=1}^{m_{\rm s}})$.
\end{proof}
We use the next notation in the later (sub)sections.
Let 
$\frak x 
= (({\bf x}_a)_{a\in \mathcal A_{\bf q}},(r_j)_{j=1}^{m_{\rm d}},(\sigma_i)_{i=1}^{m_{\rm s}})$
and 
$\epsilon^{\rm s}_i \in [\vert\sigma_i\vert,1]$, $\epsilon^{\rm d}_j \in [r_j,1]$.
We put $\vec \epsilon = ((\epsilon^{\rm s}_i),(\epsilon^{\rm d}_j))$.
Consider
\begin{equation}\label{form3737}
\aligned
&\left(\bigcup_{i=1,\dots,m_{\rm s}} (\varphi_{a^{\rm s}_{i,1}}(D^2(\epsilon^{\rm s}_i)) \cup \varphi_{a^{\rm s}_{i,2}}(D^2(\epsilon^{\rm s}_i)))\right)
\\
&\cup
\left(\bigcup_{j=1,\dots,m_{\rm d}} (\varphi_{a^{\rm d}_{j,1}}(D^2_+(\epsilon^{\rm d}_j)) 
\cup \varphi_{a^{\rm d}_{j,2}}(D^2_+(\epsilon^{\rm d}_j)))\right).
\endaligned
\end{equation}
We now define
\begin{equation}
\Sigma(\frak x;\vec \epsilon) = (\ref{form3535}) \setminus (\ref{form3737}).
\end{equation}
We write $\Sigma_{\bf x}(\vec{\epsilon}) = \Sigma(\frak x;\vec \epsilon)$
if ${\bf x} =\Phi(\frak x)$. 
\index{00Sigma_{\bf x}(\vec{\epsilon})@$\Sigma_{\bf x}(\vec{\epsilon})$}
In case 
$\Phi(({\bf x}_a)_{a\in \mathcal A_{\bf q}},(r_j)_{j=1}^{m_{\rm d}},(\sigma_i)_{i=1}^{m_{\rm s}}) 
= {\bf p}$ we denote $\Sigma_{\bf p}(\vec \epsilon)$.
(Note $\sigma_i$, $r_j$ are all $0$ in this case in particular.)
\begin{figure}[h]
\centering
\includegraphics[scale=0.7]{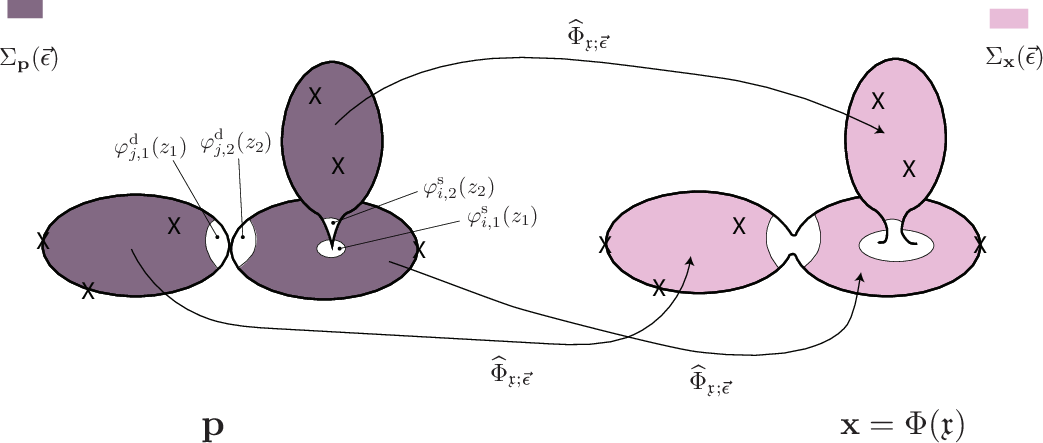}
\caption{The map $\Phi$}
\label{zu1}
\end{figure}

\subsection{Local trivialization of the universal family}
\label{subsec:Localtrivialization}

An important point of the construction of the Kuranishi structure 
is specifying the coordinate of the source curve we use.\footnote{In other words, 
we need to kill the freedom of the action of the group of 
diffeomorphisms of the source curves.}
The construction of the last subsection specifies the coordinate of the 
moduli space (especially its gluing parameter.)
We use one extra datum to specify the coordinate of the source curve.
\par
We use the notation $\bf p$, $\bf q$ etc. as in the last subsection.

\begin{defn}
Let ${\bf p}$ be as in (\ref{formpa1}), (\ref{formpa2})
and $\mathcal V_a$ a neighborhood of its irreducible component ${\bf p}_a$ in the moduli space of marked curves.
A {\it $C^{\infty}$ trivialization} $\phi_a$
\index{00C9^{\infty}trivialization@ $C^{\infty}$ trivialization}
\index{00F4phi_a@$\phi_a$}
of our universal family over  $\mathcal V_a$ is a diffeomorphism
$$
\phi_a : \mathcal V_a \times \Sigma_{\bf p}(a) \to \pi^{-1}(\mathcal V_a)
$$
with the following properties.
Here $\pi^{-1}(\mathcal V_a) \subset \mathcal C^{\rm s,reg}_{\ell(a)}$
or $\pi^{-1}(\mathcal V_a) \subset \mathcal C^{\rm d,reg}_{k(a)+1,\ell(a)}$.
\begin{enumerate}
\item
The next diagram commutes.
$$
\begin{CD}
\mathcal V_a \times \Sigma_{\bf p}(a) @ >{\phi_a}>>
\pi^{-1}(\mathcal V_a) \\
@ VVV @ VV{\pi}V\\
\mathcal V_a @ >{\rm id}>> \mathcal V_a
\end{CD}
$$
where the left vertical arrow is the projection to the first factor.
\item
If $z_j$ (resp $\frak z_i$) is the $j$-th boundary (resp. the $i$-th interior) marked point of  
$\Sigma_{\bf p}(a)$  then
$$
\phi_a({\bf x},z_j) = \frak s^{\rm d}_j({\bf x}),
\quad 
\phi_a({\bf x},\frak z_i) = \frak s^{\rm s}_i({\bf x}).
$$
\item
$\phi_a({\bf o},z) = z$. Here $z \in \Sigma_{\bf p}(a)$ and $\Sigma_{\bf p}(a)$ is regarded as a subset of 
$\mathcal C_{\ell(a)}^{{\rm s},{\rm reg}}$ or of $\mathcal C_{k(a)+1,\ell(a)}^{{\rm d},{\rm reg}}$. 
${\bf o} \in \mathcal V_a$ is the point corresponding to ${\bf p}_a$.
\end{enumerate}
\end{defn}
\begin{defn}\label{defn377}
Suppose we are given analytic families of coordinates at the nodes of ${\bf p}$.
Then we say that the $C^{\infty}$ trivialization {\it $\{\phi_a\}$ is compatible with the families} if the following holds. 
\begin{enumerate}
\item
Suppose that  the $i$-th interior marked point of ${\bf p}_a$ corresponds to a nodal point of 
$\Sigma_{\bf p}(a)$.
Let $\varphi_{a,i} : \mathcal V_a \times D^2_{\circ} \to \pi^{-1}(\mathcal V_a)$ 
be the given analytic family of coordinates at this marked point. Then
$$
\phi_a({\bf x},\varphi_{a,i}({\bf o},z)) = \varphi_{a,i}({\bf x},z).
$$
Here ${\bf o} \in \mathcal V_a$ is the point corresponding to ${\bf p}_a$.
\item
Suppose that  the $j$-th boundary marked point of ${\bf p}_a$ corresponds to a nodal point of 
$\Sigma_{\bf p}(a)$.
Let $\varphi_{a,j} : \mathcal V_a \times D^{2}_{\circ,+} \to \pi^{-1}(\mathcal V_a)$ 
be the given analytic family of coordinates at this marked point. 
(Namely $\varphi_{a,j}$ is the map defined as in (\ref{form31}).) Then
$$
\phi_a({\bf x},\varphi_{a,i}({\bf o},z)) = \varphi_{a,i}({\bf x},z).
$$
Here ${\bf o} \in \mathcal V_a$ is the point corresponding to ${\bf p}_a$.
\end{enumerate}
\end{defn}
Now we define:
\begin{defn}\label{defn37}
{\it Local trivialization data} at ${\bf p}$ consist of the 
\index{local trivialization data}
following:
\begin{enumerate}
\item Analytic families of coordinates at the nodes of ${\bf p}$.
\item A $C^{\infty}$ trivialization $\phi_a$
of our universal family over  $\mathcal V_a$ for each $a$.
We assume it is compatible with the analytic families 
of coordinates.
\item
We require that the data (1)(2) are compatible with the action of extended 
automorphisms of ${\bf p}$ in the obvious sense.
(See \cite[Definition 7.4]{FuFu5}.)
\end{enumerate}
\end{defn}
Let 
$\frak x 
= (({\bf x}_a)_{a\in \mathcal A_{\bf q}},(r_j)_{j=1}^{m_{\rm d}},(\sigma_i)_{i=1}^{m_{\rm s}})$
and 
$\epsilon^{\rm s}_i \in [\vert\sigma_i\vert,1]$, $\epsilon^{\rm d}_j \in [r_j,1]$.
We put $\vec \epsilon = ((\epsilon^{\rm s}_i),(\epsilon^{\rm d}_j))$.
\begin{lem}\label{lem3838}\footnote{See \cite[the paragraph right below (10.1)]{FO}.}
Suppose we are given local trivialization data at ${\bf p}$
and  put $\Phi(\frak x) = (\Sigma_{\frak x},\vec z_{\frak x},\vec{\frak z}_{\frak x})$.
Then the local trivialization data canonically induce a smooth embedding
\index{00FPhi\frak x;\vec \epsilon@$\widehat{\Phi}_{\frak x;\vec \epsilon}$}
$$
\widehat{\Phi}_{\frak x;\vec \epsilon} : \Sigma_{\bf p}(\vec \epsilon) \to \Sigma_{\frak x}
$$
which preserves marked points.
The map $\widehat{\Phi}_{\vec \epsilon} : \mathcal V \times \Sigma_{\bf p}(\vec \epsilon) 
\to \mathcal C^{\rm d}_{k+1,\ell}$ defined by
\begin{equation}\label{hatPhi39}
\widehat{\Phi}_{\vec \epsilon}(\frak x,z)= \widehat{\Phi}_{\frak x;\vec \epsilon}(z)
\end{equation}
is smooth, where $\mathcal V = \mathcal \prod_{a\in \mathcal A_{\bf p}^{\rm s} \cup \mathcal A_{\bf p}^{\rm d}} \mathcal V_a 
\times [0,c)^{m_{\rm d}} \times (D^2_{\circ}(c))^{m_{\rm s}}$ 
is as in (\ref{form33}).
\end{lem}
\begin{proof}
By definition we can define a canonical (holomorphic) embedding 
$\Sigma(\frak x;\vec \epsilon) \subset \Sigma_{\frak x}$.
The $C^{\infty}$ trivialization $\phi_a$ induces a diffeomorphism
$\Sigma(\frak x;\vec \epsilon) \cong \Sigma(\vec \epsilon)$.
\end{proof}

\begin{rem}
The construction of this section is similar to \cite[Section 16]{foootech}.
The only difference is that we use {\it analytic} families of coordinates here 
but {\it smooth} families of coordinates in \cite[Section 16]{foootech}.
The map $\Phi$ in Lemma \ref{lem34} is holomorphic but the corresponding map 
in \cite[Section 16]{foootech} is only smooth.
In that sense the construction here is the same as \cite[Section 8]{foooexp}.
\end{rem}

\section{Stable map topology and $\epsilon$-closeness}
\label{sec;ecloseness}

\subsection{Partial topology}
\label{subsec:weaktopology}

\begin{defn}\label{defn4141411}
Let $\mathcal X$ be a set and $\mathcal M$  its subset.
Suppose we are given a topology on $\mathcal M$, which is metrizable.
A {\it partial topology} \index{partial topology} of $(\mathcal X,\mathcal M)$
assigns $B_{\epsilon}(\mathcal X,{\bf p}) 
\subset \mathcal X$ for each ${\bf p}\in\mathcal M$ and $\epsilon >0$ with the 
following properties. 
\begin{enumerate}
\item ${\bf p}$ is an element of $B_{\epsilon}(\mathcal X,{\bf p})$
and 
$
\{ B_{\epsilon}(\mathcal X,{\bf p}) \cap \mathcal M \mid {\bf p}, \epsilon\}
$
is a basis of the topology of $\mathcal M$.
\item 
For each $\epsilon, {\bf p}$ and ${\bf q} \in B_{\epsilon}(\mathcal X,{\bf p}) \cap \mathcal M$,
there exists $\delta >0$ such that
$
B_{\delta}(\mathcal X,{\bf q}) \subset B_{\epsilon}(\mathcal X,{\bf p}).
$
\item
If $\epsilon_1 < \epsilon_2$ then
$B_{\epsilon_1}(\mathcal X,{\bf p}) \subset B_{\epsilon_2}(\mathcal X,{\bf p})$.
Moreover
$$
\bigcap_{\epsilon} B_{\epsilon}(\mathcal X,{\bf p}) = \{{\bf p}\}.
$$
\end{enumerate}
\par
We say $U \subset \mathcal X$ is a {\it neighborhood} of ${\bf p}$ if 
$U \supset B_{\epsilon}(\mathcal X,{\bf p})$ for some $\epsilon > 0$.
\par
We say two partial topologies are {\it equivalent} if the notion of neighborhood 
coincides.
\end{defn}

\begin{defn}\label{defn3144}
We define 
${\mathcal X}_{k+1,\ell}(X,L,J;\beta)$ to be 
\index{00X3{k+1,\ell}(X,L,J;\beta)@${\mathcal X}_{k+1,\ell}(X,L,J;\beta)$}
the set of all isomorphism classes of $((\Sigma,\vec z,\vec {\frak z}),u)$
which satisfy the same condition as in Definition \ref{defn24222}
except we do not require $u$ to be pseudo holomorphic.
We require $u$ to be continuous and of $C^2$ class on each irreducible component.
\par
We define the notions of isomorphisms and of extended isomorphisms between  
elements of ${\mathcal X}_{k+1,\ell}(X,L,J;\beta)$ in the same way as 
Definition \ref{stabilitydefn27}, requiring (i)(ii)(iii).
The groups of automorphisms ${\rm Aut}({\bf x})$ and of extended automorphisms ${\rm Aut}^+({\bf x})$
of an element ${\bf x} \in {\mathcal X}_{k+1,\ell}(X,L,J;\beta)$ are defined in the same way as 
Definition \ref{stabilitydefn27}.\index{extended automorphisms}
\end{defn}

\begin{prop}\label{prop43}
The pair $({\mathcal X}_{k+1,\ell}(X,L,J;\beta)),{\mathcal M}_{k+1,\ell}(X,L,J;\beta))$
has a partial topology in the sense of Definition \ref{defn4141411}.
Here the topology of ${\mathcal M}_{k+1,\ell}(X,L,J;\beta)$ is the stable 
map topology introduced in \cite[Definition 10.3]{FO}.
\end{prop}

The proof of this proposition will be given in the rest of this section.

\subsection{Weak stabilization data}
\label{subsec:weakstab}

\begin{defn}
An element $((\Sigma,\vec z,\vec {\frak z}),u)$ of ${\mathcal M}_{k+1,\ell}(X,L,J;\beta)$
is called {\it source stable} \index{source stable} if the set of $v : \Sigma \to \Sigma$ 
satisfying Definition \ref{stabilitydefn27} (i)(iii) (but not necessarily (ii)) is finite.
We can define the source stability of an element of ${\mathcal X}_{k+1,\ell}(X,L,J;\beta)$
in the same way.
\end{defn}

\begin{defn}
Let  $I \subset \{1,\dots,\ell+\ell'\}$ with $\# I = 
\ell$. The {\it forgetful map}\index{forgetful map}
$$
\frak{forget}_{\ell+\ell',I} :  {\mathcal M}_{k+1,\ell+\ell'}(X,L,J;\beta)
\to {\mathcal M}_{k+1,\ell}(X,L,J;\beta),
$$
is defined as follows.
Let $((\Sigma,\vec z,\vec {\frak z}),u)\in {\mathcal M}_{k+1,\ell+\ell'}(X,L,J;\beta)$
and $I = \{i_1,\dots,i_{\ell}\}$. ($i_j < i_{j+1}$.)
We put 
$\vec {\frak z}_I = ({\frak z}_{i_1},\cdots,{\frak z}_{i_{\ell}})$ 
and consider $((\Sigma,\vec z,\vec {\frak z}_I),u)$.
If this object is stable then it is 
$\frak{forget}_{\ell+\ell',I}((\Sigma,\vec z,\vec {\frak z}),u)$
by definition.
\par
If not there exists an irreducible component $\Sigma_a$ of $\Sigma$ 
on which $u$ is constant and $\Sigma_a$ is {\it unstable}
in the following sense.
If $\Sigma_a = S^2$ the number of singular or marked points 
on it is less than $3$. If $\Sigma_a = D^2$ 
then $2 m_{\rm s} + m_{\rm d} < 3$. Here $m_{\rm d}$ is the sum of the number of 
boundary nodes on $\Sigma_a$ and the order of $\vec z \cap \Sigma_a$.
$m_{\rm s}$ is the sum of the number of 
interior nodes on $\Sigma_a$ and the order of $\vec {\frak z}_I \cap \Sigma_a$.
\par
We shrink all the unstable components $\Sigma_a$ to points. 
We thus obtain $((\Sigma',\vec z,\vec {\frak z}_I),u)$
which is an element of ${\mathcal M}_{k+1,\ell}(X,L,J;\beta)$.
This is by definition $\frak{forget}_{\ell+\ell',I}((\Sigma,\vec z,\vec {\frak z}),u)$.
See \cite[Lemma 7.1.45]{fooobook2} for more detail.
\par
In case $I = \{1,\dots,\ell\}$ we write $\frak{forget}_{\ell+\ell',\ell}$ in place of 
\index{00Forgetlll@$\frak{forget}_{\ell+\ell,\ell}$}
$\frak{forget}_{\ell+\ell',I}$.
\end{defn}
We define
$
\frak{forget}_{\ell+\ell',I} :  {\mathcal X}_{k+1,\ell+\ell'}(X,L,J;\beta)
\to {\mathcal X}_{k+1,\ell}(X,L,J;\beta),
$
and also $\frak{forget}_{\ell+\ell',\ell}$ among those sets in the same way.
\begin{defn}\label{defn46}
Let ${\bf p} =((\Sigma_{\bf p},\vec z_{\bf p},\vec {\frak z}_{\bf p}),u_{\bf p})\in {\mathcal M}_{k+1,\ell}(X,L,J;\beta)$.
Its {\it weak stabilization data} \index{weak stabilization data} are $\vec {\frak w}_{\bf p} = ({\frak w}_{{\bf p},1},
\dots,{\frak w}_{{\bf p},\ell'})$ 
\index{00W4p@$\vec {\frak w}_{\bf p}$}
with the following properties.
\begin{enumerate}
\item ${\frak w}_{{\bf p},i} \in \Sigma_{\bf p}$.
\item
We put $\vec {\frak z}_{\bf p} \cup \vec {\frak w}_{\bf p} = ({\frak z}_{{\bf p},1},\dots,{\frak z}_{{\bf p},\ell},
{\frak w}_{{\bf p},1},\dots,{\frak w}_{{\bf p},\ell'})$. 
Then $((\Sigma_{\bf p},\vec z_{\bf p},\vec {\frak z}_{\bf p}\cup \vec {\frak w}_{\bf p}),u_{\bf p})$ represents an element of 
${\mathcal M}_{k+1,\ell+\ell'}(X,L,J;\beta)$.
We write this element ${\bf p} \cup \vec {\frak w}_{\bf p}$.
\index{00P4wp@${\bf p} \cup \vec {\frak w}_{\bf p}$}
\item
${\bf p} \cup \vec {\frak w}_{\bf p}$ is source stable.
\item
An arbitrary extended automorphism $v : \Sigma_{\bf p} \to \Sigma_{\bf p}$ of ${\bf p}$
becomes an extended automorphism of ${\bf p} \cup \vec {\frak w}_{\bf p}$.
\end{enumerate}
\end{defn}
\begin{rem}
\begin{enumerate}
\item By definition 
$
\frak{forget}_{\ell+\ell',\ell}({\bf p} \cup \vec {\frak w}_{\bf p})
=
{\bf p}.
$
\item
Condition (4) means that any extended automorphism $v : \Sigma_{\bf p} \to \Sigma_{\bf p}$ preserve
$\vec {\frak w}_{\bf p}$ up to enumeration.
\item
It is easy to prove the existence of weak stabilization data.
\end{enumerate}
\end{rem}
\begin{rem}
\begin{enumerate}
\item
Until Section \ref{subsec;markcoordinate} the symbols ${\bf p}$, ${\bf q}$
were used for the elements of the moduli space of stable marked curves.
From now on  the symbols ${\bf p}$, ${\bf q}$ 
stand for elements of the moduli space 
$\mathcal M_{k+1,\ell}(X,L,J;\beta)$.
\item
The symbol ${\bf x}$ (and ${\bf r}$) stand for the elements of $\mathcal X_{k+1,\ell}(X,L,J;\beta)$.
\item
For ${\bf p}$, ${\bf x}$ etc. we denote its representative by
$((\Sigma_{\bf p},\vec z_{\bf p},\vec{\frak z}_{\bf p}),u_{\bf p})$, 
$((\Sigma_{\bf x},\vec z_{\bf x},\vec{\frak z}_{\bf x}),u_{\bf x})$ and etc..
\item
For an element ${\bf p} = ((\Sigma_{\bf p},\vec z_{\bf p},\vec{\frak z}_{\bf p}),u_{\bf p})$ etc.
we call $(\Sigma_{\bf p},\vec z_{\bf p},\vec{\frak z}_{\bf p})$ its {\it source curve}. 
\index{00Sigma_{\bf p},\vec z_{\bf p},\vec{\frak z}_{\bf p}),u_{\bf p}@$((\Sigma_{\bf p},\vec z_{\bf p},\vec{\frak z}_{\bf p}),u_{\bf p})$}
\index{source curve}
\item
Sometimes we denote by ${\bf p}$ the source curve of ${\bf p}$, by an abuse of notation.
\end{enumerate}
\end{rem}

\subsection{The $\epsilon$-closeness}
\label{subsec:epclo}

\begin{defn}\label{situation38}
Let ${\bf p} = ((\Sigma_{\bf p},\vec z_{\bf p},\vec {\frak z}_{\bf p}),u_{\bf p})\in {\mathcal M}_{k+1,\ell}
(X,L,J;\beta)$.
\begin{enumerate}
\item
We fix its weak stabilization data $\vec {\frak w}_{\bf p}$ (consisting of $\ell'$ marked points).
\item We fix analytic families of coordinates $\{\varphi_{a,i}^{\rm s}\}$, $\{\varphi_{a,j}^{\rm d}\}$ 
at the nodes of ${\bf p} \cup \vec{\frak w}_{\bf p}$ in the sense of Definition \ref{defn33}.
\item We fix a family of $C^{\infty}$ trivializations $\{\phi_a\}$ which is compatible with 
the analytic family of coordinates given in item (2).
\item
We fix a Riemannian metric given on each irreducible component of $\Sigma_{\bf p}$.
\end{enumerate}
We denote the totality of such data by the symbol $\frak W_{\bf p}$ and call it  
\index{00W4_{\bf p}@$\frak W_{\bf p}$, $\frak W_{\bf r}$}  {\it stabilization and trivialization data}.\index{stabilization and trivialization data}
\par
$\frak W_{\bf p}$ induce the data $\frak W_{{\bf p} \cup \vec {\frak w}_{\bf p}}
= (\emptyset,\{\varphi_{a,i}^{\rm s}\}, \{\varphi_{a,j}^{\rm d}\},
\{\phi_a\})$, which are stabilization and trivialization data
of ${\bf p} \cup \vec {\frak w}_{\bf p}$.
Note ${\bf p} \cup \vec {\frak w}_{\bf p}$ is already source stable.
So we do not need to add additional marked points.
\end{defn}
\begin{rem}\label{rem499}
Throughout this paper we fix a Riemannian metric of $X$ and metrics on 
the moduli spaces $\mathcal M^{\rm d}_{k+1,\ell}$, $\mathcal M^{\rm s}_{\ell}$
and the total spaces $\mathcal C^{\rm d}_{k+1,\ell}$, $\mathcal C^{\rm s}_{\ell}$ 
of the universal families.
Since they are all compact the whole construction is independent of such a choice.
\end{rem}
\begin{defn}\label{defn41444}
Let $F : X \to Y$ be a map from a topological space to a metric space.
We say that $F$ {\it has diameter $<\epsilon$}, \index{diameter $<\epsilon$} if the images of all the 
connected components of $X$ have diameter $<\epsilon$ in $Y$.
\end{defn}
\begin{defn}\label{defn411}
Let ${\bf p} = ((\Sigma_{\bf p},\vec z_{\bf p},\vec {\frak z}_{\bf p}),u_{\bf p})\in {\mathcal M}_{k+1,\ell}(X,L,J;\beta)$ 
and $\frak W_{\bf p}$ its stabilization and trivialization data (Definition \ref{situation38}).
Let $\epsilon$ be a sufficiently small positive constant\footnote{We will specify how small it should be below.}.
\par
Let ${\bf x} = ((\Sigma_{\bf x},\vec z_{\bf x},\vec {\frak z}_{\bf x}),u_{\bf x})\in {\mathcal X}_{k+1,\ell}(X,L,J;\beta)$.
We say {\it ${\bf x}$ is $\epsilon$-close to ${\bf p}$ with respect to $\frak W_{\bf p}$} 
\index{00Epsilonclose@$\epsilon$-close} and write 
${\bf x} \in B_{\epsilon}({\mathcal X}_{k+1,\ell}(X,L,J;\beta);{\bf p},\frak W_{\bf p})$ if 
there exists $\vec{\frak w}_{\bf x} = ({\frak w}_{{\bf x},1},\dots,{\frak w}_{{\bf x},\ell'})$ with the following six 
properties.
\begin{enumerate}
\item ${\frak w}_{{\bf x},i} \in \Sigma_{\bf x}$.
\item
We put $\vec {\frak z}_{\bf x} \cup \vec {\frak w}_{\bf x} = ({\frak z}_{{\bf x},1},\dots,{\frak z}_{{\bf x},\ell},
{\frak w}_{{\bf x},1},\dots,{\frak w}_{{\bf x},\ell'})$.
Then $((\Sigma_{\bf x},\vec z_{\bf x},\vec {\frak z}_{\bf x}\cup \vec {\frak w}_{\bf x}),u_{\bf x})$ represents an element of 
${\mathcal X}_{k+1,\ell+\ell'}(X,L,J;\beta)$.
We write this element as ${\bf x} \cup \vec {\frak w}_{\bf x}$.
\item ${\bf x} \cup \vec {\frak w}_{\bf x}$ is source stable.
\item
$(\Sigma_{\bf x},\vec z_{\bf x},\vec {\frak z}_{\bf x}\cup \vec {\frak w}_{\bf x})$ is in the $\epsilon$-neighborhood 
of $(\Sigma_{\bf p},\vec z_{\bf p},\vec {\frak z}_{\bf p}\cup \vec {\frak w}_{\bf p})$ in 
$\mathcal M^{\rm d}_{k+1,\ell+\ell'}$.
\end{enumerate}
\par
We may take $\epsilon$ so small  that (4) above implies that there exists 
$\frak x$ such that $\Phi(\frak x) = (\Sigma_{\bf x},\vec z_{\bf x},\vec {\frak z}_{\bf x})\cup \vec {\frak w}_{\bf x}$.
Now the main part of the conditions is as follows.
We require that there exists $\vec \epsilon = ((\epsilon^{\rm s}_i),(\epsilon^{\rm d}_j))$ such that 
the map $\widehat{\Phi}_{\frak x;\vec \epsilon}$ in Lemma \ref{lem3838} has the following properties.
\begin{enumerate}
\item[(5)]
The $C^2$ difference between the two maps 
$$
u_{\bf x}\circ \widehat{\Phi}_{\frak x;\vec \epsilon} : \Sigma_{\bf p}(\vec \epsilon) \to X
\qquad
\text{and} 
\qquad
u_{\bf p}\vert_{\Sigma_{\bf p}(\vec \epsilon)} : \Sigma_{\bf p}(\vec \epsilon) \to X
$$
is smaller than $\epsilon$.
\item[(6)]
The restriction of $u_{\bf x}$ to $\Sigma_{\bf x} \setminus \Sigma_{\bf x}(\vec \epsilon)$ has diameter $< \epsilon$.
\end{enumerate}
Hereafter we call $\Sigma_{\bf x} \setminus \Sigma_{\bf x}(\vec \epsilon)$ the 
{\it neck region}.\index{neck region}
\end{defn}
\begin{rem}
In case $u_{\bf x}$ is pseudo holomorphic, Condition (5) corresponds to 
\cite[Definition 10.2 (10.2.1)]{FO}
and Condition (6) corresponds to 
\cite[Definition 10.2 (10.2.2)]{FO}.
So Definition \ref{defn411} is an adaptation of the definition of the stable map 
topology (which was introduced in \cite[Definition 10.3]{FO}) 
to the situation when $u_{\bf x}$ is not necessarily pseudo holomorphic.
\par
We remark that in various other references, in place of Condition (6), 
the condition that the energy of $u_{\bf x}$ is close to that of $u_{\bf p}$ is required
\footnote{Such a topology (using energy condition in place of (6)) is sometimes called `Gromov topology'.
We use the name `stable map topology' in order to distinguish it from 
`Gromov topology'.} to define a topology of the moduli space of pseudo holomorphic curves.
In the case when $u_{\bf x}$ is pseudo holomorphic 
this condition on the energy is equivalent to (6)
(when (5) is satisfied). To include the case when $u_{\bf x}$ is not necessarily pseudo holomorphic,
Condition (6) seems to be more suitable than the condition on the energy.
\end{rem}
\begin{lem}\label{lem413}
Let ${\bf p}$ and $\frak W_{\bf p}$ be as in Definition \ref{defn411}.
Then for any sufficiently small $\epsilon>0$ the following holds.
\par
Let ${\bf q} \in {\mathcal M}_{k+1,\ell}(X,L,J;\beta) \cap B_{\epsilon}({\mathcal X}_{k+1,\ell}(X,L,J;\beta);{\bf p},\frak W_{\bf p})$
and $\frak W_{\bf q}$ its stabilization and trivialization data (Definition \ref{situation38}).
Then there exists $\delta>0$ such that:
\begin{equation}\label{lem413occc}
B_{\delta}({\mathcal X}_{k+1,\ell}(X,L,J;\beta);{\bf q},\frak W_{\bf q})
\subset 
B_{\epsilon}({\mathcal X}_{k+1,\ell}(X,L,J;\beta);{\bf p},\frak W_{\bf p}).
\end{equation}
\end{lem}
This is mostly the same as \cite[Lemma 7.26]{FuFu5}
and can be proved in the same way.
See also (the proof of) \cite[Lemma 12.13]{FuFu6}.
We prove it in Section \ref{proofoflem413} for completeness' sake.
\begin{proof}[Proof of Proposition \ref{prop43}]
We take $\frak W_{\bf p}$ for each ${\bf p} \in {\mathcal M}_{k+1,\ell}(X,L,J;\beta)$
and fix them.
We then put
$$
B_{\epsilon}({\mathcal X}_{k+1,\ell}(X,L,J;\beta),{\bf p})
=
B_{\epsilon}({\mathcal X}_{k+1,\ell}(X,L,J;\beta);{\bf p},\frak W_{\bf p}).
$$
Lemma \ref{lem413} implies that this choice satisfies Definition \ref{defn4141411} (2).
Definition \ref{defn4141411} (3) is obvious from construction.
\par
From the definition of the stable map topology \index{stable map topology} on ${\mathcal M}_{k+1,\ell}(X,L,J;\beta)$ 
(\cite[Definition 10.3]{FO} and \cite[Definition 7.1.42]{fooobook2})
we find that the totality of all the subsets
${\mathcal M}_{k+1,\ell}(X,L,J;\beta) \cap B_{\epsilon}({\mathcal X}_{k+1,\ell}(X,L,J;\beta);{\bf p},\frak W_{\bf p})$
moving $\epsilon$, ${\bf p}$, $\frak W_{\bf p}$ is a basis of the 
stable map topology.
Then Lemma \ref{lem413} implies that 
when we  fix ${\bf p} \mapsto \frak W_{\bf p}$, the set 
$\{{\mathcal M}_{k+1,\ell}(X,L,J;\beta) \cap B_{\epsilon}({\mathcal X}_{k+1,\ell}(X,L,J;\beta);{\bf p},\frak W_{\bf p})
\mid {\bf p}, \epsilon\}$ is still a basis of the stable map topology.
This implies Definition \ref{defn4141411} (1).
\end{proof}
\begin{rem}
Lemma \ref{lem413} also implies that the partial topology 
we defined above is independent of the choice of ${\bf p} \mapsto \frak W_{\bf p}$,
up to equivalence.
\end{rem}

\section{Obstruction bundle data}
\label{sec;obstbundledata}

\begin{defn}\label{defn51}
{\it Obstruction bundle data} of the \index{obstruction bundle data}
moduli space ${\mathcal M}_{k+1,\ell}(X,L,J;\beta)$
assign to each 
${\bf p} \in {\mathcal M}_{k+1,\ell}(X,L,J;\beta)$
a neighborhood $\mathscr U_{\bf p}$ of ${\bf p}$ in ${\mathcal X}_{k+1,\ell}(X,L,J;\beta)$
and an object
$E_{\bf p}(\bf x)$ to each ${\bf x} \in \mathscr U_{\bf p}$ . We require that they have the following properties.
\begin{enumerate}
\item
We put ${\bf x} = ((\Sigma_{\bf x},\vec z_{\bf x},\vec {\frak z}_{\bf x}),u_{\bf x})$.
Then $E_{\bf p}({\bf x})$ 
\index{00E1_{p}({\bf x})@$E_{\bf p}({\bf x})$} is a finite dimensional linear subspace of 
the set of $C^2$ sections 
$$
E_{\bf p}({\bf x}) \subset C^2(\Sigma_{\bf x};u_{\bf x}^*TX \otimes \Lambda^{01}),
$$
whose support is away from nodal points.  (See Remark \ref{rem5757}.)
\item (Smoothness)
$E_{\bf p}({\bf x})$ depends smoothly on ${\bf x}$ as defined in  
Definition \ref{defn8911}.
\item (Transversality)
$\{E_{\bf p}({\bf x})\}$ satisfies the transversality condition as in
Definition \ref{defn54564}.
\item (Semi-continuity)
$E_{\bf p}({\bf x})$ is semi-continuous on  ${\bf p}$ as defined in Definition \ref{defn52}.
\item (Invariance under extended automorphisms)
$E_{\bf p}({\bf x})$ is invariant under the extended automorphism group of 
${\bf x}$ as in Condition \ref{conds55}.
\end{enumerate}
For a fixed ${\bf p}$ we call ${\bf x} \mapsto E_{\bf p}({\bf x})$ 
{\it obstruction bundle data at ${\bf p}$} if (1)(2)(3)(5) above are satisfied.
\end{defn}
We now define Conditions (3)(4)(5).
(2) will be defined in Section \ref{sec;obstbundleexpdecay}.
\par
\begin{defn}\label{defn52}
We say $E_{\bf p}({\bf x})$ is {\it semi-continuous}\index{semi-continuous} on  ${\bf p}$ if the following holds.
\par
If ${\bf q} \in \mathscr U_{\bf p} \cap {\mathcal M}_{k+1,\ell}(X,L,J;\beta)$
and 
${\bf x} \in \mathscr U_{\bf p} \cap \mathscr U_{\bf q}$,
then
$$
E_{\bf q}({\bf x}) \subseteq E_{\bf p}({\bf x}).
$$
\end{defn}
We require the transversality condition for ${\bf x} = {\bf p}$ only.
We put ${\bf p} = ((\Sigma_{\bf p},\vec z_{\bf p},\vec {\frak z}_{\bf p}),u_{\bf p})$.
We decompose $\Sigma_{\bf p}$ into irreducible
components as
$$
\Sigma_{\bf p}
= \bigcup_{a \in \mathcal A_{\bf p}^{\rm s}} \Sigma_{\bf p}(a)
\cup \bigcup_{a \in \mathcal A_{\bf p}^{\rm d}} \Sigma_{\bf p}(a).
$$
See (\ref{form3232}).
Let $u_{{\bf p},a}$ be the restriction of $u_{\bf p}$ to $\Sigma_{\bf p}(a)$.
The linearization of the non-linear Cauchy-Riemann equation 
defines a linear elliptic operator
\begin{equation}\label{form51}
\aligned
D_{u_{{\bf p},a}}\overline\partial : &L^2_{m+1}(\Sigma_{\bf p}(a),\partial \Sigma_{\bf p}(a);
u_{{\bf p},a}^*TX,u_{{\bf p},a}^*TL)\\
&\to L^2_{m}(\Sigma_{\bf p}(a);u_{{\bf p},a}^*TX\otimes \Lambda^{01})
\endaligned
\end{equation}
for $a \in \mathcal A_{\bf p}^{\rm d}$  and
\begin{equation}\label{form52}
D_{u_{{\bf p},a}}\overline\partial : L^2_{m+1}(\Sigma_{\bf p}(a);
u_{{\bf p},a}^*TX) \to L^2_{m}(\Sigma_{\bf p}(a);u_{{\bf p},a}^*TX\otimes \Lambda^{01})
\end{equation}
for $a \in \mathcal A_{\bf p}^{\rm s}$.
Here $L^2_{m+1}(\Sigma_{{\bf p}}(a),\partial \Sigma_{{\bf p},a}(a);
u_{{\bf p},a}^*TX,u_{{\bf p},a}^*TL)$ is the space of all sections of the bundle 
$u_{{\bf p},a}^*TX$ of $L^2_{m+1}$-class  whose boundary values lie in $u_{{\bf p},a}^*TL$.
Other spaces are appropriate Sobolev spaces of the sections. Take $m$ sufficiently 
large.
We take a direct sum
\begin{equation}\label{form5333}
\aligned
&\bigoplus_{a \in \mathcal A_{\bf p}^{\rm s}}L^2_{m+1}(\Sigma_{\bf p}(a);
u_{{\bf p},a}^*TX\otimes \Lambda^{01}) \\
&\oplus
\bigoplus_{a \in \mathcal A_{\bf p}^{\rm d}}
L^2_{m+1}(\Sigma_{\bf p}(a),\partial \Sigma_{\bf p}(a);
u_{{\bf p},a}^*TX,u_{{\bf p},a}^*TL).
\endaligned
\end{equation}
We also consider
\begin{equation}\label{form5444}
\bigoplus_{a \in \mathcal A_{\bf p}^{\rm s} \cup \mathcal A_{\bf p}^{\rm d}}L^2_{m}(\Sigma_{\bf p}(a);
u_{{\bf p},a}^*TX\otimes \Lambda^{01}).
\end{equation}
\begin{defn}
We define 
$L^2_{m}(\Sigma_{\bf p};u_{{\bf p}}^*TX\otimes \Lambda^{01})$
\index{00L1^2_{m}(\Sigma_{\bf p};u_{{\bf p}}^*TX\otimes \Lambda^{01})@$L^2_{m}(\Sigma_{\bf p};u_{{\bf p}}^*TX\otimes \Lambda^{01})$}
to be the Hilbert space (\ref{form5444}).
\par
We define a Hilbert space $W^2_{m+1}(\Sigma_{\bf p},\partial\Sigma_{\bf p};u_{{\bf p}}^*TX,u_{{\bf p}}^*TL)$
\index{00W1^2_{m+1}(\Sigma_{\bf p},\partial\Sigma_{\bf p};u_{{\bf p}}^*TX,u_{{\bf p}}^*TL)@$W^2_{m+1}(\Sigma_{\bf p},\partial\Sigma_{\bf p};u_{{\bf p}}^*TX,u_{{\bf p}}^*TL)$}
as the subspace of the Hilbert space (\ref{form5333}) consisting of  elements 
$
\sum_{a \in \mathcal A_{\bf p}^{\rm s} \cup 
\mathcal A_{\bf p}^{\rm d}}V_a
$
(where $V_a$ is a section on $\Sigma_{\bf p}(a)$)
with the following properties.
Let $p \in \Sigma_{\bf p}$ be a nodal point. We take $a_1(p)$, $a_2(p)$
such that
$\{p\} = \Sigma_{\bf p}(a_1(p)) \cap \Sigma_{\bf p}(a_2(p))$. We require
$$
V_{a_1(p)}(p) = V_{a_2(p)}(p).
$$
We require this condition at all the nodal points $p$.
\end{defn}
The operators (\ref{form51}), (\ref{form52}) induce a Fredholm  operator
\begin{equation}\label{form523}
D_{u_{{\bf p}}}\overline\partial : W^2_{m+1}(\Sigma_{\bf p},\partial\Sigma_{\bf p};u_{{\bf p}}^*TX,u_{{\bf p}}^*TL) \to L^2_{m}(\Sigma_{\bf p};u_{{\bf p}}^*TX\otimes \Lambda^{01}).
\end{equation}
\begin{rem}
We define $L^2_{m}(\Sigma_{\bf x};u_{{\bf x}}^*TX\otimes \Lambda^{01})$,
$W^2_{m+1}(\Sigma_{\bf x},\partial\Sigma_{\bf x};u_{{\bf x}}^*TX,u_{{\bf x}}^*TL)$
and the operator $D_{u_{{\bf x}}}\overline\partial$ between them for 
${\bf x} \in {\mathcal X}_{k+1,\ell}(X,L,J;\beta)$ in the same way. (Here $u_{{\bf x}}$ may not 
be pseudo holomorphic but is of $L^2_{m+1}$ class.)
\end{rem}
Now we describe the transversality condition.
When ${\bf x} = {\bf p}$  
we require $E_{\bf p}({\bf p})$ consists of smooth sections
as a part of Definition \ref{defn51} (2).
(See Definition \ref{defn86} (1).)

\begin{defn}\label{defn54564}
We say that $\{E_{\bf p}({\bf x})\}$ satisfies the {\it transversality condition}\index{transversality condition}
if
$$
{\rm Im}(\ref{form523}) + E_{\bf p}({\bf p}) = L^2_{m}(\Sigma_{\bf p};u_{{\bf p}}^*TX\otimes \Lambda^{01}).
$$
\end{defn}
By ellipticity this condition is independent of $m$.
\par
We next describe Definition \ref{defn51} (5).
Let $v : \Sigma_{\bf x} \to \Sigma_{\bf x}$ be an extended 
automorphism. It induces an isomorphism
$$
v_* : C^2(\Sigma_{{\bf x}};u_{{\bf x}}^*TX\otimes \Lambda^{01}) 
\to C^2(\Sigma_{{\bf x}};u_{{\bf x}}^*TX\otimes \Lambda^{01})
$$
since $v$ is biholomorphic and $u_{{\bf x}}\circ v = u_{\bf x}$.
In case ${\bf x} = {\bf p}$ the group ${\rm Aut}^+({\bf p})$  acts also on the domain and target of 
(\ref{form523}) and the operator $D_{u_{{\bf p}}}\overline\partial$ is invariant under this action.
Let $\frak{aut}(\Sigma_{\bf p},\vec z_{\bf p},\vec{\frak z}_{\bf p})$ be the Lie algebra 
of the group of automorphisms of the source curve $(\Sigma_{\bf p},\vec z_{\bf p},\vec{\frak z}_{\bf p})$ of 
${\bf p}$. We can embed it into the kernel of $D_{u_{\bf p}}\overline{\partial}$ by 
differentiating $u_{\bf p}$, so that it becomes ${\rm Aut}({\bf p})$ invariant.
\begin{conds}\label{conds55}
We require $v_*(E_{\bf p}({\bf x})) = E_{\bf p}({\bf x})$ for any $v \in {\rm Aut}^+({\bf x})$.
\par
We also assume  that
the action of the group of automorphisms 
${\rm Aut}({\bf p})$ of ${\bf p}$ on 
$(D_{u_{{\bf p}}}\overline\partial)^{-1}(E_{\bf p}({\bf p}))/\frak{aut}(\Sigma_{\bf p},\vec z_{\bf p},\vec{\frak z}_{\bf p})$ is effective,
where $D_{u_{{\bf p}}}\overline\partial$ is as in (\ref{form523}).
\end{conds}
\begin{rem}\label{rem5757}
Note an element of $\mathcal X_{k+1,\ell}(X,L,J;\beta)$ is an  {\it equivalence class} 
of objects ${\bf x} = ((\Sigma_{\bf x},\vec z_{\bf x},\vec {\frak z}_{\bf x}),u_{\bf x})$.
Therefore for the data $E_{\bf p}({\bf x})$ to be well-defined we need to assume the following.
\begin{enumerate}
\item[(*)]
If $v : \Sigma_{\bf x} \to \Sigma_{\bf x'}$ is an isomorphism from 
${\bf x}$ to ${\bf x}' = ((\Sigma_{\bf x'},\vec z_{\bf x'},\vec {\frak z}_{\bf x'}),u_{\bf x'})$
then $v_*(E_{\bf p}({\bf x})) = E_{\bf p}({\bf x'})$.
\end{enumerate}
\par
We include this condition as a part of Definition \ref{defn51} (1).
In particular (*) implies that $E_{\bf p}({\bf x})$ is invariant under the action of ${\rm Aut}({\bf x})$.
The first half of Condition \ref{conds55} is slightly stronger than (*).
We add the second half of  Condition \ref{conds55} so that 
orbifolds appearing in our Kuranishi structure become 
effective. 
\end{rem}
(*) and Condition \ref{conds55} imply the next lemma.
Let $v : \Sigma_{\bf p} \to \Sigma_{\bf p}$ be an extended 
automorphism.
Let ${\bf x} \in \mathscr U_{\bf p}$.
We may write ${\bf x} = \Phi(\frak x)$.
Here $\Phi$ is the map in Lemma \ref{lem34}.
The map $v$ induces a map
$$
v_* : \prod_{a\in \mathcal A_{\bf p}^{\rm s} \cup \mathcal A_{\bf p}^{\rm d}} \mathcal V_a 
\times [0,c)^{m_{\rm d}} \times (D^2_{\circ}(c))^{m_{\rm s}}
\to 
\prod_{a\in \mathcal A_{\bf p}^{\rm s} \cup \mathcal A_{\bf p}^{\rm d}} \mathcal V_a 
\times [0,c)^{m_{\rm d}} \times (D^2_{\circ}(c))^{m_{\rm s}}.
$$
We put $v_*({\bf x}) = \Phi(v_*(\frak x))$.
$v \mapsto v_*$ determines an action of the group 
${\rm Aut}^+({\bf p})$ of extended automorphisms on
$\mathscr U_{\bf p}$.
\par
By Definition \ref{defn37} (3), Definition \ref{defn46} (4) etc. 
$v$ induces a biholomorphic map
$
\hat v : \Sigma_{{\bf x}} \to \Sigma_{v_*({\bf x})}
$
such that
$u_{v_*({\bf x})} \circ \hat v = u_{{\bf x}}$, 
$\hat v(z_{{\bf x},j}) = z_{v_*({\bf x}),j}$ and 
$\hat v(\frak z_{{\bf x},i}) = \frak z_{v_*({\bf x}),\sigma(i)}$.
Here $\sigma$ is the permutation such that
$v(\frak z_{{\bf p},i}) = \frak z_{{\bf p},\sigma(i)}$.
\par
Therefore the map $\hat v$ induces an isomorphism
\begin{equation}\label{form5666}
\hat v_* : L^2_{m}(\Sigma_{{\bf x}};u_{{\bf x}}^*TX\otimes \Lambda^{01}) 
\to L^2_{m}(\Sigma_{v_*({\bf x})};u_{v_*({\bf x})}^*TX\otimes \Lambda^{01}).
\end{equation}
\begin{lem}\label{conds55lem}
$\hat v_*(E_{\bf p}({\bf x})) = E_{\bf p}(v_*({\bf x}))$.
\end{lem}

\section{Kuranishi structure: review}
\label{sec;Kurareview}

The main result, Theorem \ref{constthm}, we prove in this article assigns a Kuranishi structure 
to each obstruction bundle data in the sense of Definition 
\ref{defn51}.
We refer readers to \cite[Section 15]{foootech2} for the version of the terminology of 
orbifold we use.\footnote{See also \cite{ofdruan} for an exposition on orbifold.}
(We always assume orbifolds to be effective, in particular.)
In this article we consider the case of orbifolds with boundary and corner.
\par
To state Theorem \ref{constthm} later we review the definition of Kuranishi structure in this section.
Let $\mathcal M$ be a compact metrizable space.

\begin{defn}\label{defnKchart}
A {\it Kuranishi chart}
\index{Kuranishi chart} of $\mathcal M$ is $\mathcal U =(U,\mathcal E,\psi,s)$ with
the following properties.
\begin{enumerate}
\item
$U$ is an (effective) orbifold.
\item
$\mathcal E$ is an orbi-bundle on $U$.
\item
$s$ is a smooth section of $\mathcal E$.
\item
$\psi : s^{-1}(0) \to \mathcal M$ is a homeomorphism onto an open set.
\end{enumerate}
We call $U$ a {\it Kuranishi neighborhood},\index{Kuranishi neighborhood} $\mathcal E$ an
{\it obstruction bundle},\index{obstruction bundle}
$s$ a {\it Kuranishi map}\index{Kuranishi map} and $\psi$ a {\it parametrization}.\index{parametrization}

If $U'$ is an open subset of $U$, then by restricting $\mathcal E$, $\psi$ and $s$
to $U'$, we obtain a
Kuranishi chart, which we write $\mathcal U\vert_{U'}$
and call an {\it open subchart}.\index{open subchart}
\par
The {\it dimension} %\index{dimension} 
$\mathcal U =(U,\mathcal E,\psi,s)$
is by definition 
$
\dim \mathcal U = \dim U  - {\rm rank}~ \mathcal E.
$
Here ${\rm rank}~ \mathcal E$ is the dimension  of the fiber $\mathcal E \to U$.
\end{defn}
\begin{defn}\label{defKchart}
Let $\mathcal U = (U,\mathcal E,\psi,s)$, $\mathcal U'
= (U',\mathcal E',\psi',s')$ be Kuranishi charts of $\mathcal M$.
An {\it embedding} of Kuranishi charts $: \mathcal U\to \mathcal U'$
\index{embedding of Kuranishi charts}
is  a pair $(\varphi,\widehat\varphi)$ with the following properties.
\begin{enumerate}
\item
$\varphi : U \to U'$ is an embedding of orbifolds.
\item
$\widehat\varphi : \mathcal E \to \mathcal E'$ is an embedding of orbi-bundles
over $\varphi$.
\item
$\widehat\varphi \circ s = s' \circ \varphi$.
\item
$\psi' \circ \varphi = \psi$ holds on $s^{-1}(0)$.
\item
For each $x \in U$ with $s(x) = 0$, the derivative
$
D_{\varphi(x)}s'
$
induces an isomorphism
\begin{equation}\label{form3.1111}
\frac{T_{\varphi(x)}U'}{(D_x\varphi)(T_xU)}
\cong
\frac{\mathcal E'_{\varphi(x)}}{\widehat\varphi(\mathcal E_x)}.
\end{equation}
\end{enumerate}
If $\dim U = \dim U'$ in addition, we call $(\varphi,\widehat\varphi)$ an {\it open embedding}.
\index{open embedding of Kuranishi chart}
\end{defn}
\begin{defn}\label{coordinatechangedef}
Let $\mathcal U_1 = (U_1,\mathcal E_1,\psi_1,s_1)$, $\mathcal U_2
= (U_2,\mathcal E_2,\psi_2,s_2)$ be Kuranishi charts of $\mathcal M$.
A {\it coordinate change in weak sense}\index{coordinate change} from $\mathcal U_1$
to
$\mathcal U_2$ is $(U_{21},\varphi_{21},\widehat\varphi_{21})$
with the following properties (1) and (2):
\begin{enumerate}
\item
$U_{21}$ is an open subset of $U_1$.
\item
$(\varphi_{21},\widehat\varphi_{21})$ is an
embedding of Kuranishi charts $: \mathcal U_1\vert_{U_{21}}
\to \mathcal U_2$.
\end{enumerate}
\end{defn}
\begin{defn}\label{kstructuredefn}
A {\it Kuranishi structure}\index{Kuranishi structure} $\widehat{\mathcal U}$ of $\mathcal M$
assigns a Kuranishi chart $\mathcal U_p
= (U_p,\mathcal E_p,\psi_p,s_p)$
with $p \in {\rm Im}(\psi_p)$  to each $p \in \mathcal M$ and a coordinate change in weak sense
$(U_{pq},\varphi_{pq},\widehat\varphi_{pq}) : \mathcal U_q \to \mathcal U_p$
to each $p$ and $q \in {\rm Im}(\psi_p)$ such that
$q \in \psi_{q}(U_{pq}\cap s_{q}^{-1}(0))$ and
the following holds
for each $r \in \psi_q(s_q^{-1}(0) \cap U_{pq})$.
\par
We put
$U_{pqr} = \varphi_{qr}^{-1}(U_{pq}) \cap U_{pr}$. Then we have
\begin{equation}\label{form3333}
\varphi_{pr}\vert_{U_{pqr}} = \varphi_{pq}\circ
\varphi_{qr}\vert_{U_{pqr}} 
,\qquad
\hat\varphi_{pr}\vert_{\pi^{-1}(U_{pqr})} = \hat\varphi_{pq}\circ
\hat\varphi_{qr}\vert_{\pi^{-1}(U_{pqr})}. 
\end{equation}
\par
We also require that the dimension of $\mathcal U_p$ is independent of 
$p$ and call it the {\it dimension} of $\widehat{\mathcal U}$.
When $U_{p}$ has corner, we call $\widehat{\mathcal U}$ a Kuranishi structure with corner.
\end{defn}
So far in this section, we consider orbifolds, orbibundles, embeddings between them, sections of $C^{\infty}$
class. The notion of Kuranishi structure we defined then is one of $C^{\infty}$ class.
By considering those objects of $C^{n}$ class ($1 \le n < \infty$) instead, we 
define the notion of Kuranishi structure 
of  $C^{n}$ class.
\begin{rem}
The definition of Kuranishi structure here 
is equivalent to the definition of Kuranishi structure with 
tangent bundle in \cite[Section A1]{fooobook2},\footnote{There is no 
mathematical change of the definition of Kuranishi structure since then.} where 
certain errors in \cite{FO} were corrected.\footnote{None of those errors 
affect any of the applications of Kuranishi structure and virtual fundamental chain.}
\end{rem}
\begin{defn}
Let $\widehat{\mathcal U}$ be a Kuranishi structure of $\mathcal M$.
We replace $\mathcal U_p$ by its open subchart containing 
$\psi^{-1}_p(p)$ and restrict coordinate changes in the obvious way.
We then obtain a Kuranishi structure of $\mathcal M$.
We call such a Kuranishi structure an {\it open substructure.}
\par
We say two Kuranishi structures $\widehat{\mathcal U}$, $\widehat{\mathcal U}'$
determine the same {\it germ of Kuranishi structures},\index{germ of Kuranishi structures}
if they have open substructures which are isomorphic.
\footnote{Here two Kuranishi structures $\widehat{\mathcal U} = (\{\mathcal U_p\},\{(U_{pq},
{\varphi}_{pq},\widehat{\varphi}_{pq},)\}$, $\widehat{\mathcal U}' = (\{\mathcal U'_p\},\{(U'_{pq},
{\varphi}'_{pq},\widehat{\varphi}'_{pq},)\}$ are {\it isomorphic} if
there exist diffeomorphisms of orbifolds $\phi_p : U_p \to U'_p$ between the 
Kuranishi neighborhoods, covered by isomorphisms of obstruction bundles $\widehat\phi_p : \mathcal E_p \to \mathcal E'_p$, such that, Kuranishi maps, parametrizations and coordinate 
changes commute with them. (We also require $\phi_q(U_{pq}) = U'_{pq}$.)}
\end{defn}
\begin{defn}\label{mapkura}
Let $\widehat{\mathcal U}$ be a Kuranishi structure  of $\mathcal M$.
\begin{enumerate}
\item
A {\it strongly continuous map $\widehat f$}\index{strongly continuous map}
from $(\mathcal M;\widehat{\mathcal U})$ to
a topological space $Y$ assigns a continuous map $f_{p}$ from
$U_{p}$ to $Y$ to each $p\in X$ such that
$f_p \circ \varphi_{pq} = f_q$ holds on $U_{pq}$.
\item
In the situation of (1), the map $f:\mathcal M \to Y$ defined by
$f(p) = f_p(p)$ is a continuous map from $\mathcal M$ to $Y$.
We call $f : \mathcal M \to Y$ the {\it underlying continuous map}\index{underlying continuous map} of $\widehat f$.
\item
When $Y$ is a smooth manifold, we say $\widehat f$ is {\it strongly smooth}\index{strongly smooth}
if each $f_p$ is smooth.
\item
A strongly smooth map is said to be {\it weakly submersive}\index{weakly submersive}
if each 
$f_p$ is a submersion.
\end{enumerate}
\end{defn}

\section{Construction of Kuranishi structure}
\label{sec;kuraconst}

\subsection{Statement}
\label{subsec:kuraconststate}

We say two obstruction bundle data $(\{\mathscr U_{\bf p}\},\{E_{\bf p}({\bf x})\})$ 
and $(\{\mathscr U'_{\bf p}\},\{E'_{\bf p}({\bf x})\})$ determine the same 
{\it germ} if \index{germ of obstruction bundle data}
$
E_{\bf p}({\bf x}) = E'_{\bf p}({\bf x})
$
for every ${\bf x} \in \mathscr U_{\bf p} \cap \mathscr U'_{\bf p}$.

\begin{thm}\label{constthm}
\begin{enumerate}
\item
To arbitrary obstruction bundle data of the 
moduli space ${\mathcal M}_{k+1,\ell}(X,L,J;\beta)$
we can associate a germ of a Kuranishi structure on ${\mathcal M}_{k+1,\ell}(X,L,J;\beta)$
in a canonical way.
\item
If two obstruction bundle data determine the same germ 
then the induced Kuranishi structures determine the same germ.
\item
The evaluation maps 
${\rm ev}_{j}$ ($j=0,1,\dots,k$), ${\rm ev}^{\rm int}_{i}$ ($i=1,\dots,\ell$)
are the underlying continuous maps of strongly smooth maps.
\end{enumerate}
\end{thm}

In this section we prove Theorem \ref{constthm}
except the part where smoothness of obstruction bundle data (Definition \ref{defn51} (2)) concerns, 
which will be discussed in Sections \ref{sec;obstbundleexpdecay}, \ref{sec;Kuracharsmooth}, \ref{sec:changesmoo}.

\subsection{Construction of Kuranishi charts}
\label{subsec:kurastconst}

Let $(\{\mathscr U_{\bf p}\},\{E_{\bf p}({\bf x})\})$ be 
obstruction bundle data at ${\bf p}$.
We will define a Kuranishi chart of ${\mathcal M}_{k+1,\ell}(X,L,J;\beta)$ at ${\bf p}$ 
using this data.

\begin{defn}\label{defn72}
We define
$
U_{\bf p}
$
to be the set of all ${\bf x} \in \mathscr U_{\bf p}$ such that
\begin{equation}\label{form71}
\overline{\partial} u_{\bf x} \in E_{\bf p}(\bf x).
\end{equation}
(\ref{form71})
is independent of the choice of representative ${\bf x}$ because of Remark \ref{rem5757} (*).
We also put
$$
{\mathcal E}_{\bf p} = 
\bigcup_{{\bf x} \in U_{\bf p}} E_{{\bf p}}({\bf x})/{\rm Aut}({\bf x}) \times \{{\bf x}\}.
$$
Here the group ${\rm Aut}^+({\bf x})$ acts on $E_{{\bf p}}({\bf x})$ by 
Definition \ref{defn51} (5).
We have a natural projection $\pi : \mathcal E_{\bf p} \to U_{\bf p}$.
\footnote{To get an obstruction bundle we divide $E_{\bf p}({\bf x})$ by 
${\rm Aut}({\bf x})$ not by ${\rm Aut}^+({\bf x})$.}
We define a map $\frak s_{\bf p} : U_{\bf p} \to \mathcal E_{\bf p}$ by 
$$
\frak s_{\bf p}({\bf x}) = [\overline{\partial} u_{\bf x},{\bf x}] \in {\mathcal E}_{\bf p}.
$$
(The right hand side is independent of the choice of representative 
of ${\bf x}$.)
\end{defn}
\begin{lem}\label{lem73}
After replacing $\mathscr U_{\bf p}$ by a smaller neighborhood if necessary, 
$
U_{\bf p}
$ has a structure of (effective) smooth orbifold.
$\mathcal E_{\bf p}$ becomes the underlying topological 
space of a smooth orbi-bundle on $U_{\bf p}$ and $\pi : \mathcal E_{\bf p} \to U_{\bf p}$
is its projection.
$\frak s_{\bf p}$ becomes a smooth section of $\mathcal E_{\bf p}$.
\end{lem}

We use smoothness of $E_{\bf p}(\bf x)$ (Definition \ref{defn51} (2)) 
and \cite[Theorem 6.4]{foooexp}
to prove Lemma \ref{lem73}. See Section \ref{sec;Kuracharsmooth}.
\par
We define $\psi_{\bf p} : \frak s_{\bf p}^{-1}(0) \to {\mathcal M}_{k+1,\ell}(X,L,J;\beta)$
as follows.
If ${\bf x} \in \frak s_{\bf p}^{-1}(0)$ then $\overline{\partial} u_{\bf x}=0$
by definition.
Therefore ${\bf x}$ represents an element of ${\mathcal M}_{k+1,\ell}(X,L,J;\beta)$.
We define $\psi_{\bf p}({\bf x})$ to be the element of  ${\mathcal M}_{k+1,\ell}(X,L,J;\beta)$
represented by ${\bf x}$.

\begin{lem}\label{lemm74}
$(U_{\bf p},\mathcal E_{\bf p},\frak s_{\bf p},\psi_{\bf p})$
is a Kuranishi chart of ${\mathcal M}_{k+1,\ell}(X,L,J;\beta)$ at ${\bf p}$.
\end{lem}
This is immediate from Lemma \ref{lem73} and the definition.

\subsection{Construction of coordinate change}
\label{subsec:cchangeconst}

\begin{shitu}
Let $(\{\mathscr U_{\bf p}\},\{E_{\bf p}({\bf x})\})$ be 
obstruction bundle data.
Suppose ${\bf q} \in \mathscr U_{\bf p} \cap \mathcal M_{k+1,\ell}(X,L,J;\beta)$. 
Let $(U_{\bf p},\mathcal E_{\bf p},\frak s_{\bf p},\psi_{\bf p})$ 
(resp. $(U_{\bf q},\mathcal E_{\bf q},\frak s_{\bf q},\psi_{\bf q})$) 
be the Kuranishi chart at ${\bf p}$ (resp. ${\bf q}$) obtained by 
Lemma \ref{lemm74}.
$\diamond$
\end{shitu}
We put
$
U_{{\bf p}{\bf q}} = U_{\bf q} \cap \mathscr U_{\bf p}.
$
Let ${\bf x} \in U_{{\bf p}{\bf q}}$. Then by Definition \ref{defn51} (4) 
and Definition \ref{defn72} we have
$$
\overline \partial u_{\bf x} \in
E_{\bf q}({\bf x}) \subseteq E_{\bf p}({\bf x}).
$$
Thus $U_{{\bf p}{\bf q}} \subseteq U_{{\bf p}}$.
(Note both are subsets of $\mathcal X_{k+1,\ell}(X,L,J;\beta)$.)
Let $\varphi_{{\bf p}{\bf q}} : U_{{\bf p}{\bf q}} \to U_{{\bf p}}$
be the inclusion map.
\par
To define the bundle map part of the coordinate change
we introduce:
\begin{defn}
We consider a pair $(((\Sigma,\vec z,\vec{\frak z}),u),V)$
where $((\Sigma,\vec z,\vec{\frak z}),u)$ is a representative 
of an element of $\mathcal X_{k+1,\ell}(X,L,J;\beta)$ and 
$V \in L^2_{0}(\Sigma;u^*TX\otimes \Lambda^{01})$.
\par
We say $(((\Sigma,\vec z,\vec{\frak z}),u),V)$ is equivalent to 
$(((\Sigma',\vec z^{\,\prime},\vec{\frak z}^{\,\prime}),u'),V')$
if there exists a map $v : \Sigma \to \Sigma'$ which becomes an 
isomorphism $((\Sigma,\vec z,\vec{\frak z}),u) \to ((\Sigma',\vec z^{\,\prime},\vec{\frak z}^{\,\prime}),u')$ 
in the sense of Definition \ref{defn3144}
and 
$$
v_*(V) = V'.
$$
Note $v$ induces a map 
$
v_* : L^2_{0}(\Sigma;u^*TX\otimes \Lambda^{01})
\to L^2_{0}(\Sigma';(u')^*TX\otimes \Lambda^{01}).
$
\par
We denote by $\mathcal{EX}_{k+1,\ell}(X,L,J;\beta)$
\index{00E3X_{k+1,\ell}(X,L,J;\beta)@$\mathcal{EX}_{k+1,\ell}(X,L,J;\beta)$}
the set of all such equivalence classes of $(((\Sigma,\vec z,\vec{\frak z}),u),V)$.
\par
There exists an obvious projection
$
\pi : \mathcal{EX}_{k+1,\ell}(X,L,J;\beta) \to \mathcal{X}_{k+1,\ell}(X,L,J;\beta).
$
If ${\bf x}$ is represented by $((\Sigma_{\bf x},\vec z_{\bf x},\vec{\frak z}_{\bf x}),u_{\bf x})$
then the fiber $\pi^{-1}({\bf x})$ is canonically identified with 
$
L^2_{0}(\Sigma_{\bf x};u_{\bf x}^*TX\otimes \Lambda^{01})/{\rm Aut}(\bf x).
$
Here the action of ${\rm Aut}(\bf x)$ is defined in the same way as (\ref{form5666}).
\end{defn}
Let $(U_{\bf p},\mathcal E_{\bf p},\frak s_{\bf p},\psi_{\bf p})$ be a 
Kuranishi chart as in Lemma \ref{lemm74}.
By definition the total space of $\mathcal E_{\bf p}$, which we denote also by 
$\mathcal E_{\bf p}$ by an abuse of notation, is canonically embedded into 
$\mathcal{EX}_{k+1,\ell}(X,L,J;\beta)$ such that the next diagram commutes.

\begin{equation}\label{diag00rt1333}
\begin{CD}
\mathcal E_{\bf p}
@ >>>
\mathcal{EX}_{k+1,\ell}(X,L,J;\beta)
\\
@ VV{\pi_{\bf p}}V @ VV{\pi}V \\
U_{\bf p} @>>> 
\mathcal{X}_{k+1,\ell}(X,L,J;\beta)
\end{CD}
\end{equation}
Let ${\bf q} \in \psi_{\bf p}(s_{\bf p}^{-1}(0))$.
Then by definition $\mathcal E_{\bf q}\vert_{U_{{\bf p}{\bf q}}}$ ($= \pi_{\bf q}^{-1}(U_{{\bf p}{\bf q}}) \subset 
\mathcal E_{\bf q}$)
is a subset of $\mathcal E_{\bf p}$, 
when we regard them as subsets of $\mathcal{EX}_{k+1,\ell}(X,L,J;\beta)$.
\par
We define $\widehat{\varphi}_{{\bf p}{\bf q}}$ to be the inclusion 
map $\mathcal E_{\bf q}\vert_{U_{{\bf p}{\bf q}}} \to \mathcal E_{\bf p}$.

\begin{lem}\label{lem7878}
The pair $({\varphi}_{{\bf p}{\bf q}},\widehat{\varphi}_{{\bf p}{\bf q}})$
is a coordinate change from 
$(U_{\bf q},\mathcal E_{\bf q},\frak s_{\bf q},\psi_{\bf q})$
to
$(U_{\bf p},\mathcal E_{\bf p},\frak s_{\bf p},\psi_{\bf p})$.
\end{lem}
This is nothing but \cite[Theorem 8.32]{foooexp}, once the notion of 
smoothness of $E_{\bf p}({\bf x})$ will be clarified. See Subsection \ref{subsec:smoproof}.

\subsection{Wrapping up the construction of Kuranishi structure}
\label{subsec:cchkuraeconst}

\begin{lem}\label{lem79}
Let ${\bf p}$, ${\bf q} \in {\rm Im}(\psi_{\bf p})$, ${\bf r} \in \psi_{\bf q}(s_{\bf q}^{-1}(0) 
\cap U_{{\bf p}{\bf q}})$.
We put
$U_{{\bf p}{\bf q}{\bf r}} = \varphi_{{\bf q}{\bf r}}^{-1}(U_{{\bf p}{\bf q}}) \cap U_{{\bf p}{\bf r}}$. 
Then we have
\begin{equation}\label{form3d333}
\aligned
\varphi_{{\bf p}{\bf r}}\vert_{U_{{\bf p}{\bf q}{\bf r}}} = \varphi_{{\bf p}{\bf q}}\circ
\varphi_{{\bf q}{\bf r}}\vert_{U_{{\bf p}{\bf q}{\bf r}}},
\qquad
\hat\varphi_{{\bf p}{\bf r}}\vert_{\pi^{-1}(U_{{\bf p}{\bf q}{\bf r}})} = \hat\varphi_{{\bf p}{\bf q}}\circ
\hat\varphi_{{\bf q}{\bf r}}\vert_{\pi^{-1}(U_{{\bf p}{\bf q}{\bf r}})}. 
\endaligned
\end{equation}
\end{lem}
\begin{proof}
If we regard the domain and the target of both sides of (\ref{form3d333})
as subsets of $\mathcal{X}_{k+1,\ell}(X,L,J;\beta)$
or of $\mathcal{EX}_{k+1,\ell}(X,L,J;\beta)$
then the both sides are the identity map. 
Therefore the equalities are obvious.
\end{proof}
\begin{rem}\label{lem710}
\begin{enumerate}
\item 
The orbifold we use are always effective and maps between them are 
embeddings. Therefore to check the equality of the two maps it suffices to show 
that they coincide set-theoretically.
This fact simplifies the proof.
\item
The proof of Lemma \ref{lem79} given above is  simpler than the proof in 
\cite[Section 24]{foootech} etc.
%\footnote{On the other hand, the proof of \cite[Section 24]{foootech}
%is correct and the geometric intuition behind this proof is simple.}.
This is because we use the ambient {\it set} $\mathcal{X}_{k+1,\ell}(X,L,J;\beta)$.
\par
Note however we do {\it not} use {\it any} structure of 
$\mathcal{X}_{k+1,\ell}(X,L,J;\beta)$.
The ambient set is used only to show the set-theoretical equality (\ref{form3d333}).
It seems to the authors that putting various structures such as topology 
on $\mathcal{X}_{k+1,\ell}(X,L,J;\beta)$ is rather cumbersome
since this infinite dimensional `space' can be pathological.
Using it only as a set and  proving set-theoretical equality seems easier to carry out. 
Since it makes the proof of Lemma \ref{lem79} simpler, it is 
worth using this ambient {\it set}.
\end{enumerate}
\end{rem}
The proof of Theorem \ref{constthm} (1) is complete.
The proof of Theorem \ref{constthm} (2) is immediate from 
construction and is omitted.

\subsection{Evaluation maps}
\label{subsec:evsmooth}

We study the evaluation maps in this subsection.
\begin{lem}
The evaluation maps ${\rm ev}_j : 
\mathcal{M}_{k+1,\ell}(X,L,J;\beta) \to L$ and 
${\rm ev}^{\rm int}_i: 
\mathcal{M}_{k+1,\ell}(X,L,J;\beta) \to X$
are strongly continuous.
\end{lem}
\begin{proof}
An element of $U_{\bf p}$  as defined in Definition \ref{defn72}
consists of ${\bf x} = ((\Sigma_{\bf x},\vec z_{\bf x},\vec{\frak z}_{\bf x}),u_{\bf x})$.
We define
$
{\rm ev}_{{\bf p},j}({\bf x}) = u_{{\bf x}}(z_{{\bf x},j}), 
$
$
{\rm ev}^{\rm int}_{{\bf p},i}({\bf x}) = u_{{\bf x}}(\frak z_{{\bf x},i}).
$
It is obvious that they are compatible with the coordinate change.
\end{proof}
It follows from the  construction of smooth structure of $U_{\bf p}$
(in Sections \ref{sec;Kuracharsmooth} and \ref{sec:Cmugen})  that 
${\rm ev}_{{\bf x},j}({\bf x})$ and 
${\rm ev}^{\rm int}_{{\bf x},i}({\bf x})$ are smooth.
So ${\rm ev}_j$ and ${\rm ev}^{\rm int}_i$ are strongly smooth.

\begin{conds}\label{cond22}
We say that $E_{\bf p}(\bf p)$ satisfies the {\it mapping transversality 
condition}\index{mapping transversality 
condition} for ${\rm ev}_0$ if the map 
$$
{\rm Ev}_0 : (D_{u_{{\bf p}}}\overline\partial)^{-1}(E_{\bf p}({\bf p})) \to T_{{\rm ev}_0({\bf p})}L
$$
is surjective. Here ${\rm Ev}_0$ is defined as follows.
Let $\sum V_a$ be an element of $(D_{u_{{\bf p}}}\overline\partial)^{-1}(E_{\bf p}({\bf p}))$.
Suppose $z_0$ is in the component $\Sigma_{a_0}$.
Then
$
{\rm Ev}_0(\sum V_a) = V_{a_0}(z_0).
$
\end{conds}
\begin{lem}\label{lem71213}
If Condition \ref{cond22} is satisfied then 
${\rm ev}_0 : 
\mathcal{M}_{k+1,\ell}(X,L,J;\beta) \to L$
is weakly submersive.
\end{lem}
\begin{proof}
It is easy to see that ${\rm Ev}_0$ induces the differential of the 
map ${\rm ev}_{{\bf p},0}$ at ${\bf p}$.
The lemma is an immediate consequence of this fact.
\end{proof}
We can define the mapping transversality condition for other marked points and 
generalize Lemma \ref{lem71213} in the obvious way.

\section{Smoothness of obstruction bundle data}
\label{sec;obstbundleexpdecay}

In this section we define Condition (2) in Definition \ref{defn51}.

\subsection{Trivialization of families of function spaces}
\label{subsec:trivobst}

\begin{rem}\label{shitu84}
We choose a unitary connection on $TX$ and fix it.
\end{rem}

\begin{shitu}\label{situ85}
Let ${\bf p} \in \mathcal M_{k+1,\ell}(X,L;\beta)$. 
We take stabilization  and trivialization data $\frak W_{\bf p}$, 
part of which are the weak stabilization data $\vec {\frak w}_{\bf p}$ at ${\bf p}$ 
consisting of $\ell'$ extra marked points.
We assume $E_{\bf p}({\bf x})$ satisfies Definition \ref{defn51} (1)(3)(5).
$\diamond$
\end{shitu}

Note ${\bf p}\cup \vec {\frak w}_{\bf p} \in \mathcal M_{k+1,\ell+\ell'}(X,L;\beta)$.
Let ${\bf y} = ((\Sigma_{\bf y},\vec z_{\bf y},\vec {\frak z}_{\bf y}),u_{\bf y})$ be an 
element of $\mathcal X_{k+1,\ell+\ell'}(X,L;\beta)$ which is $\epsilon_0$-close to 
${\bf p}\cup \vec {\frak w}_{\bf p}$.
We apply Lemma \ref{lem34} to ${\bf p}\cup \vec {\frak w}_{\bf p}$ 
and obtain  $\frak y$, an element of the domain of $\Phi$ in (\ref{form33}),
such that $\Phi(\frak y) = (\Sigma_{\bf y},\vec z_{\bf y},\vec{\frak z}_{\bf y})$.
\par
By Lemma \ref{lem3838} we obtain a smooth embedding
$
\hat{\Phi}_{\frak y,\vec \epsilon} : \Sigma_{{\bf p}\cup \vec {\frak w}_{\bf p}}(\vec \epsilon) \to \Sigma_{\bf y}
$
which sends $\vec z_{\bf p}$, $\vec{\frak z}_{\bf p}\cup \vec{\frak w}_{\bf p}$
to $\vec z_{\bf y}$, $\vec{\frak z}_{\bf y}$, respectively.
We remark
$
\Sigma_{\bf y}(\vec \epsilon) = \hat{\Phi}_{\frak y,\vec \epsilon}(\Sigma_{\bf p}(\vec \epsilon)).
$
We put ${\bf x} = \frak{forget}_{\ell+\ell',\ell}(\bf y)$
and obtain 
$
E_{\bf p}({\bf x}) \subset 
C^2(\Sigma_{\bf x};u_{{\bf x}}^*TX\otimes \Lambda^{01})
$. 
Note $\Sigma_{\bf x} = \Sigma_{\bf y}$ and $u_{{\bf x}} = u_{{\bf y}}$.
We also remark 
$\Sigma_{{\bf p}\cup \vec {\frak w}_{\bf p}} = \Sigma_{\bf p}$, 
$u_{{\bf p}\cup \vec {\frak w}_{\bf p}} = u_{\bf p}$.
\par
We define a linear map 
\index{00P3_{{\bf y}}@$\mathcal P_{{\bf y}}$}
\begin{equation}\label{form8000}
\mathcal P_{{\bf y}} : C^2(\Sigma_{\bf y}(\vec \epsilon);u_{{\bf y}}^*TX\otimes \Lambda^{01})
\to 
C^2(\Sigma_{\bf p};u_{{\bf p}}^*TX\otimes \Lambda^{01})
\end{equation}
below.
We first define a bundle map
\begin{equation}\label{form8181}
u_{{\bf y}}^*TX \to u_{{\bf p}}^*TX
\end{equation}
over the diffeomorphism $\hat{\Phi}_{\frak y,\vec \epsilon}^{-1}$.
Let $z \in \Sigma_{{\bf p}\cup \vec {\frak w}_{\bf p}}(\vec \epsilon)$.
By our choice, the distance between $u_{{\bf y}}(\hat{\Phi}_{\frak y,\vec \epsilon}(z))$ and $u_{{\bf p}}(z)$ 
is smaller than $\epsilon_0$. We may choose $\epsilon_0$ smaller 
than the injectivity radius of the Riemannian metric in Remark \ref{rem499}.
Therefore there exists a unique minimal geodesic joining $u_{{\bf y}}(\hat{\Phi}_{\frak y,\vec \epsilon}(z))$ and $u_{{\bf p}}(z)$.
We take a parallel transport by the connection in 
Remark \ref{shitu84} along this geodesic.
We thus obtain (\ref{form8181}). This bundle map is complex linear, 
since the connection in Remark \ref{shitu84} is unitary.
\par
We next take the differential of $\hat{\Phi}_{\frak y,\vec \epsilon}$ to obtain 
a bundle map 
$
\Lambda^1(\Sigma_{\bf y}(\vec \epsilon))
\to \Lambda^1(\Sigma_{{\bf p}\cup \vec {\frak w}_{\bf p}}(\vec \epsilon)).
$
We take its complex linear part to obtain
\begin{equation}\label{form8281}
\Lambda^{01}(\Sigma_{\bf y}(\vec \epsilon))
\to \Lambda^{01}(\Sigma_{{\bf p}\cup \vec {\frak w}_{\bf p}}(\vec \epsilon)).
\end{equation}
This is a complex linear bundle map over $\hat{\Phi}_{\frak y,\vec \epsilon}^{-1}$.
In case ${\bf y} = {\bf p}\cup \vec {\frak w}_{\bf p}$ this is the identity map.
So if we take $\epsilon_0$ sufficiently small, (\ref{form8281}) is an isomorphism.
\par
Taking a tensor product of (\ref{form8181}) and (\ref{form8281}) over $\C$ we obtain 
a bundle isomorphism
\begin{equation}\label{form8333}
u_{{\bf y}}^*TX\otimes \Lambda^{01} \to u_{{\bf p}}^*TX\otimes \Lambda^{01}
\end{equation}
over $\hat{\Phi}_{\frak y,\vec \epsilon}^{-1}$.
Roughly speaking the smoothness of $E_{\bf p}({\bf x})$ means that 
$\mathcal P_{{\bf y}}(E_{\bf p}({\bf x}))$ depends smoothly on ${\bf y}$.
We will formulate it precisely in the next subsection.

\subsection{The smoothness condition of  obstruction bundle data}
\label{subsec:smoothobst}

\begin{defn}\label{defn860}
Suppose we are in Situation \ref{situ85}.
We say $E_{\bf p}({\bf x})$ is {\it independent of $u_{\bf x}\vert_{\rm neck}$}
\index{independent of $u_{\bf x}\vert_{\rm neck}$}
if the following holds for some $\epsilon_0$, $\vec\epsilon$.
\par
Let ${\bf y},{\bf y}'$ be elements of
$\mathcal X_{k+1,\ell+\ell'}(X,L;\beta)$ which are $\epsilon_0$-close to 
${\bf p}\cup \vec {\frak w}_{\bf p}$.
We put ${\bf x} = \frak{forget}_{\ell+\ell',\ell}(\bf y)$ and
${\bf x}' = \frak{forget}_{\ell+\ell',\ell}({\bf y}')$.
(Note $\Sigma_{{\bf x}} = \Sigma_{{\bf y}}$, $\Sigma_{{\bf x}'} = \Sigma_{{\bf y}'}$.)
We assume that there exists $v : \Sigma_{{\bf x}} \to \Sigma_{{\bf x}'}$
such that
\begin{enumerate}
\item $v$ is biholomorphic and sends $\vec z_{{\bf y}}$, $\vec {\frak z}_{{\bf y}}$ to $\vec z_{{\bf y}'}$, $\vec {\frak z}_{{\bf y}'}$, respectively.
\item $v(\Sigma_{\bf y}(\vec \epsilon)) = \Sigma_{{\bf y}'}(\vec \epsilon)$
and the equality
$
u_{{\bf y}'} \circ v
 = u_{\bf y}$
holds on $\Sigma_{\bf y}(\vec \epsilon)$.
\end{enumerate}
Then we require that all the elements of 
$E_{\bf p}({\bf x})$ (resp. $E_{\bf p}({\bf x}')$)
are supported on $\Sigma_{\bf y}(\vec \epsilon)$ 
(resp. $\Sigma_{{\bf y}'}(\vec \epsilon)$)  and 
$
\mathcal P_{{\bf y}}(E_{\bf p}({\bf x})) = \mathcal P_{{\bf y}'}(E_{\bf p}({\bf x}')).
$
\end{defn}
This is a part of the definition of the smoothness of obstruction bundle data, that is, 
Definition \ref{defn51} (2).
To formulate the main part of this condition we use the next:
\begin{defn}\label{defn87}
Let $\mathbb H$ be a Hilbert space and $\{E(\xi)\}$ a family of 
finite dimensional linear subspaces of $\mathbb H$ parametrized by 
$\xi \in Y$, where $Y$ is a Hilbert manifold.
We say $\{E(\xi)\}$ is a {\it $C^n$ family}
%\index{00C1nfamily@$C^n$ family!of linear subspaces}
if there exists a finite number of  $C^n$ maps:
$e_i : Y \to \mathbb H$ ($i=1,\dots,N$) such that for each $\xi \in Y$, 
$(e_1(\xi),\dots,e_N(\xi))$ is a basis of $E(\xi)$.
\end{defn}
Suppose we are in Situation \ref{situ85}.
In particular we have chosen ${\frak W}_{\bf p}$.
We assume $E_{\bf p}({\bf x})$ is independent of $u_{\bf x}\vert_{\rm neck}$.
We take $\vec\epsilon$ which is sufficiently smaller 
than the one appearing in Definition \ref{defn860}.
We put ${\bf p}^+ = {\bf p} \cup \vec{\frak w}_{\bf p}$
where $\vec{\frak w}_{\bf p}$ is a part of ${\frak W}_{\bf p}$.\index{00P++@${\bf p}^+$}
We consider the map (\ref{form33})
\begin{equation}\label{form8585}
\Phi : 
\prod_{a\in \mathcal A_{{\bf p}}^{\rm s} \cup \mathcal A_{{\bf p}}^{\rm d}} \mathcal V^+_a 
\times [0,c)^{m_{\rm d}} \times (D^2_{\circ}(c))^{m_{\rm s}}
\to \mathcal M^{\rm d}_{k+1,\ell},
\end{equation}
for ${\bf p}^+$. Here we decompose $(\Sigma_{\bf p^+},\vec z_{\bf p},\vec{\frak z}_{\bf p}
\cup \vec{\frak w}_{\bf p})$ into irreducible components and let $\mathcal V^+_a$ be the 
deformation parameter space of each irreducible component 
$
(\Sigma_{\bf p}(a),\vec z_{\bf p}(a),\vec{\frak z}_{\bf p}(a)
\cup \vec{\frak w}_{\bf p}(a)).
$
Here 
$$
\aligned\vec{z}_{\bf p}(a) &= (\vec{z}_{\bf p} \cap \Sigma_{\bf p}(a)) \cup 
\{\text{boundary nodes in  $\Sigma_{\bf p}(a)$}\} \\
\vec{\frak z}_{\bf p}(a) &= (\vec{\frak z}_{\bf p} \cap \Sigma_{\bf p}(a)) \cup 
\{\text{interior nodes in  $\Sigma_{\bf p}(a)$}\} \\ 
\vec{\frak w}_{\bf p}(a) &= \vec{\frak w}_{\bf p} \cap \Sigma_{\bf p}(a).
\endaligned$$
\par
Now we take the direct product 
\index{00V3({{\bf p}^+};\vec\epsilon)@$\mathcal V({{\bf p}^+};\vec\epsilon)$}
\begin{equation}\label{form86}
\mathcal V({{\bf p}^+};\vec\epsilon) = 
\prod_{a\in \mathcal A_{{\bf p}}^{\rm s} \cup \mathcal A_{{\bf p}}^{\rm d}}\mathcal V_a^+
\times
\prod_{j=1}^{m_{\rm d}} [0,{\epsilon_{j}^{\rm d}})
\times
\prod_{i=1}^{m_{\rm s}} D^2_{\circ}({\epsilon_{i}^{\rm s}}).
\end{equation}
Note we have already taken $\varphi_{a,i}^{\rm s}$, $\varphi_{a,j}^{\rm d}$, $\phi_{a}$
as a part of $\frak W_{\bf p}$.
\par
To each $\frak y \in \mathcal V({\bf p}^+;\vec\epsilon)$ we  
associate a marked nodal disk $(\Sigma_{\frak y},\vec z_{\frak y},\vec{\frak z}_{\frak y})$ 
by Lemma \ref{lem34}.
We also obtain a diffeomorphism
$
\widehat{\Phi}_{\frak y,\vec\epsilon}: \Sigma_{{\bf p}^+}(\vec\epsilon) \to \Sigma_{\frak y}(\vec\epsilon)
$
by Lemma \ref{lem3838}.
\par
Let $\mathscr L_m$ be a small neighborhood of $u_{{\bf p}^+}\vert_{\Sigma_{{\bf p}^+}(\vec\epsilon)}$
in $L^2_{m}((\Sigma_{{\bf p}^+}(\vec\epsilon),\partial\Sigma_{{\bf p}^+}(\vec\epsilon));X,L)$.
We will associate 
a finite dimensional subspace 
$
E_{\bf p;{\frak W}_{\bf p}}({\frak y},u')$
of 
$L^2_m(\Sigma_{{\bf p}^+}(\vec \epsilon);u_{{\bf p}}^*TX\otimes \Lambda^{01})
$
to $({\frak y},u') \in \mathcal V({\bf p}^+;\vec\epsilon) \times \mathscr L_m$ below.
\par
We assume $E_{\bf p}(\bf x)$ is independent of $u_{\bf x}\vert_{\rm neck}$
and consider 
$$
u'' = u' \circ \widehat{\Phi}_{\frak y,\vec\epsilon}^{-1} : \Sigma_{\frak y}(\vec\epsilon) \to X.
$$
We can extend $u''$ to $\Sigma_{\frak y}$ (by modifying it near the small 
neighborhood of the boundary of $\Sigma_{\frak y}(\vec\epsilon)$), still denoted by $u''$,
so that $u''\vert_{\Sigma_{\frak y} \setminus \Sigma_{\frak y}(\vec\epsilon)}$ has diameter $<\epsilon_0$.\footnote{
Note $u'\vert_{\partial \Sigma_{\frak y}(\vec\epsilon)}$ has diameter $< \epsilon'$ (in the sense of Definition \ref{defn41444})
with $\epsilon' < \epsilon_0$, and $\epsilon_0$ is smaller than the injectivity radius of $X$. 
We can use these facts to show the existence of $u''$.\label{fn16}}
\par
We now take 
${\bf y} = ((\Sigma_{\frak y},\vec z_{\frak y},\vec{\frak z}_{\frak y} \cup \vec{\frak w}_{\frak y}),u'')$
and  ${\bf x} = \frak{forget}_{\ell+\ell',\ell}(\bf y)$.
Then using $E_{\bf p}({\bf x})$ we define
\begin{equation}\label{form255}
E_{{\bf p};{\frak W}_{\bf p}}({\frak y},u') 
=
\mathcal P_{{\bf y}}(E_{\bf p}({\bf x})).
\end{equation}
Since $E_{\bf p}({\bf x})$ is independent of $u_{\bf x}\vert_{\rm neck}$ 
this is independent of the choice of the extension of $u''$.
As a part of our condition, we require 
$$
E_{{\bf p};{\frak W}_{\bf p}}({\frak y},u') 
\subset L^2_m(\Sigma_{\bf p};u_{{\bf p}}^*TX\otimes \Lambda^{01}).
$$
See Definition \ref{defn86} (1).
This is a finite dimensional subspace of 
$L^2_m(\Sigma_{\bf p};u_{{\bf p}}^*TX\otimes \Lambda^{01})$
depending on ${\bf p},{\frak W}_{\bf p}$ and ${\frak y},u'$.
\par
\begin{defn}\label{defn86}
We say $E_{\bf p}({\bf x})$ {\it depends smoothly on ${\bf x}$ with respect to 
$({\bf p},{\frak W}_{\bf p})$} if the following holds.
For each $n$ there exists $m_0$ such that if $m \ge m_0$ and 
$\vec {\epsilon}$ is small then:
\begin{enumerate}
\item Elements of $E_{\bf p}({\bf x})$ are of $L^2_m$ class if $u_{\bf x}$ is  of $L^2_m$ class.
\item If ${\bf x} = \frak{forget}_{\ell+\ell',\ell}({\bf y})$, $(\Sigma_{\bf y},\vec z_{\bf y},\vec{\frak z}_{\bf y}) = \Phi(\frak y)$
(where $\frak y$ and ${\bf y}$ are as above)
then the supports of elements of 
$E_{\bf p}({\bf x})$ are contained in $\Sigma_{\bf y}(\vec \epsilon)$.
\item
$E_{{\bf p};{\frak W}_{\bf p}}({\frak y},u')$ 
is a $C^n$ family 
parametrized by $({\frak y},u')$ in the sense of Definition \ref{defn87}.
\end{enumerate}
\end{defn}
\begin{rem}\label{rem83}
Let ${\bf r} = ((\Sigma_{{\bf r}},\vec z_{{\bf r}},\vec {\frak z}_{{\bf r}}),u_{{\bf r}}) 
\in \mathcal{X}_{k+1,\ell}(X,L,J;\beta)$ be an element of $\mathscr U_{\bf p}$
such that $u_{{\bf r}}$ is smooth but not necessarily pseudo holomorphic.
We can still define the notion of stabilization and trivialization data $\frak W_{{\bf r}}$ in the same 
way as Definition \ref{situation38}.\index{00W4_{\bf p}@$\frak W_{\bf p}$, $\frak W_{\bf r}$}
\end{rem}
\begin{defn}\label{defn8911}
We say $E_{\bf p}({\bf x})$ {\it depends smoothly on ${\bf x}$}\index{depends smoothly on ${\bf x}$}
if:
\begin{enumerate}\label{defn899}
\item  $E_{\bf p}({\bf x})$ is  independent of $u_{\bf x}\vert_{\rm neck}$.
\item  $E_{\bf p}({\bf x})$ depends smoothly on ${\bf x}$ with respect to 
$({\bf p},{\frak W}_{\bf p})$ for {\it any} choice of ${\frak W}_{\bf p}$.
\item
Let  ${\bf r} = ((\Sigma_{{\bf r}},\vec z_{{\bf r}},\vec {\frak z}_{{\bf r}}),u_{{\bf r}}) 
\in \mathcal{X}_{k+1,\ell}(X,L,J;\beta)$ be as in Remark \ref{rem83}.
Then for any $({\bf r},{\frak W}_{{\bf r}})$, the same conclusion 
as (2) holds.
\end{enumerate}
\end{defn}
We will elaborate (3) at the end of this subsection.
\begin{rem}
In our previous writing such as \cite{FO,foootech,foooexp}
we defined the obstruction spaces $E_{\bf p}({\bf x})$ 
in the way we will describe in  Section \ref{sec:exiobst}.
We will prove in Section \ref{sec:exiobst} that it satisfies Definition \ref{defn899}.
\par
On the other hand, the gluing analysis such as those in \cite{foooexp}
works not only for this particular choice but also for more general 
$E_{\bf p}({\bf x})$ which satisfies Definition \ref{defn899}.
In fact, in some situation such as in \cite{foootoric32,FuFu5} (where we studied 
an action of a compact Lie group on the target space), we used 
somewhat different choice of $E_{\bf p}({\bf x})$ where
Definition \ref{defn899} is also satisfied.
(See \cite[Subsection 7.4]{FuFu5} and \cite[Appendix]{foootoric32}, for example.) 
Other methods of defining $E_{\bf p}({\bf x})$ 
may be useful also in the future in some other situations.
\par
Therefore, formulating the condition for $E_{\bf p}({\bf x})$ to  satisfy, such as 
Definition \ref{defn899}, rather than using some specific choice 
of $E_{\bf p}({\bf x})$ is more flexible and widens the scope of  
its applications.
\end{rem}
We now explain Definition \ref{defn8911} (3).
We can construct a Kuranishi structure of $C^n$ class for any but fixed $n$ 
without using this condition.
This condition is used to obtain a Kuranishi structure of $C^{\infty}$ class.
See Section \ref{sec:Cmugen}.
\par
Let ${\bf r}$ be as in Remark \ref{rem83}.
We can define the notion of stabilization and trivialization data ${\frak W}_{{\bf r}}$.
We also define $\mathcal V({\bf r}\cup \vec{\frak w}_{{\bf r}};\vec\epsilon)$ in the same way as (\ref{form86}).
Then, for each $\frak y\in\mathcal V({\bf r}\cup \vec{\frak w}_{{\bf r}};\vec\epsilon)$, we can 
associate $(\Sigma_{\frak y},\vec z_{\frak y},\vec{\frak z}_{\frak y})$ 
in the same way  and 
obtain a diffeomorphism
$
\widehat{\Phi}_{\frak y,\vec\epsilon}: \Sigma_{{\bf r}\cup \vec{\frak w}_{{\bf r}}}(\vec\epsilon) \to \Sigma_{\frak y}(\vec\epsilon).
$
Let $\mathscr L_m$ be a small neighborhood of $u_{{\bf r}\cup \vec{\frak w}_{{\bf r}}}\vert_{\Sigma_{{\bf r}\cup \vec{\frak w}_{{\bf r}}}(\vec\epsilon)}$
in $L^2_{m}((\Sigma_{{\bf r}\cup \vec{\frak w}_{\bf r}},\partial\Sigma_{{\bf r}
\cup \vec{\frak w}_{{\bf r}}}(\vec\epsilon));X,L)$.
Now for each $u' \in \mathscr L_m$ and 
${\frak y}\in\mathcal V({\bf r}\cup \vec{\frak w}_{{\bf r}};\vec\epsilon)$ we use 
$E_{\bf p}({\bf x})$\footnote{This is $E_{\bf p}({\bf x})$
and is not $E_{{\bf r}}({\bf x})$. The later is not defined.} for ${\bf x} = \frak{forget}_{\ell+\ell',\ell}
(\Phi_{{\bf r} \cup \vec{\frak w}_{\bf r}}({\frak y}))$ in the same way as above 
to define 
$
E_{{\bf p};{{\frak W}_{{\bf r}}}}({\frak y},u')
\subset L^2_m(\Sigma_{{\bf r}\cup \vec{\frak w}_{{\bf r}}}(\vec\epsilon);u_{{\bf r}}^*TX\otimes \Lambda^{01}).
$
Definition \ref{defn8911} (3) requires that this is a $C^n$ family 
parameterized by  $({\frak y},u')$ for any ${\frak W}_{{\bf r}}$ if $m$ is large and $\vec{\epsilon}$ 
is small.

\section{Kuranishi charts are of $C^n$ class}
\label{sec;Kuracharsmooth}

In this section we review how the gluing analysis (especially those 
detailed in \cite{foooexp}) implies that the construction of 
Section \ref{sec;kuraconst} provides Kuranishi charts of $C^n$ class.
In other words we prove the $C^n$ version of Lemma \ref{lem73}.

\subsection{Another smooth structure on the moduli space of source curves}
\label{subsec:smoothanother}

As was explained in \cite[Remark A1.63]{fooobook2} the standard smooth structure 
of $\mathcal M_{k+1,\ell}^{\rm d}$ is not appropriate to define 
smooth Kuranishi charts of $\mathcal M_{k+1,\ell}^{\rm d}(X,L;\beta)$.
Following the discussion of \cite[Subsection A1.4]{fooobook2}, 
we will define another smooth structure on $\mathcal M_{k+1,\ell}^{\rm d}$ 
in this subsection.
(The notion of profile due to Hofer, Wysocki and Zehnder \cite[Section 2.1]{hwze}
is related to this point.)
We consider the map (\ref{form33}).
\par
Let $r_j \in [0,c)$ ($j=1,\dots,m_{\rm d}$) be the standard coordinates of $[0,c)^{m_{\rm d}}$ 
and $\sigma_i \in D^2_{\circ}(c)$ ($i=1,\dots,m_{\rm s}$)  
the standard coordinates of $(D^2_{\circ}(c))^{m_{\rm s}}$.
\index{00R1dj@$r_j$}\index{00Sigma1sj@$\sigma_i$}
We put
\index{00T1dj@$T^{\rm d}_j$}\index{00T1sj@$T^{\rm s}_i$}
\index{00Thetai@$\theta_i$}
\begin{equation}\label{form91}
\aligned
&T^{\rm d}_j = -\log r_j \in \R_{+} \cup \{\infty\},  \\
&T^{\rm s}_i = -\log \vert\sigma_i\vert\in \R_{+}\cup \{\infty\}, 
\quad \theta_i = -\rm{Im}(\log \sigma_i) \in \R/2\pi\Z.
\endaligned
\end{equation}
We then define 
\index{00S1j@$s_j$}
\index{00Rhoi@$\rho_i$}
\begin{equation}\label{form92}
s_j = 1/T^{\rm d}_j \in [0,-1/\log c), 
\quad
\rho_i =  \exp(\theta_i\sqrt{-1})/T^{\rm s}_j \in D^2(-1/\log c).
\end{equation}
\par
Composing  these coordinate changes with the map $\Phi$ in Lemma \ref{lem34}, we define
\begin{equation}\label{form9393}
\Phi_{s,\rho} : 
\prod_{a\in \mathcal A_{\bf p}^{\rm s} \cup \mathcal A_{\bf p}^{\rm d}} \mathcal V_a 
\times [0,-1/\log c)^{m_{\rm d}} \times (D^2_{\circ}(-1/\log c))^{m_{\rm s}}
\to \mathcal M^{\rm d}_{k+1,\ell}.
\end{equation}
\begin{lem}\label{lem91}
There exists a unique structure of smooth manifold with 
corners on $\mathcal M^{\rm d}_{k+1,\ell}$ 
such that $\Phi_{s,\rho}$ is a diffeomorphism onto its image for each 
${\bf p} \in \mathcal M^{\rm d}_{k+1,\ell}$.
\end{lem}
Note in this subsection ${\bf p}, {\bf q}$ are elements of 
$\mathcal M^{\rm d}_{k+1,\ell}$ and not of $\mathcal M^{\rm d}_{k+1,\ell}(X,L,J;\beta)$.
\begin{proof}
During this proof we write $\Phi^{\bf p}$ etc. to clarify that it is associated to 
${\bf p} \in \mathcal M^{\rm d}_{k+1,\ell}$.
We denote by $\frak v^{\bf p}_a$ an 
element of the first factor of the domain of (\ref{form9393})
for ${\bf p}$.
\par
Suppose 
${\bf q} \in {\rm Im}(\Phi^{\bf p})$.
Then $m_{\rm d}^{\bf q} \le m_{\rm d}^{\bf p}$, $m_{\rm s}^{\bf q} \le m_{\rm s}^{\bf p}$.
We may enumerate the marked points so that the $j$-th boundary node 
(resp. the $i$-th interior node) of $\Sigma_{\bf p}$ corresponds to the
$j$-th boundary node (resp. the $i$-th interior node) of $\Sigma_{\bf q}$ for 
$j=1,\dots,m_{\rm d}^{\bf q}$ (resp. $i=1,\dots,m_{\rm s}^{\bf q}$).
Then we can easily prove the next inequalities:
\begin{equation}\label{form94}
\aligned
\left\Vert \nabla^{n-1} \frac{\partial}{\partial T_{j}^{\rm d,\bf q}}(T^{\rm d,\bf p}_{j_0} - T^{\rm d,\bf q}_{j_0})
\right\Vert &\le C_n e^{-c_n T_{j}^{\rm d,\bf q}} \\
\left\Vert \nabla^{n-1} \frac{\partial}{\partial T_{i}^{{\rm s},\bf q}}(T^{\rm d,\bf p}_{j_0} - T^{\rm d,\bf q}_{j_0})
\right\Vert &\le C_n e^{-c_n T_{i}^{{\rm s},\bf q}} \\
\left\Vert \nabla^{n-1} \frac{\partial}{\partial \theta_{i}^{\bf q}}(T^{\rm d,\bf p}_{j_0} - T^{\rm d,\bf q}_{j_0})
\right\Vert &\le C_n e^{-c_n T_{i}^{{\rm s},\bf q}} 
\endaligned 
\end{equation}
for $j_0 = 1,\dots,m_{\rm d}^{\bf q}$.
Here $\nabla^{n-1}$ is the ($n-1$)-th derivatives on the variables  $\frak v^{\bf q}_a$, $T_{j}^{\rm d,\bf q}$
$T_{i}^{\rm s,\bf q}$, $\theta_{i}^{\bf q}$
and $\Vert \cdot \Vert$ is the $C^0$ norm.
\par
In fact, to prove the 2nd and 3rd inequalities of (\ref{form94}), we use the fact that 
$\sigma^{\bf p}_i$, $\sigma^{\bf q}_i$ are holomorphic functions defining the same 
divisor to show that  $\sigma^{\bf p}_i/\sigma^{\bf q}_i$ is a nowhere vanishing holomorphic 
function. Then in the same way as \cite[Sublemma 8.29]{foooexp}
we obtain the 2nd and 3rd inequalities of (\ref{form94}).
The 1st inequality is proved in the same way by taking the double 
as in Section \ref{subsec;universal}.
\par
We can prove the same inequality with $T^{\rm d,\bf p}_{j_0} - T^{\rm d,\bf q}_{j_0}$
replaced by $T^{\rm s,\bf p}_{i_0} - T^{\rm s,\bf q}_{i_0}$,
$\theta^{\bf p}_{i_0} - \theta^{\bf q}_{i_0}$, ($i_0 = 1,\dots,m_{\rm s}^{\bf p}$), $s_{j_0}$ (${j_0}>m_{\rm d}^{\bf q}$), 
$\sigma_{i_0}$ (${i_0}>m_{\rm s}^{\bf q}$) or
coordinates of $\frak v_a^{\bf p}$.
\par
In fact, the estimates for $s_{j_0}$ (${j_0}>m_{\rm d}^{\bf q}$), 
$\sigma_{i_0}$ (${i_0}>m_{\rm s}^{\bf q}$) or  coordinates of $\frak v_a^{\bf p}$ are proved using the fact that 
they are smooth functions of $r_j^{\bf q}$, $\frak v_a^{\bf q}$ and $\sigma_i^{\bf q}$.
\par
These facts combined with strata-wise smoothness of 
$(\Phi_{s,\rho}^{\bf p})^{-1}\circ \Phi_{s,\rho}^{\bf q}$ 
imply that the coordinate change 
$(\Phi_{s,\rho}^{\bf p})^{-1}\circ \Phi_{s,\rho}^{\bf q}$  is smooth.
The lemma is a consequence of this fact.
\end{proof}
Hereafter we write $\mathcal M^{{\rm d}.\log}_{k+1,\ell}$  
\index{00M3^{{\rm d}.\log}_{k+1,\ell}@$\mathcal M^{{\rm d}.\log}_{k+1,\ell}$}
when we use the smooth structure given in Lemma \ref{lem91}.

\subsection{Gluing analysis: review}
\label{subsec:glue}

We review the conclusion of the gluing analysis of \cite[Theorem 6.4]{foooexp} in this subsection.
\par
We take $m$ sufficiently larger than $n$. Especially it is larger than $m_0$ 
appearing in Definition \ref{defn86}.
Let $\{E_{\bf p}({\bf x})\}$ be  obstruction bundle data at ${\bf p} \in \mathcal M^{\rm d}_{k+1,\ell}(X,L,J;\beta)$.
We fix the stabilization and trivialization data ${\frak W}_{\bf p}$
and put ${\bf p}^+ = {\bf p} \cup \vec{\frak w}_{\bf p}$.
We decompose $\Sigma_{{\bf p}^+}$ into irreducible components 
$$
\Sigma_{{\bf p}^+}
= \bigcup_{a \in \mathcal A_{\bf p}^{\rm s}} \Sigma_{{\bf p}^+}(a)
\cup \bigcup_{a \in \mathcal A_{\bf p}^{\rm d}} \Sigma_{{\bf p}^+}(a).
$$
Let ${\bf p}^+_a$ be as in (\ref{formpa1})(\ref{formpa2})
\footnote{Note we consider ${\bf p}^+$ here in place of ${\bf p}$ in (\ref{formpa1})(\ref{formpa2}).} and
${\mathcal V}^+_a$  a neighborhood of the source curve of ${\bf p}^+_a$ in 
$\mathcal M^{\rm s,reg}_{\ell(a)}$ or $\mathcal M^{\rm d,reg}_{k(a)+1,\ell(a)}$.
We put
$$
\mathcal V^+ = \prod {\mathcal V}^+_a.
$$
For $\frak v \in \mathcal V^+$ we obtain $(\Sigma(\frak v),\vec z(\frak v),\vec{\frak z}(\frak v)) 
= \Phi(\frak v)$
with the same number of irreducible components as $
\Sigma_{{\bf p}^+}$.
(Namely we put all the gluing parameters to be $0$.)
Using the given trivialization data we obtain a diffeomorphism
$\widehat{\Phi}_{\frak v} : \Sigma_{{\bf p}^+} 
\to \Sigma(\frak v)$ which preserves the marked and the nodal points.
\begin{defn}\label{defn9292}
By $\mathscr V({\bf p};E_{\bf p}(\cdot);\epsilon_0)$, 
\index{00V3({\bf p};E_{\bf p}(\cdot);\epsilon_0)@$\mathscr V({\bf p};E_{\bf p}(\cdot);\epsilon_0)$}
we denote the set of pairs $(\frak v,u')$
such that:
\begin{enumerate}
\item $\frak v \in \mathcal V^+$.
\item  $u' : (\Sigma(\frak v),\partial \Sigma(\frak v)) 
\to (X,L)$ is an $L^2_m$ map such that the $L^2_m$-difference between $u' \circ \widehat{\Phi}_{\frak v}$ 
and $u$ is smaller than $\epsilon_0$.
\item
\begin{equation}\label{form9595}
\overline{\partial}u' \in E_{{\bf p}}(\frak v,u').
\end{equation}
Here
$
E_{{\bf p}}(\frak v,u')
= 
E_{{\bf p}}(\frak{forget}_{\ell+\ell',\ell}(\Phi(\frak v),u'))
\subset L^2_m(\Sigma(\frak v);(u')^*TX\otimes \Lambda^{01})
$
is the case of $E_{\bf p}({\bf x})$ when ${\bf x} = \frak{forget}_{\ell+\ell',\ell}({\bf y})$,
${\bf y} = ((\Sigma(\frak v),\vec z(\frak v),\vec{\frak z}(\frak v)),u')$.
\end{enumerate}
\end{defn}
We define maps
\begin{equation}\label{map96}
{\rm Pr}^{\rm source} :
\mathscr V({\bf p};E_{\bf p}(\cdot);\epsilon_0) 
\to \prod_{a \in \mathcal A_{\bf p}^{\rm s}} \mathcal M^{\rm s,reg}_{\ell(a)} 
\times \prod_{a \in \mathcal A_{\bf p}^{\rm d}} \mathcal M^{\rm d,reg}_{k(a)+1,\ell(a)},
\end{equation}
\begin{equation}
\aligned
{\rm Pr}^{\rm map} :
&\mathscr V({\bf p};E_{\bf p}(\cdot);\epsilon_0) \\
&\to
\prod_{a \in \mathcal A_{\bf p}^{\rm s}} L^2_m(\Sigma_{{\bf p}^+}(a);X) 
\times \prod_{a \in \mathcal A_{\bf p}^{\rm d}} L^2_m(\Sigma_{{\bf p}^+}(a),\partial\Sigma_{{\bf p}^+}(a);X,L)
\endaligned
\end{equation}
by
$$
{\rm Pr}^{\rm source}(\frak v,u') = \frak v \quad \text{and} \quad
{\rm Pr}^{\rm map}(\frak v,u') = 
\left(
u' \circ \widehat{\Phi}_{\frak v}\vert_{\Sigma_{{\bf p}^+}(a)} : a \in \mathcal A_{\bf p}^{\rm s}\cup \mathcal A_{\bf p}^{\rm d}
\right).
$$
\begin{lem}
There exists a unique $C^n$ structure on $\mathscr V({\bf p};E_{\bf p}(\cdot);\epsilon_0)$ 
such that $({\rm Pr}^{\rm source},{\rm Pr}^{\rm map})$ is a $C^n$ embedding.
\par
Moreover the action of ${\rm Aut}^+(\bf p)$ is of $C^n$ class and $({\rm Pr}^{\rm source},{\rm Pr}^{\rm map})$
is ${\rm Aut}^+(\bf p)$-equivariant.
\end{lem}
\begin{proof}
This is a consequence of Definition \ref{defn51} (2)(3), Lemma \ref{conds55lem} and the implicit function theorem.
In fact Definition \ref{defn51} (3) implies that hypothesis of the implicit function theorem
is satisfied.
\end{proof}
$\mathscr V({\bf p};E_{\bf p}(\cdot);\epsilon_0)$ is a part of the `thickened' moduli space 
consisting of elements that have the same number of nodal points as ${\bf p}$.
We next include the gluing parameter.
Recall from
Definition \ref{defn72} that 
$
U_{{\bf p}^+}
$ for ${\bf p}^+ = {\bf p} \cup \vec{\frak w}_{\bf p}$
\index{00U1p+@$U_{{\bf p}^+}$}
is the set of all ${\bf x} \in \mathscr U_{{\bf p}^{+}}$ such that
\begin{equation}\label{eq9898}
\overline{\partial} u_{\bf x} \in E_{\bf p}(\bf x).
\end{equation}
Here $\mathscr U_{{\bf p}^{+}}$ is $B_{\epsilon_0}(\mathcal{X}_{k+1,\ell}(X,L,J;\beta),
{\bf p}\cup \vec{\frak w}_{\bf p},\frak W_{{\bf p} \cup \vec {\frak w}_{\bf p}})$ 
for some small $\epsilon_0$.
\par
We define a map 
\begin{equation}
\aligned
{\rm Pr}^{\rm map} :
U_{{\bf p}^+} \to L^2_m(\Sigma_{{\bf p} \cup \vec {\frak w}_{\bf p}}(\vec{\epsilon}),
\partial \Sigma_{{\bf p} \cup \vec {\frak w}_{\bf p}}(\vec{\epsilon});X,L)
\endaligned
\end{equation}
below. We first observe that ${\bf p} \cup \vec {\frak w}_{\bf p}$ has no 
nontrivial automorphism.
(It may have a nontrivial extended automorphism.)
Therefore if ${\bf x} \in \mathscr U_{{\bf p}^{+}}$
and $\epsilon_0$ is sufficiently small there exist {\it unique}
$\frak v \in \mathcal V^+$ and $(s_j)_{j=1}^{m_{\rm d}},(\rho_i)_{i=1}^{m_{\rm d}}$
such that 
$$
(\Sigma_{\bf x},\vec z_{\bf x},\vec{\frak z}_{\bf x}) 
= 
\Phi_{s,\rho}(\frak v,(s_j),(\rho_i)),
$$
where $\Phi_{s,\rho}$ is as in (\ref{form9393}).
We put $\frak y = (\frak v,(s_j),(\rho_i))$.
Then by Lemma \ref{lem3838} we obtain a smooth embedding:
$
\widehat{\Phi}_{\frak y;\vec \epsilon} : \Sigma_{{\bf p}^+}(\vec \epsilon) \to \Sigma_{\frak y}
= \Sigma_{\bf x}.
$
We define 
\index{00P1rmap@${\rm Pr}^{\rm map}$}
\begin{equation}
{\rm Pr}^{\rm map}({\bf x}) = u_{\bf x} \circ \widehat{\Phi}_{\frak y;\vec \epsilon}
: 
 \Sigma_{{\bf p}^+}(\vec \epsilon) \to X.
\end{equation}
We also define
$
{\rm Pr}^{\rm source} :
U_{{\bf p}^+}
\to \mathcal M^{\rm d,log}_{k+1,\ell+\ell'}
$\index{00P1rsoiurce@${\rm Pr}^{\rm source}$}
by
$
{\rm Pr}^{\rm source}({\bf x}) = [\Sigma_{\bf x},\vec z_{\bf x},\vec{\frak z}_{\bf x}].
$
They together define :
\begin{equation}\label{form911}
({\rm Pr}^{\rm source},{\rm Pr}^{\rm map}) : U_{{\bf p}^+} 
\to 
\mathcal M^{\rm d,log}_{k+1,\ell+\ell'}\times 
L^2_m(\Sigma_{{\bf p}^+}(\vec{\epsilon}),
\partial \Sigma_{{\bf p}^+}(\vec{\epsilon});X,L).
\end{equation}
The target of the map (\ref{form911}) has a structure of Hilbert manifold 
since it is a direct product of a Hilbert space and a smooth manifold.
\begin{prop}\label{prop94}
If $m$ is large enough  compared to $n$ and $\epsilon_0$, $\vec{\epsilon}$ are small,
then the image of the map (\ref{form911}) is a finite dimensional submanifold 
of $C^n$ class.
\par
Moreover the map (\ref{form911}) is injective.
\end{prop}
\begin{proof}
Below we explain how we use 
\cite[Theorem 6.4]{foooexp}  to prove Proposition \ref{prop94}.
\cite{foooexp} discusses the case when $\Sigma_{{\bf p}}$ has 
two irreducible components. However the argument there 
can be easily generalized to the case when it has arbitrarily many 
irreducible components. (See also 
\cite[Section 19]{foootech} where the same gluing analysis is discussed 
in the general case.)
We consider 
\begin{equation}\label{new912}
\mathcal V = 
\mathscr V({\bf p};E_{\bf p}(\cdot);\epsilon_0)
\times
\prod_{j=1}^{m_{\rm d}} [0,{\epsilon_{j}^{\rm d}})
\times
\prod_{i=1}^{m_{\rm s}} D^2_{\circ}({\epsilon_{i}^{\rm s}}).
\end{equation}
We change the variables from $r_j \in [0,{\epsilon_{j}^{\rm d}})$, 
$\sigma_i \in D^2_{\circ}({\epsilon_{i}^{\rm s}})$
to
$s_j \in [0,-1/\log(\epsilon_{j}^{\rm d}))$,
$\rho_i \in D^2_{\circ}(-1/\log({\epsilon_{i}^{\rm s}}))$
by (\ref{form92}).
We write $\mathcal V^{\log}$ 
\index{00V3^{\log}@$\mathcal V^{\log}$} when we use the smooth structure 
so that $s_j$, $\rho_i$ are the coordinates.
\par
\begin{rem}
The identity map $\mathcal V^{\log} \to \mathcal V$ is smooth 
but $\mathcal V \to \mathcal V^{\log}$ is not smooth.
\end{rem}

In \cite[Theorems 3.13,8.16]{foooexp} the map
$
{\rm Glue} : \mathcal V^{\log} \to U_{{\bf p}^+}
$
is constructed as follows.
\par
Let $((\frak v,u'),(r_j),(\sigma_i)) \in \mathcal V^{\log}$.
We put 
$
(\Sigma_{\bf x},\vec z_{\bf x},\vec{\frak z}_{\bf x}) = \Phi(\frak v,(r_j),(\sigma_i)).
$
(Namely we glue the source curve $\Sigma_{\frak v} = \Phi(\frak v)$ by 
using the gluing parameter $(r_j),(\sigma_i)$.) 
\par
Using
$u' : \Sigma_{\frak v} \to X$ and a partition of unity 
we obtain a map 
$u_{(0)} : \Sigma_{\bf x} \to X$
(In other words, this is the map \cite[(5.4)]{foooexp}
and is the map obtained by `pre-gluing'.)
The map $u_{(0)}$ mostly satisfies the equation  (\ref{eq9898}).
However at the neck region  $\overline{\partial}u_{(0)}$ has certain error term.
We can solve the linearized equation 
of (\ref{eq9898})  using the assumption Definition \ref{defn51} (3)
and the `alternating method'. Then by Newton's iteration scheme 
we inductively obtain $u_{(a)}$ ($a=1,2,3,\dots$).
By using Definition \ref{defn51} (2) 
we can carry out the estimate we need to work out 
this iteration process (\cite[Section 5]{foooexp}).
Then $\lim_{a \to \infty} u_{(a)}$ converges to a 
solution of (\ref{eq9898}), which is by definition $u_{\bf x}$.
We define
\index{00G1lue@${\rm Glue}$}
$$
{\rm Glue}((\frak v,u'),(r_j),(\sigma_i)) = ((\Sigma_{\bf x},\vec z_{\bf x},\vec{\frak z}_{\bf x}),u_{\bf x}).
$$
Replacing $\mathcal V^{\log}$ by its open subset, the map ${\rm Glue}$ defines a bijection between  
$\mathcal V^{\log}$ and $U_{{\bf p}^+}$.
(This is a consequence of \cite[Section 7]{foooexp}.)
\par
To prove Proposition \ref{prop94} it suffices to show 
that $({\rm Pr}^{\rm source},{\rm Pr}^{\rm map})\circ{\rm Glue}$ is a $C^n$ embedding.
Note the smooth coordinates we use here are $s_j$ and $\rho_i$ given in 
(\ref{form92}).
By definition and Lemma \ref{lem91}, the map ${\rm Pr}^{\rm source} \circ {\rm Glue}$
is a smooth submersion with respect to this smooth structure.
\par
We use the coordinates $T^{\rm d}_j$, $T^{\rm s}_i$ and $\theta_i$
in place of $s_j$ and $\rho_i$ for the gluing parameter
and denote
$$
({\rm Pr}^{\rm map}({\rm Glue}((\frak v,u'),(T^{\rm d}_j),(T^{\rm s}_i,\theta_i))))(z)
= u(((\frak v,u'),(T^{\rm d}_j),(T^{\rm s}_i,\theta_i));z).
$$
Here $z \in \Sigma_{{\bf p} \cup \vec {\frak w}_{\bf p}}(\vec{\epsilon})$
is the domain variable of 
${\rm Pr}^{\rm map}({\rm Glue}(\frak v,(T^{\rm d}_j),(T^{\rm s}_i,\theta_i)))$.
Then the conclusion of \cite[Theorem 6.4]{foooexp} is the next inequalities:
\begin{equation}\label{form9224}
\aligned
\left\Vert \nabla^{n'-1} \frac{\partial}{\partial T_{j}^{\rm d}}
(u(((\frak v,u'),(T^{\rm d}_j),(T^{\rm s}_i,\theta_i));\cdot))
\right\Vert_{L^2_{m-n'}} &\le C_n e^{-c_n T^{\rm d}_{j}} \\
\left\Vert \nabla^{n'-1} \frac{\partial}{\partial T_{i}^{{\rm s}}}
(u(((\frak v,u'),(T^{\rm d}_j),(T^{\rm s}_i,\theta_i));\cdot))
\right\Vert_{L^2_{m-n'}} &\le C_n e^{-c_n T_{i}^{\rm s}} \\
\left\Vert \nabla^{n'-1} \frac{\partial}{\partial \theta_{i}}
(u(((\frak v,u'),(T^{\rm d}_j),(T^{\rm s}_i,\theta_i));\cdot))
\right\Vert_{L^2_{m-n'}} &\le C_n e^{-c_n T_{i}^{\rm s}} 
\endaligned 
\end{equation}
for $j = 1,\dots,m_{\rm d}^{\bf q}$ and $n' \le n$.
Here $\nabla^{n'-1}$ is the ($n'-1$)-th derivatives on the variables  $\frak v^{\bf q}_a$, $T_{j}^{\rm d,\bf q}$
$T_{i}^{\rm s,\bf q}$, $\theta_{i}^{\bf q}$.
\par
From these inequalities it is easy to see that ${\rm Pr}^{\rm map} \circ {\rm Glue}$ 
is of $C^n$ class.
\par
We now fix $(\frak v,(T^{\rm d}_j),(T^{\rm s}_i,\theta_i))$
and consider the map
\begin{equation}\label{form913}
u' \mapsto u((\frak v,u'),(T^{\rm d}_j),(T^{\rm s}_i,\theta_i));\cdot)
\end{equation}
This is a map
$$
\mathscr L \to L^2_m(\Sigma_{{\bf p}^+}(\vec{\epsilon}),
\partial \Sigma_{{\bf p}^+}(\vec{\epsilon});X,L)
$$
where $\mathscr L$ is the set of $u'$ satisfying Definition \ref{defn9292} (2)(3).
\par
To complete the proof of Proposition \ref{prop94} it suffices to 
show that (\ref{form913}) is a $C^n$ {\it embedding}.
Using (\ref{form9224}) again it suffices to prove it in case $\Phi(\frak v,(T^{\rm d}_j),(T^{\rm s}_i,\theta_i))
= (\Sigma_{{\bf p}^+},\vec z_{{\bf p}^+},\vec {\frak z}_{{\bf p}^+})$
(by taking a smaller neighborhood of ${\bf p}$ if necessary).
In that case (\ref{form913}) is nothing but the restriction map.
Therefore by the unique continuation (\ref{form913}) is a $C^n$ embedding.
\end{proof}

\begin{lem}\label{lem9555}
The group ${\rm Aut}^+({\bf p})$ of extended automorphisms of ${\bf p}$ 
has $C^n$ action on $U_{{\bf p}^+}$.
The map (\ref{form911}) is ${\rm Aut}^+({\bf p})$ 
equivariant.
\end{lem}
\begin{proof}
This is immediate from Lemma \ref{conds55lem}.
\end{proof}

Thus we obtain a $C^n$ orbifold $U_{{\bf p}^+}/{\rm Aut}({\bf p})$ with ${\bf p}^+ 
={\bf p}\cup \vec{\frak{w}}_{\bf p}$.

\subsection{Local transversal and stabilization data}
\label{subsec:stabfunc}

The $C^n$ orbifold $U_{{\bf p}^+}/{\rm Aut}({\bf p})$ obtained in the last subsection 
is not the Kuranishi neighborhood appearing in the Kuranishi chart we look for.
In fact it still contains the extra parameters 
to move $(\ell+1)$-th, \dots, $(\ell+\ell')$-th interior marked points.
To kill these parameters we proceed in the same way as \cite[appendix]{FO}
to use local transversals.
We use the same trick in Section \ref{sec:exiobst} to prove the 
existence of  obstruction bundle data.

\begin{defn}\label{defn81}
Let ${\bf p} = ((\Sigma_{\bf p},\vec z_{\bf p},\vec {\frak z}_{\bf p}),u_{\bf p}) 
\in \mathcal{M}_{k+1,\ell}(X,L,J;\beta)$.
{\it Stabilization data at ${\bf p}$} are by definition 
weak stabilization data  $\vec {\frak w}_{\bf p} = ({\frak w}_{{\bf p},1},\dots,{\frak w}_{{\bf p},\ell'})$
as in Definition \ref{defn46}
together with $\vec{\mathcal N}_{\bf p} = \{ \mathcal N_{{\bf p},i} \mid i=1,\dots,\ell'\}$
\index{00N3vec@$\vec{\mathcal N}_{\bf p}$} which have the following 
properties.
\par
\begin{enumerate}
\item $\mathcal N_{{\bf p},i}$ is a codimension $2$ submanifold of $X$.
\item There exists a neighborhood $U_i$ of $\frak w_{{\bf p},i}$ in $\Sigma_{\bf p}$ 
such that $u_{\bf p}(U_i)$ intersects transversally with $\mathcal N_{{\bf p},i}$ at unique point $u_{\bf p}(\frak w_{{\bf p},i})$.
Moreover, the restriction of $u_{\bf p}$ to $U_i$ is a smooth embedding.
\item Suppose $v : \Sigma_{\bf p} \to \Sigma_{\bf p}$ is an extended automorphism of 
${\bf p}$ and  $v(\frak w_{{\bf p},i}) = \frak w_{{\bf p},i'}$.
Then $\mathcal N_{{\bf p},i} = \mathcal N_{{\bf p},i'}$ and $v(U_i) = U_{i'}$.
\end{enumerate}
We call $\mathcal N_{{\bf p},i}$ a {\it local transversal}
and $\vec{\mathcal N}_{\bf p}$ local transversals.\index{local transversal}
\end{defn}
Local transversals are used to choose $\ell'$ additional marked points in a canonical way  
for each ${\bf x} \in \mathcal{X}_{k+1,\ell}(X,L,J;\beta)$.
Lemma \ref{lem832} below formulates it precisely.

\begin{shitu}\label{shitu82}
Let ${\bf p} = ((\Sigma_{\bf p},\vec z_{\bf p},\vec {\frak z}_{\bf p}),u_{\bf p}) 
\in \mathcal{M}_{k+1,\ell}(X,L,J;\beta)$.
We take its stabilization data  $(\vec {\frak w}_{\bf p},\vec{\mathcal N}_{\bf p})$.
We also take 
$\{\varphi_{a,i}^{\rm s}\}, \{\varphi_{a,j}^{\rm d}\},
\{\phi_a\}$
so that $\frak W_{\bf p} = (\vec {\frak w}_{\bf p},\{\varphi_{a,i}^{\rm s}\}, \{\varphi_{a,j}^{\rm d}\},
\{\phi_a\})$ become  stabilization and trivialization data 
in the sense of Definition \ref{situation38}.
We call $(\frak W_{\bf p},\vec{\mathcal N}_{\bf p})$ {\it strong stabilization data}.\index{strong stabilization data}
$\diamond$
\end{shitu}

\begin{lem}\label{lem832}
Suppose we are in Situation \ref{shitu82}.
There exists $\epsilon_0 > 0$ and $o(\epsilon)$ with $\lim_{\epsilon \to 0}o(\epsilon) = 0$ that have 
the following properties.\par   
If ${\bf x} \in B_{\epsilon}(\mathcal{X}_{k+1,\ell}(X,L,J;\beta);{\bf p};\frak W_{\bf p})$,
$\epsilon \in (0,\epsilon_0)$, then there exists $\vec {\frak w}_{\bf x}$ such that:
\begin{enumerate}
\item 
$
{\bf x} \cup \vec {\frak w}_{\bf x} \in 
B_{o(\epsilon)}(\mathcal{X}_{k+1,\ell+\ell'}(X,L,J;\beta);{\bf p}
\cup \vec {\frak w}_{\bf p};\frak W_{{\bf p}\cup \vec {\frak w}_{\bf p}}).
$
Note the right hand side is defined in Definition \ref{defn411}.
\item
$u_{\bf x}({\frak w}_{{\bf x},i}) \in \mathcal N_{{\bf p},i}$ for $i=1,\dots,\ell'$.
\end{enumerate}
\par
Moreover $\vec {\frak w}_{\bf x}$ satisfying (1)(2) is unique up to the action of 
${\rm Aut}({\bf p})$.
Elements of ${\rm Aut}^+({\bf p})$ preserve $\vec {\frak w}_{\bf x}$ as a set.
\end{lem}
\begin{proof}
According to Definition \ref{defn411}, 
${\bf x} \in B_{\epsilon}(\mathcal{X}_{k+1,\ell}(X,L,J;\beta);{\bf p};\frak W_{\bf p})$,
implies that there {\it exists} $\vec {\frak w}^0_{\bf x}$ such that
$$
{\bf x} \cup \vec {\frak w}^0_{\bf x} \in 
B_{o(\epsilon)}(\mathcal{X}_{k+1,\ell+\ell'}(X,L,J;\beta);{\bf p}\cup \vec {\frak w}_{\bf p};\frak W_{{\bf p}\cup \vec {\frak w}_{\bf p}}).
$$
We  use the implicit function theorem and the fact 
that $u_{\bf x}$ is $C^2$ close to $u_{\bf p}$ to
prove that 
there exists $\vec {\frak w}_{\bf x}$ in a small neighborhood of $\vec {\frak w}^0_{\bf x}$
such that (2) is satisfied. 
It is then easy to see that (1) is also satisfied.
\par
In case ${\bf x} = {\bf p}$, the uniqueness of $\vec {\frak w}_{\bf x}$ up to 
${\rm Aut}({\bf p})$ action is obvious.
We can use the $C^2$ small isotopy between $u_{\bf x}$ and $u_{\bf p}$
(defined outside of the neck region) to reduce  the proof for the general case 
to the case ${\bf x} = {\bf p}$.
\end{proof}

\subsection{$C^n$ structure of the Kuranishi chart}
\label{subsec:Cnkurastru}

We now prove the $C^n$-version of Lemma \ref{lem73}
using the construction of the last two subsections.
Suppose we are in the situation of Proposition \ref{prop94}.
In addition to the stabilization and trivialization data 
$\frak W_{\bf p}$ we have already chosen  
local transversals $\vec{\mathcal N}_{\bf p}$ so that
$(\frak W_{\bf p},\vec{\mathcal N}_{\bf p})$ are strong stabilization data.

\begin{defn}\label{defn910}
We define $V_{\bf p}$ 
\index{00V3p@$V_{\bf p}$} to be the subset of $U_{{\bf p}^+}$
(with ${\bf p}^+ = {\bf p} \cup \vec{\frak w}_{\bf p}$)
consisting of ${\bf x} = ((\Sigma_{\bf x},\vec z_{\bf x},\vec{\frak z}_{\bf x}),u_{\bf x})$ 
such that
\begin{equation}
u_{\bf x}({\frak z}_{{\bf x},\ell+i}) \in \mathcal N_{{\bf p},i},
\quad \text{for $i=1,\dots,\ell'$}.
\end{equation}
\end{defn}
We remark that the points ${\frak z}_{{\bf x},\ell+i}$, $i=1,\dots,\ell'$, correspond to the additional 
marked points $\vec{\frak w}_{\bf p}$.
\begin{lem}\label{lem91111}
$V_{\bf p}$ is a $C^n$ submanifold of $U_{{\bf p}^+}$
if $\epsilon_0$ and $\vec{\epsilon}$ are sufficiently small.
\end{lem}
\begin{proof}
By definition
\begin{equation}\label{formfor911}
u_{\bf x}({\frak z}_{{\bf x},\ell+i})
= {\rm Pr}^{\rm map}({\bf x})(\frak w_{{\bf p},i}).
\end{equation}
Therefore ${\bf x} \mapsto u_{\bf x}({\frak z}_{{\bf x},\ell+i})$
is a $C^n$ map by Proposition \ref{prop94}.
It suffices to show that this map is transversal to $\mathcal N_{{\bf p},i}$.
\par
By taking $\epsilon_0$ and $\vec{\epsilon}$ small, it suffices to show the transversality 
for ${\bf p}^+ = {\bf p} \cup \vec{\frak w}_{\bf p}$.
Note that if $\vec{\frak w}'_{\bf p}$ is sufficiently close to 
$\vec{\frak w}_{\bf p}$ then ${\bf p} \cup \vec{\frak w}'_{\bf p} \in U_{{\bf p}^+}$.
In fact since $\overline\partial u_{\bf p}$ is not only an element of 
$E_{\bf p}({\bf p} \cup \vec{\frak w}_{\bf p})$ but also zero, the element ${\bf p} \cup \vec{\frak w}'_{\bf p}$ 
still satisfies the condition $\overline\partial u_{\bf p} \in E_{\bf p}({\bf p} \cup \vec{\frak w}'_{\bf p})$
after we move $\vec{\frak w}'_{\bf p}$.
Therefore Definition \ref{defn81} (2) implies that the map 
${\bf x} \mapsto u_{\bf x}({\frak z}_{{\bf x},\ell+i})$ 
is transversal to $\mathcal N_{{\bf p},i}$.
\end{proof}

\begin{lem}
$V_{\bf p}$ is invariant under the action of the group 
${\rm Aut}^+({\bf p})$.
\end{lem}
\begin{proof}
This is a consequence of Lemma \ref{lem9555} and 
Definition \ref{defn81} (3).
\end{proof}
The set $U_{\bf p}$ as in Definition \ref{defn72}
is an open neighborhood of ${\bf p}$ in $V_{\bf p}/{\rm Aut}({\bf p})$ 
by Lemma \ref{lem832}. Therefore it has a structure of $C^n$ orbifold.
We remark that the tangent space of $V_{\bf p}$ at ${\bf p}$ contains 
$(D_{u_{\bf p}}\overline{\partial})^{-1}(E_{\bf p}({\bf p}))/\frak{aut}(\Sigma_{\bf p},\vec z_{\bf p},\vec{\frak z}_{\bf p})$.
Therefore the second half of Condition \ref{conds55} implies that  $V_{\bf p}/{\rm Aut}({\bf p})$
is an effective orbifold.
We thus have proved the $C^n$ version of the first statement of Lemma \ref{lem73}.
\par
We next study the bundles.
On 
$L^2_m(\Sigma_{{\bf p}}(\vec{\epsilon}),
\partial \Sigma_{{\bf p}}(\vec{\epsilon});X,L)$
there exists a bundle of $C^{n}$ class 
whose fiber at $h$ is
$
L^2_m(\Sigma_{{\bf p}}(\vec{\epsilon});h^*TX \otimes \Lambda^{01}).
$
We pull it back to $V_{\bf p}$ by ${\rm Pr}^{\rm map}$.
Then the bundle whose fiber at ${\bf x} \in V_{\bf p} \subset U_{{\bf p}^+}$
is $E_{\bf p}(\frak{forget}_{\ell'+\ell,\ell}({\bf x}))$ is its $C^n$ subbundle by 
Definition \ref{defn86}. Let $\tilde {\mathcal E}_{\bf p}$ be this subbundle.
We can show that the ${\rm Aut}^+({\bf p})$ action on $V_{\bf p}$ lifts to a $C^n$ action
on $\tilde{\mathcal E}_{\bf p}$ by Lemma \ref{conds55lem}.
We thus obtain a required $C^n$ (orbi)bundle 
${\mathcal E}_{\bf p} = \tilde{\mathcal E}_{\bf p}/{\rm Aut}({\bf p})$.
\par
It is easy to check that ${\bf x} 
\mapsto \frak s({\bf x}) = \overline{\partial} u_{\bf x} \in E_{\bf p}({\bf x})$
is a $C^n$ section.
We have thus proved the $C^n$ version of Lemma \ref{lem73}.
\qed

\section{Coordinate change is of $C^n$ class}
\label{sec:changesmoo}

\subsection{The main technical result}
\label{subsec:maintech}

\begin{shitu}\label{soti1011}
Let ${\bf p} \in \mathcal M_{k+1,\ell}(X,L;\beta)$.
We take its strong stabilization data
$(\frak W_{\bf p},\vec{\mathcal N}_{\bf p})$, where 
$\frak W_{\bf p} = (\vec{\frak w}_{\bf p},\{\varphi^{{\rm s},{\bf p}}_{a,i}\},
\{\varphi^{{\rm d},{\bf p}}_{a,j}\},\{\phi^{\bf p}_a\})$ are 
stabilization and trivialization data.
$\vec{\frak w}_{\bf p}$ consist of $\ell'$ additional marked points and so
${\bf p} \cup \vec{\frak w}_{\bf p} \in \mathcal M_{k+1,\ell+\ell'}(X,L;\beta)$.
Suppose ${\bf q} \in \mathcal M_{k+1,\ell}(X,L;\beta)$ is $\epsilon$-close to 
${\bf p}$ and
take its stabilization and trivialization data 
$\frak W_{\bf q} = (\vec{\frak w}_{\bf q},\{\varphi^{{\rm s},{\bf q}}_{a,i}\},
\{\varphi^{{\rm d},{\bf q}}_{a,j}\},\{\phi^{\bf q}_a\})$.
$\vec{\frak w}_{\bf q}$ consist of $\ell''$ additional marked points and so
${\bf q} \cup \vec{\frak w}_{\bf q} \in \mathcal M_{k+1,\ell+\ell''}(X,L;\beta)$.
$\diamond$
\end{shitu}
By the definition of $\epsilon$-closeness there exist 
$\ell'$ additional marked points ${}_{\bf p}\vec{\frak w}_{\bf q}$ on 
$
\Sigma_{\bf q}$ and 
$
\frak q \in \prod_{a\in \mathcal A_{\bf p}^{\rm s} \cup \mathcal A_{{\bf p}^+}^{\rm d}} \mathcal V^{{\bf p}^+}_a 
\times [0,c)^{m_{\rm d}} \times (D^2_{\circ}(c))^{m_{\rm s}}
$ such that
$$
(\Sigma_{\bf q},\vec z_{\bf q},\vec{\frak z}_{\bf q}\cup {}_{\bf p}\vec{\frak w}_{\bf q})
= \Phi_{{\bf p}^+}({\frak q}).
$$
Here $\Phi_{{\bf p}^+}$ is the map $\Phi$ in (\ref{form33}).
(Here we apply Lemma \ref{lem34} to ${\bf p}^+ = {\bf p} \cup \vec{\frak w}_{\bf p}$
in place of ${\bf p}$ there to obtain the map $\Phi_{{\bf p}^+}$.)
By Lemma \ref{lem832} we may assume
$$
u_{\bf q}({}_{\bf p}{\frak w}_{{\bf q},i}) \in \mathcal N_{{\bf p},i}
$$
in addition.
By Lemma \ref{lem3838} we obtain a smooth embedding
$$
\widehat{\Phi}_{{\bf p}^+;\frak q,\vec{\epsilon}} : \Sigma_{{\bf p}^+}(\vec{\epsilon}) \to 
\Sigma_{\bf q}
$$
whose image is by definition $\Sigma_{\bf q}(\vec{\epsilon})$.
Here and hereafter we include ${\bf p}^+$ in the notation $\widehat{\Phi}_{{\bf p}^+;\frak q,\vec{\epsilon}}$.
We do so in order to distinguish (\ref{phiqqq}) from (\ref{Phippp}) for example. 
\par
We take $\epsilon'$, $\vec{\epsilon'}$ sufficiently small 
compared to $\epsilon$ and $\vec{\epsilon}$.
Let ${\bf x} \in \mathcal X_{k+1,\ell}(X,L;\beta)$.
Suppose ${\bf x} \cup {}_{\bf q}\vec{\frak w}_{{\bf x}} \in \mathcal X_{k+1,\ell+\ell''}(X,L;\beta)$ 
is $\epsilon'$-close to 
${\bf q}^+ = {\bf q} \cup \vec{\frak w}_{\bf q}$.  
Then there exists 
${\frak x}_{\bf q}$ such that
$
(\Sigma_{{\bf x}},\vec z_{{\bf x}},\vec{\frak z}_{{\bf x} }\cup {}_{\bf q}\vec{\frak w}_{{\bf x}})
= \Phi_{{\bf q}^+}({\frak x}_{\bf q}).
$
Again by Lemma \ref{lem3838} we obtain a smooth embedding
\begin{equation}\label{phiqqq}
\widehat{\Phi}_{{\bf q}^+;{\frak x}_{\bf q},\vec{\epsilon'}} : \Sigma_{{\bf q}^+}(\vec{\epsilon'}) \to 
\Sigma_{{\bf x}}
\end{equation}
whose image is $\Sigma_{\bf x}(\vec{\epsilon'})$ by definition.
\par
By Lemma \ref{lem832} there exists a unique $\ell'$-tuple of additional 
marked points ${}_{\bf p}\vec{\frak w}_{\bf x}$ on $\Sigma_{{\bf x}}$ 
\index{00WPX@${}_{\bf p}\vec{\frak w}_{\bf x}$}
such that:
\begin{conds}\label{conds102}
\begin{enumerate}
\item 
${\bf x} \cup {}_{\bf p}\vec{\frak w}_{{\bf x}}$ is $(o(\epsilon) + o(\epsilon'))$-close to
${\bf p} \cup \vec{\frak w}_{\bf p}$.
\item
$u_{\bf x}({}_{\bf p}{\frak w}_{{\bf x},i}) \in \mathcal N_{{\bf p},i}$.
\item
$\widehat{\Phi}_{{\bf q}^+;\frak x_{\bf q}\vec{\epsilon'}}({}_{\bf p}{\frak w}_{{\bf q},i})$
is $o(\epsilon')$-close to ${}_{\bf p}{\frak w}_{{\bf x},i}$.
\end{enumerate}
\end{conds}
In fact the existence of ${}_{\bf p}\vec{\frak w}_{\bf x}$ 
satisfying Condition \ref{conds102} (1)(2) directly follows from Lemma \ref{lem832}.
Such ${\bf x} \cup {}_{\bf p}\vec{\frak w}_{\bf x}$ is unique up to the ${\rm Aut}({\bf p})$ action.
With Condition \ref{conds102} (3) in addition it becomes unique.

\begin{figure}[h]
\centering
\includegraphics[scale=0.9]{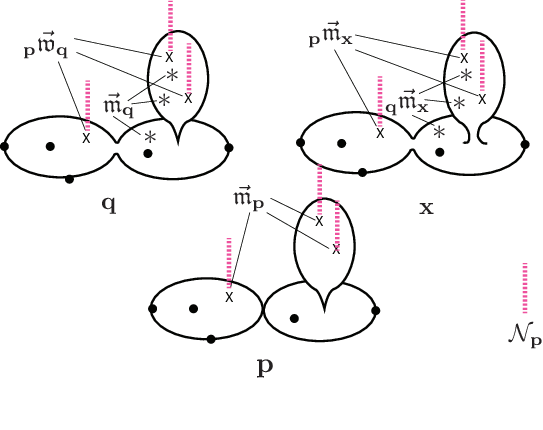}
\caption{${\bf p}$,${\bf q}$ and ${\bf x}$}
\label{zu3}
\end{figure}

Now by Lemmas \ref{lem34} and \ref{lem3838} we obtain
$
\frak x_{\bf p} \in \prod_{a\in \mathcal A_{\bf p}^{\rm s} \cup \mathcal A_{{\bf p}}^{\rm d}} \mathcal V^{{\bf p}^+}_a 
\times [0,c)^{m_{\rm d}} \times (D^2_{\circ}(c))^{m_{\rm s}}
$
with 
$
{\bf x} \cup {}_{\bf p}\vec{\frak w}_{{\bf x}}
= \Phi_{{\bf p}^+}(\frak x_{\bf p})
$
and a smooth embedding
\begin{equation}\label{Phippp}
\widehat{\Phi}_{{\bf p}^+,\frak x_{\bf p},\vec{\epsilon}} : \Sigma_{{\bf p}^+}(\vec{\epsilon}) \to 
\Sigma_{\bf x}
\end{equation}
whose image is by definition $\Sigma_{\bf x}(\vec{\epsilon})$.
\par
Since $\vec{\epsilon'}$ is sufficiently small compared to $\vec{\epsilon}$ we have
\begin{equation}\label{10333}
\Sigma_{\bf x}(\vec{\epsilon}) \subset \Sigma_{\bf x}(\vec{\epsilon'}).
\end{equation}
(The right hand side is the image of $\Sigma_{{\bf q}^+}(\vec{\epsilon'})$ by (\ref{phiqqq})
and the left hand side is the image of $\Sigma_{{\bf p}^+}(\vec{\epsilon})$ by (\ref{Phippp}).)
Now we define a map
\begin{equation}\label{10111}
\Psi_{{\bf p},{\bf q};{\bf x}\cup {}_{\bf q}\vec{\frak w}_{{\bf x}}}
= 
\widehat{\Phi}_{{\bf q}^+,{\frak x}_{\bf q},\vec{\epsilon'}}^{-1} \circ 
\widehat{\Phi}_{{\bf p}^+,\frak x_{\bf p},\vec{\epsilon}}
: \Sigma_{{\bf p}^+}(\vec{\epsilon}) \to 
\Sigma_{{\bf q}^+}(\vec{\epsilon'}).
\end{equation}
This is a family of smooth open embeddings parametrized by 
${\bf x}\cup {}_{\bf q}\vec{\frak w}_{{\bf x}}$, with the domain and target independent of ${\bf x}\cup {}_{{\bf q}}\vec{\frak w}_{{\bf x}}$.
Proposition \ref{prop103} below claims that it is a $C^n$ family if $m$ is sufficiently larger 
than $n$. To precisely state it  we need to choose a 
coordinate of the set of the objects ${\bf x}\cup {}_{\bf q}\vec{\frak w}_{{\bf x}}$.
The way to do so is similar to Definition \ref{defn86} and the paragraph thereafter. The detail follows.
\par
We take the direct product (Compare with (\ref{form86}).)
\begin{equation}\label{form862}
\mathcal V({{\bf q}^+};\vec{\epsilon'}) = 
\prod_{a\in \mathcal A_{{\bf q}}^{\rm s} \cup \mathcal A_{{\bf q}}^{\rm d}}\mathcal V_a^{{\bf q}^+}
\times
\prod_{j=1}^{m_{\rm d}} [0,{\epsilon_{j}^{\prime \rm d}})
\times
\prod_{i=1}^{m_{\rm s}} D^2_{\circ}({\epsilon_{i}^{\prime \rm s}}).
\end{equation}
and consider the map
$
\Phi_{{\bf q}^+} : 
\mathcal V({{\bf q}^+};\vec{\epsilon'})\to \mathcal M^{\rm d}_{k+1,\ell+\ell''},
$
where $\mathcal V^{{\bf q}^+}_a$ is the deformation space of an irreducible component  
of ${\bf q}^+ ={\bf q} \cup \vec{\frak w}_{\bf q}$.
(This is the map (\ref{form8585}) by taking ${\bf q}^+$ in place of ${\bf p}$ in (\ref{form8585}).
$\mathcal V^{+}_a$ is $\mathcal V^{{\bf q}^+}_a$ here.)
We denote its image by
\begin{equation}
\overline{\mathcal V}({{\bf q}^+};\vec{\epsilon'})
= 
\Phi_{{\bf q}^+}(\mathcal V({{\bf q}^+};\vec{\epsilon'}))
\subset \mathcal M^{\rm d}_{k+1,\ell+\ell''}.
\end{equation}
\index{00V3qepsilonVV @$\overline{\mathcal V}({{\bf q}^+};\vec{\epsilon'})$}
This is a neighborhood of the source curve of ${\bf q}^+$ in $\mathcal M^{\rm d}_{k+1,\ell+\ell''}$.

Let ${\frak a} = (\Sigma_{\frak a},\vec z_{\frak a},\vec{\frak z}_{\frak a}) 
\in \overline{\mathcal V}({{\bf q}^+};\vec{\epsilon'})$.
\index{00A5@${\frak a}$}
We put $(\Sigma_{\frak a},\vec z_{\frak a},\vec{\frak z}_{\frak a}) = \Phi_{{\bf q}^+}(\frak y)$
and  obtain a diffeomorphism
$
\widehat{\Phi}_{{\bf q}^+,\frak a,\vec\epsilon} = 
\widehat{\Phi}_{{\bf q}^+,\frak y,\vec\epsilon}: \Sigma_{{\bf q}^+}(\vec{\epsilon'}) \to \Sigma_{\frak a}(\vec{\epsilon'}),
$
by Lemma \ref{lem3838}.
Let $\mathscr L^{\bf q}_m$ be a small neighborhood of $u_{{\bf q}^+}\vert_{\Sigma_{{\bf q}^+}(\vec{\epsilon'})}$
in the space $L^2_{m}((\Sigma_{{\bf q}^+}(\vec{\epsilon'}),\partial\Sigma_{{\bf q}^+}(\vec{\epsilon'}));X,L)$.
\par
For $({\frak a},u') \in \overline{\mathcal V}({\bf q}^+;\vec{\epsilon'}) \times \mathscr L^{\bf q}_m$ as above,
we consider 
\begin{equation}\label{form1040400}
u'' = u' \circ \widehat{\Phi}_{{\bf q}^+,\frak a,\vec{\epsilon'}}^{-1} : \Sigma_{\frak a}(\vec{\epsilon'}) \to X.
\end{equation}
We can extend $u''$ to $\Sigma_{\frak a}$ (by modifying it near the small 
neighborhood of the boundary of $\Sigma_{\frak a}(\vec{\epsilon'})$)
so that $u''\vert_{\Sigma_{\frak a} \setminus \Sigma_{\frak a}(\vec{\epsilon'})}$ has diameter $<\epsilon_0$.
See footnote \ref{fn16}.
\par
Put ${\bf x} \cup \vec{\frak w}_{\bf x} = ({\frak a},u'')$
and  consider
$
\Psi_{{\bf p},{\bf q};{\bf x}\cup {}_{\bf q}\vec{\frak w}_{{\bf x}}},
$
as in (\ref{10111}).
We remark that during the construction of (\ref{10111}) the map $u_{\bf x}$ 
is used only to determine ${}_{\bf p}\vec{\frak w}_{\bf x}$ by requiring Condition \ref{conds102}
 (2).
Therefore the way to extend $u''$ to the neck region does not affect 
$
\Psi_{{\bf p},{\bf q};{\bf x}\cup {}_{\bf q}\vec{\frak w}_{{\bf x}}}
$.
\begin{defn}\label{defn102}
We define 
\index{00YPsi_{{\bf p},{\bf q}}@$\Psi_{{\bf p},{\bf q}}$}
$
\Psi_{{\bf p},{\bf q}} : \overline{\mathcal V}({{\bf q}^+};\vec{\epsilon'}) \times \mathscr L^{\bf q}_m 
\times \Sigma_{{\bf p}^+}(\vec{\epsilon}) \to 
\Sigma_{{\bf q}^+}(\vec{\epsilon'})
$
by
$$
\Psi_{{\bf p},{\bf q}}(\frak a,u',z) = 
\Psi_{{\bf p},{\bf q};{\bf x}\cup {}_{\bf q}\vec{\frak w}_{{\bf x}}}(z).
$$
\end{defn}
\begin{prop}\label{prop103}
If $m$ is sufficiently larger than $n$ then $\Psi_{{\bf p},{\bf q}}$ is a $C^n$ map.
In addition, it is $C^{\infty}$ in the direction of $\Sigma_{{\mathbf p}^+} (\vec{\epsilon})$.
\end{prop}
We will prove this proposition in Subsection \ref{subsec:techlemproof}.
Proposition \ref{prop103} is used in Subsection \ref{subsec:smoproof} to show the $C^n$ version of Lemma \ref{lem7878}.
We also use it in Section \ref{sec:exiobst} to prove the existence of obstruction bundle data.
We also use Lemma \ref{lem106}.
\begin{defn}\label{defn105}
We define
$
\Xi_{{\bf p},{\bf q}} : \overline{\mathcal V}({{\bf q}^+};\vec{\epsilon'}) \times \mathscr L^{\bf q}_m \to 
\mathcal M^{\rm d}_{k+1,\ell+\ell'}
$
\index{00Xipq@$\Xi_{{\bf p},{\bf q}}$}
by
$$
\Xi_{{\bf p},{\bf q}}(\frak a,u') = 
(\Sigma_{\bf x},\vec z_{\bf x},\vec{\frak z}_{\bf x} \cup {}_{\bf p}\vec{\frak w}_{\bf x}).
$$
Here ${\bf x} \cup {}_{\bf q}\vec{\frak w}_{\bf x} = (\frak a,u'')$,
$u''$ is the extension of (\ref{form1040400}) as above,
and ${}_{\bf p}\vec{\frak w}_{\bf x}$ is determined by Condition \ref{conds102} from
${\bf x} \cup {}_{\bf q}\vec{\frak w}_{\bf x}$.
\end{defn}
\begin{lem}\label{lem106}
If $m$ is sufficiently larger than $n$ then $\Xi_{{\bf p},{\bf q}}$ is $C^n$ map.
\end{lem}
Lemma \ref{lem106} is proved in Subsection \ref{subsec:techlemproof}.
\par
\begin{rem}
Note we use the smooth structure on $\overline{\mathcal V}({\bf q}^+;\vec{\epsilon'}) \cong \mathcal V({\bf q}^+;\vec{\epsilon'})$
whose coordinates of 
gluing parameters are $r_j \in [0,{\epsilon_{j}^{\prime \rm d}})$
and $\sigma_i \in D^2_{\circ}({\epsilon_{i}^{\prime \rm s}})$.
We use $s_j$ and $\rho_i$ as in (\ref{form92})
to define an alternative smooth structure 
on $\overline{\mathcal V}({\bf q}^+;\vec{\epsilon'})$,
and write it as $\overline{\mathcal V}({\bf q}^+;\vec{\epsilon'})^{\log}$.
We remark that Proposition \ref{prop103} 
and Lemma \ref{lem106} imply the 
same conclusion with $\overline{\mathcal V}({\bf q}^+;\vec{\epsilon'})$
replaced by $\overline{\mathcal V}({\bf q}^+;\vec{\epsilon'})^{\log}$ and 
$\mathcal M^{\rm d}_{k+1,\ell+\ell'}
$ by $\mathcal M^{\rm d,\log}_{k+1,\ell+\ell'}$.
As for Proposition \ref{prop103} this follows from the fact that the identity map
$\overline{\mathcal V}({{\bf q}^+};\vec{\epsilon'})^{\log}
\to \overline{\mathcal V}({\bf q}^+;\vec{\epsilon'})$ is smooth.
As for Lemma \ref{lem106} the proof that the `$\log$' version follows 
from the original version is similar 
to the proof of Lemma \ref{lem91} using a formula similar to (\ref{form94}).
\end{rem}

\subsection{Proof of Proposition \ref{prop103}}
\label{subsec:techlemproof}

In this subsection we prove Proposition \ref{prop103} and 
Lemma \ref{lem106}.
We use the notation of Subsection \ref{subsec:maintech}.
In this subsection we use $\overline{\mathcal V}({{\bf q}^+};\vec{\epsilon'})$ but 
not $\overline{\mathcal V}({{\bf q}^+};\vec{\epsilon'})^{\log}$.
\par
We first define a map
$
\Xi_i :  \overline{\mathcal V}({{\bf q}^+};\vec{\epsilon'}) \times \mathscr L^{\bf q}_m 
\to \Sigma_{\bf q}
$
for $i=1,\dots,\ell'$ by the equality:
\begin{equation}
\Xi_i(\frak a,u')
= \widehat{\Phi}^{-1}_{{\bf q}^+,\frak a,\vec{\epsilon'}}\,({}_{\bf p}\frak w_{{\bf x},i}).
\end{equation}
Here ${}_{\bf p}\vec{\frak w}_{\bf x}$ is determined by Condition \ref{conds102} from
${\bf x} \cup {}_{\bf q}\vec{\frak w}_{\bf x} = (\frak a,u'')$.

\begin{lem}\label{lem11009}
$\Xi_i$ is a $C^n$ map. 
\end{lem}
\begin{proof}
Let $\frak U_i$ be a neighborhood of $\mathcal N_{{\bf p},i}$ in $X$ and 
$h=(h_1,h_2) : \frak U_i \to \R^2$  a smooth map such that
$h^{-1}(0) = \mathcal N_{{\bf p},i}$ and $dh_1$ is linearly independent to $dh_2$
on $\frak U_i$.
\par
Let $U_i$ be a neighborhood of ${}_{\bf p}\frak w_{{\bf q};i}$ in $\Sigma_{\bf q}$.
We define a map
$
\hat\Xi_i : \overline{\mathcal V}({{\bf q}^+};\vec{\epsilon'}) \times \mathscr L^{\bf q}_m 
\times U_i \to \R^2
$
by
$
\hat\Xi_i(\frak a,u',z) = h(u''(z))
$
where $u''$ is as in (\ref{form1040400}).
For each fixed $(\frak a,u')$ the element $0 \in \R^2$ is a regular value 
of the restriction of $\hat\Xi_i$ to $\{(\frak a,u')\} \times U_i$.
Moreover $\hat\Xi_i$ is a $C^n$ map if $m$ is sufficiently larger than $n$.
Lemma \ref{lem11009} then is a consequence of the implicit function theorem.
\end{proof}
We pull back the universal family 
$\pi : \mathcal C^{\rm d}_{k+1,\ell+\ell''} \to \mathcal M^{\rm d}_{k+1,\ell+\ell''}$, 
by the inclusion map 
$\overline{\mathcal V}({{\bf q}^+};\vec{\epsilon'})
\to \mathcal M^{\rm d}_{k+1,\ell+\ell''}$ 
and take a direct product with $\mathscr L^{\bf q}_m$.
We thus obtain
\begin{equation}\label{form109}
\mathcal C({\bf q}^+) \to \overline{\mathcal V}({{\bf q}^+};\vec{\epsilon'}) \times \mathscr L^{\bf q}_m.
\end{equation}
It comes with $k+1$ sections $\frak s^{\rm d}_{0},\dots,\frak s^{\rm d}_{k}$ 
corresponding to the $k+1$ boundary marked points and $\ell + \ell''$ sections $\frak s^{\rm s}_{1},\dots,\frak s^{\rm s}_{\ell+\ell''}$
corresponding to the $\ell +\ell''$ interior marked points.
\par
We consider sections $\xi_1,\dots,\xi_{\ell'}$ defined by $\Xi_i$ ($i=1,\dots,\ell'$).
Then (\ref{form109}) together with $k+1$ sections $\frak s^{\rm d}_{0},\dots,\frak s^{\rm d}_{k}$
and $\ell+\ell'$ sections $\frak s^{\rm s}_{1},\dots,\frak s^{\rm s}_{\ell},\xi_1,\dots,\xi_{\ell'}$
becomes a family of nodal disks with $k+1$ boundary marked points and $\ell+\ell'$ interior marked 
points. Therefore by the universality of $\pi : \mathcal C^{\rm d}_{k+1,\ell+\ell'} \to \mathcal M^{\rm d}_{k+1,\ell+\ell'}$
we obtain the next commutative diagram.
\begin{equation}\label{diagramunivn}
\begin{CD}
\mathcal C({\bf q}^+)
@ >{\hat F}>>
\mathcal C^{\rm d}_{k+1,\ell+\ell'} \\
@ VV{(\ref{form109})}V @VV{\pi}V 
\\
\overline{\mathcal V}({{\bf q}^+};\vec{\epsilon'}) \times \mathscr L^{\bf q}_m @>{F}>>
\mathcal M^{\rm d}_{k+1,\ell+\ell'}
\end{CD}
\end{equation}
Here the horizontal arrows $\hat F$, $F$ are $C^n$ maps  satisfying:
$\hat F \circ \frak s_j^{\rm d} = \frak s_j^{\rm d} \circ F$ for $j=0,\dots,k$:
$\hat F \circ \frak s_i^{\rm s} = \frak s_i^{\rm s} \circ F$ for $i=1,\dots,\ell$:
$\hat F \circ \xi_i = \frak s_{\ell + i}^{\rm s} \circ F$ for $i=1,\dots,\ell'$:
(\ref{diagramunivn}) is a Cartesian square:
$\hat F$ is fiber-wise holomorphic.
\par
We can prove the existence of such $\hat F$ and $F$ by taking the double and 
using the corresponding universality statement of the Deligne-Mumford moduli space
of marked spheres.
\begin{rem}
In our genus $0$ case, we can use cross ratio to give an 
elementary proof of the fact that $\hat F$, $F$ are $C^n$ maps.
 A similar facts can be proved for the case of arbitrary 
genus.
\end{rem}

\begin{proof}[Proof of Lemma \ref{lem106}]
$\Xi_{{\bf p},{\bf q}}$ is nothing but the map $F$ in Diagram (\ref{diagramunivn}).
\end{proof}
\begin{proof}[Proof of Proposition \ref{prop103}]
We consider the next diagram.
\begin{equation}\label{dianew}
\nonumber
\begin{CD}
\overline{\mathcal V}({{\bf q}^+};\vec{\epsilon'}) \times \mathscr L^{\bf q}_m \times \Sigma_{{\bf q}^+}(\vec{\epsilon'})
@>{\Phi_1}>>
\mathcal C({\bf q}^+)
@ >{\hat F}>>
\mathcal C^{\rm d}_{k+1,\ell+\ell'}
@<{\Phi_2}<<
\mathcal V \times \Sigma_{{\bf p}^+}(\vec{\epsilon})
\end{CD}
\end{equation}
Here $\Phi_1$ is the map which is $\widehat{\Phi}_{{\bf q}^+,\frak a,\vec{\epsilon'}}$
(see (\ref{phiqqq}))
on the fiber of $(\frak a,u') \in \overline{\mathcal V}({{\bf q}^+};\vec{\epsilon'}) \times \mathscr L^{\bf q}_m$.
$\mathcal V$ is a small neighborhood of ${\bf p}^+ = {\bf p} \cup \vec{\frak w}_{\bf p}$ 
in $\mathcal M_{k+1,\ell+\ell'}^{\rm d}$.
$\Phi_2$ is the map which is $\widehat{\Phi}_{{\bf p}^+,\frak x_{\bf p},\vec{\epsilon}}$
(see (\ref{Phippp})) on the fiber of $\Phi_{{\bf p}^+}(\frak x_{\bf p})$.
\par
By (\ref{10333}), we choose $\vec{\epsilon '}$ sufficiently small compared to $\vec{\epsilon}$ and 
take fiberwise inverse to $\hat F \circ \Phi_1$ on the image of $\Phi_2$ and 
compose it fiberwise with $\Phi_2$ to obtain a
$C^n$ map, that is,
$$
{\mathcal V}({\mathbf q}^+; \vec{\epsilon '}) \times {\mathscr L}^{\mathbf q}_m \times \Sigma_{{\mathbf p}^+}(\vec{\epsilon}) \to \Sigma_{{\mathbf q}^+}(\vec{\epsilon'}).
$$
This is nothing but the map $\Psi_{{\bf p},{\bf q}}$ in Definition \ref{defn102}.
\end{proof}

\subsection{Proof of the fact that coordinate change is of $C^n$ class}
\label{subsec:smoproof}

In this subsection we will prove the $C^n$ version of Lemma \ref{lem7878} using Proposition \ref{prop103}.

For given $n$, let  $m_1,m_2$  both be sufficiently large 
and $\vec{\epsilon}$,$\epsilon$ so small that we obtain Kuranishi charts of $C^n$ class by
the argument 
of Subsections \ref{subsec:glue}, \ref{subsec:Cnkurastru}.

\begin{lem}\label{lem108}
The $C^n$ structures on a neighborhood of ${\bf p}$ in $V_{\bf p}$ obtained 
in Subsections \ref{subsec:glue}, \ref{subsec:Cnkurastru} 
using $L^2_m$ space with $m=m_1$ coincides with one using $L^2_m$ space with $m=m_2$.
\end{lem} 
\begin{proof}
We first observe that the solution of (\ref{form9595}) is automatically 
of $C^{\infty}$ class.
This is a consequence of standard bootstrapping argument.
($u' \in L^2_m$ implies $E_{{\bf p}}(\frak a,u')$ consists of $L^2_m$ sections 
and so by (\ref{form9595}) $u' \in L^2_{m+1}$.)
Let $m_1 > m_2$. Then 
a finite dimensional $C^n$ submanifold of 
$\mathcal M^{{\rm d},\log}_{k+1,\ell+\ell'} 
\times L^2_{m_1}(\Sigma_{{\bf p}^+}(\vec{\epsilon}),
\partial \Sigma_{{\bf p}^+}(\vec{\epsilon});X,L)$
becomes a  $C^n$ submanifold of 
$\mathcal M^{{\rm d},\log}_{k+1,\ell+\ell'} 
\times L^2_{m_2}(\Sigma_{{\bf p}^+}(\vec{\epsilon}),
\partial \Sigma_{{\bf p}^+}(\vec{\epsilon});X,L)$
by the obvious embedding.
Therefore the two $C^n$ structures of $U_{{\bf p}^+}$ coincide.
$V_{\bf p}$ is a $C^n$ submanifold of $U_{{\bf p}^+}$ and so its two $C^n$ structures
coincide.
\end{proof}
\begin{lem}\label{lem109}
Let $M_1,M_2,X$ be finite dimensional $C^{\infty}$
manifolds and $K_1 \subset M_1$,$K_2 \subset M_2$
relatively compact open subsets and let $\mathbb A$ be a Hilbert manifold.
Let 
$
\psi : \mathbb A \times M_1 \to M_2
$
be a $C^{m_1}$ map and $C^{\infty}$ in the direction of $M_1$. 
Suppose $m_1,m_2-m_1$ are sufficiently large compared to $n$.
We assume:
\begin{enumerate}
\item
For each $\frak a \in  \mathbb A$, 
$x \mapsto \psi(\frak a,x)$ is an open embedding $M_1 \to M_2$.
\item
$\psi(\mathbb A \times K_1) \subset K_2$.
\end{enumerate}
Then 
the map
$
\psi_* : \mathbb A \times L^2_{m_2}(K_2,X) \to L^2_{m_1}(K_1,X)
$
defined by 
$$
(\psi_*(\frak a,h))(z) = h(\psi(\frak a,z))
$$
induces a $C^n$-map ${\mathbb A} \to C^{\infty}(L^2_{m_2}(K_2,X),L^2_{m_1}(K_1,X))$.
\end{lem}
\begin{proof}
This is easy and standard. We omit the proof.
\end{proof}
Suppose we are in Situation \ref{soti1011}.
We use notations in Subsection \ref{subsec:maintech}.
We consider the next diagram.
\begin{equation}\label{diagramCn}
\begin{CD}
U_{{\bf q}^+}
@ > {({\rm Pr}_{\bf q}^{\rm source},{\rm Pr}_{\bf q}^{\rm map})} >>
\overline{\mathcal V}({{\bf q}^+};\vec{\epsilon'})^{\log}
\times L^2_{m_2}(\Sigma_{{\bf q}^+}(\vec{\epsilon'}),
\partial \Sigma_{{\bf q}^+}(\vec{\epsilon'});X,L)
\\
@ VV{\varphi}V @ VV{(\Xi_{{\bf p},{\bf q}},\psi)}V \\
U_{{\bf p}^+} @>{({\rm Pr}_{\bf p}^{\rm source},{\rm Pr}_{\bf p}^{\rm map})}>>
\mathcal M^{\rm d,log}_{k+1,\ell+\ell'}\times 
L^2_{m_1}(\Sigma_{{\bf p}^+}(\vec{\epsilon}),
\partial \Sigma_{{\bf p}^+}(\vec{\epsilon});X,L)
\end{CD}
\end{equation}
See Subsection \ref{subsec:glue} for the definition of the horizontal arrows.
(Note $\overline{\mathcal V}({{\bf q}^+};\vec{\epsilon'})^{\log}$ 
is an open neighborhood of ${\bf q}^+$ in $\mathcal M^{\rm d,log}_{k+1,\ell+\ell''}$.)
\par
The map $\psi$ in the right vertical arrow is defined by
\begin{equation}\label{newfor108}
\psi(\frak a,u')(z) = u'(\Psi_{{\bf p},{\bf q}}(\frak a,u',z)),
\end{equation}
where $\Psi_{{\bf p},{\bf q}}$ is defined by Definition \ref{defn102}.
The map $\Xi_{{\bf p},{\bf q}}$ in the right vertical arrow is as in Definition \ref{defn105}.
\par
The left vertical arrow ${\varphi}$ is defined by
\begin{equation}\label{defnphi}
{\bf x} \cup {}_{\bf q}\vec{\frak w}_{{\bf x}}\quad \mapsto \quad{\bf x} \cup  {}_{\bf p}\vec{\frak w}_{{\bf x}},
\end{equation}
where ${}_{\bf p}\vec{\frak w}_{{\bf x}}$ is determined by Condition \ref{conds102}.
The commutativity of the diagram is immediate from the definitions.
\par
The horizontal arrows are $C^n$ embeddings by Proposition \ref{prop94}.
(We use the smooth structure $\overline{\mathcal V}({{\bf q}^+};\vec{\epsilon'})^{\log}$
here.)
The right vertical arrow is a $C^n$ map by Proposition \ref{prop103}
and Lemmas \ref{lem106},\ref{lem108},\ref{lem109}.
Therefore ${\varphi}$ is a $C^n$ map.
\par
Now we take local transversals $\vec{\mathcal N}_{\bf q}$ such that 
$(\frak W_{\bf q},\vec{\mathcal N}_{\bf q})$ is a strong stabilization data.
We define $V_{\bf q} \subset U_{{\bf q}^+}$ by using it. (See Definition \ref{defn910}.)
Using Condition \ref{conds102} (2), which we required for ${}_{\bf p}\vec{\frak w}_{{\bf x}}$, the image of ${\varphi}$  is in 
$V_{\bf p} \subset U_{{\bf p}^+}$ 
and the restriction of $\varphi$ to $V_{\bf q}$ induces the map
$
\varphi_{{\bf p}{\bf q}} : V_{\bf q}/{\rm Aut}({\bf q}) \to V_{\bf p}/{\rm Aut}({\bf p}).
$
Therefore $\varphi_{{\bf p}{\bf q}}$ is a $C^n$ map.
Note $U_{\bf q} = V_{\bf q}/{\rm Aut}(\bf q)$, $U_{\bf p} = V_{\bf p}/{\rm Aut}(\bf p)$.
\begin{prop}\label{lem101010}
The $C^n$ map
$
{\varphi}_{{\bf p}{\bf q}} : U_{\bf q} \to U_{\bf p}
$
becomes a $C^n$ embedding if we take a smaller neighborhood  $U'_{\bf q}$ of $[{\bf q}]$ in $U_{\bf q}$ .
\end{prop}
\begin{proof}
The proof is divided into 5 steps.
In the first 3 steps we assume ${\bf p} = {\bf q}$.
We take two different choices of strong stabilization data $(\frak W^o_{\bf p},\vec{\mathcal N}_{\bf p}^o)$ ($o=1,2$), 
and obstruction bundle data $E^o_{\bf p}({\bf x})$
($o=1,2$) at ${\bf p}$ with $E^1_{\bf p}({\bf x}) \subseteq E^2_{\frak p}({\bf x})$.
We then obtain $U_{\bf p}^o$, $V_{\bf p}^o$ and $\varphi_{21} : U_{\bf p}^1 \to U_{\bf p}^2$.
We proved already that $\varphi_{21}$ is a $C^n$ map. We will prove that it is a $C^n$ embedding.
\par\smallskip
\noindent(Step 1): The case ${\bf p} = {\bf q}$, $(\frak W^1_{\bf p},\vec{\mathcal N}_{\bf p}^1) 
= (\frak W^2_{\bf p},\vec{\mathcal N}_{\bf p}^2)$.
It is easy to see that $\varphi_{21}$ is an embedding in this case.
\par\smallskip
\noindent(Step 2): The case ${\bf p} = {\bf q}$, $E^1_{\bf p}({\bf x})=E^2_{\bf p}({\bf x})$,
but $(\frak W^1_{\bf p},\vec{\mathcal N}_{\bf p}^1) 
\ne (\frak W^2_{\bf p},\vec{\mathcal N}_{\bf p}^2)$.
In this case we can exchange the role of $(\frak W^1_{\bf p},\vec{\mathcal N}_{\bf p}^1)$ 
and $(\frak W^2_{\bf p},\vec{\mathcal N}_{\bf p}^2)$ and obtain $\varphi_{12}$.
Then in the same way as the proof of Lemma \ref{lem79} we can show that 
$\varphi_{21} \circ \varphi_{12}$ and $\varphi_{12} \circ \varphi_{21}$ are identity maps.
Therefore they are $C^n$ diffeomorphisms.
\par\smallskip
\noindent(Step 3): The case ${\bf p} = {\bf q}$ in general. 
Note if we have three choices $o=1,2,3$ then we can show $\varphi_{32} \circ \varphi_{21} = \varphi_{31}$
in the same way as Lemma \ref{lem79}. Therefore combining Step 1 and Step 2 we can prove this case.
\par\smallskip
\noindent(Step 4): Suppose we are given a strong stabilization data $(\frak W_{\bf p},\vec{\mathcal N}_{\bf p})$ 
at ${\bf p}$. Let ${\bf q}$ be sufficiently 
close to ${\bf p}$. We also assume that we are given obstruction bundle data $E_{\bf p}({\bf x})$ and 
$E_{\bf q}({\bf x})$ at ${\bf p}$ and ${\bf q}$ respectively, such that $E_{\bf q}({\bf x}) = E_{\bf p}({\bf x})$
when both sides are defined. We will prove that there {\it exist} 
strong stabilization data $(\frak W_{\bf q},\vec{\mathcal N}_{\bf q})$ at ${\bf q}$
such that the map $\varphi_{{\bf p}{\bf q}}$ is an open embedding.
\par
The proof is based on the next lemma.
We take weak stabilization data $\vec{\frak w}_{\bf q} = ({\frak w}_{{\bf q},i})$ at ${\bf q}$ such that 
Condition \ref{conds102} (1)(2) with ${\bf x}$ replaced by ${\bf q}$ is satisfied.
Note $\ell' = \#\vec {\frak w}_{\bf p} = \#\vec {\frak w}_{\bf q} = \ell''$ in our case.

\begin{lem}\label{lem1244}
We can choose 
$\{\varphi_{{\bf q},a,i}^{\rm s}\}, \{\varphi_{{\bf q},a,j}^{\rm d}\},
\{\phi_{{\bf q},a}\}$ 
so that the next diagram commutes.
\begin{equation}\label{diagram12-2}
\begin{CD}
U'_{{\bf q}^+}
@ >{({\rm Pr}^{\rm source},{\rm Pr}^{\rm map}_{\bf q})}>>
\!\!\!\!\!\!\!\!\!\!\!\!\!\!\!\!\!\!\!\!\!\!\!\!\!\!\!\!\!\!\!\!\!\!\!\!\!\!\!\!\!\!\!\!\!\!\!\!\!\!\!\!\!\!\!\!\!\!\!\!\!\!\!\!\!
\mathcal M^{\rm d,log}_{k+1,\ell+\ell'}\atop
\times L^2_{m'}(\Sigma_{{\bf q} \cup \vec {\frak w}_{\bf q}}(\vec{\epsilon'}),
\partial \Sigma_{{\bf q} \cup \vec {\frak w}_{\bf q}}(\vec{\epsilon'});X,L)
  \\
@ VVV @VVV 
\\
U_{{\bf p}^+}
@ >{({\rm Pr}^{\rm source},{\rm Pr}^{\rm map}_{\bf p})}>>
\!\!\!\!\!\!\!\!\!\!\!\!\!\!\!\!\!\!\!\!\!\!\!\!\!\!\!\!\!\!\!\!\!\!\!\!\!\!\!\!\!\!\!\!\!\!\!\!\!\!\!\!\!\!\!\!\!\!\!\!\!\!\!\!\!
\mathcal M^{\rm d,log}_{k+1,\ell+\ell'}
\atop \times \ L^2_m(\Sigma_{{\bf p} \cup \vec {\frak w}_{\bf p}}(\vec{\epsilon}),
\partial \Sigma_{{\bf p} \cup \vec {\frak w}_{\bf p}}(\vec{\epsilon});X,L)
\end{CD}
\end{equation}
The left vertical arrow is the inclusion map. 
\par
There exists a 
smooth embedding 
$\phi : \Sigma_{{\bf p} \cup \vec {\frak w}_{\bf p}}(\vec{\epsilon}) \to 
\Sigma_{{\bf q} \cup \vec {\frak w}_{\bf q}}(\vec{\epsilon'})$ 
such that
$$
(\frak a,u') \mapsto (\frak a,u'\circ \phi)
$$
is the map in the right vertical arrow.
($\frak a \in \mathcal M^{\rm d,log}_{k+1,\ell+\ell'}$.)
\par
The number  $m'$ can
be arbitrary large.
Here $U'_{{\bf q}^+}$ is a neighborhood in $U_{{\bf p}^+}$ of ${\bf q}^+$ 
which depends on $m'$.  $\vec{\epsilon'}$ depends on $m'$ also.\footnote{
The last part of lemma is not used here but will be used in Section \ref{sec:Cmugen}
.}
\end{lem}

Postponing the proof of the lemma until Subsection \ref{sub:kurasmstrlem} we continue the proof.
We take $\mathcal N_{{\bf q},i} = 
\mathcal N_{{\bf p},i}$.
Since $\vec{\frak w}_{\bf q}$ satisfies Condition \ref{conds102} (1)(2), this choice of 
$\mathcal N_{{\bf q},i}$ satisfies the conditions of 
Definition \ref{defn81}.
\par
The commutativity of Diagram (\ref{diagram12-2})
implies the next:
\begin{cor}\label{cor1014}
Consider Diagram (\ref{diagramCn}) and $V_{\bf q} \subset U_{{\bf q}^+}$.
If $(\frak a,u') \in ({\rm Pr}^{\rm source},{\rm Pr}_{\bf q}^{\rm map})(V_{{\bf q}})$
then 
$
\Xi_{{\bf p},{\bf q}}(\frak a,u') = \frak a$
and 
$\psi(\frak a,u') = u'\circ \phi$. 
Here $\psi$ is the map in the right vertical arrow of Diagram (\ref{diagramCn})
and $\phi$ is as in Lemma \ref{lem1244}.
\end{cor}

\begin{proof}
Let ${\bf x} \cup {}_{\bf q}\vec{\frak w}_{\bf x} \in V_{{\bf q}} \subset U_{{\bf q}^+}$.
Note by shrinking $U_{{\bf q}^+}$ we may assume $U_{{\bf q}^+} \subset U_{{\bf p}^+}$.
Then $\mathcal N_{{\bf q},i} = \mathcal N_{{\bf p},i}$ implies
${\bf x} \cup {}_{\bf q}\vec{\frak w}_{\bf x} \in V_{{\bf p}}$.
In other words ${}_{\bf q}\vec{\frak w}_{\bf x} = 
{}_{\bf p}\vec{\frak w}_{\bf x}$.
Therefore $\varphi({\bf x} \cup {}_{\bf q}\vec{\frak w}_{\bf x})
= {\bf x} \cup {}_{\bf q}\vec{\frak w}_{\bf x}$.
Here $\varphi$ is the map in the left vertical arrow of Diagram (\ref{diagramCn}).
(See (\ref{defnphi}).)
\par
Put 
$(\frak a,u') = ({\rm Pr}^{\rm source}{\rm Pr}_{\bf q}^{\rm map})({\bf x} \cup {}_{\bf q}\vec{\frak w}_{\bf x})$.
The commutativity of Diagram (\ref{diagramCn}) 
implies 
$({\rm Pr}_{\bf p}^{\rm source},{\rm Pr}_{\bf p}^{\rm map})({\bf x} \cup {}_{\bf q}\vec{\frak w}_{\bf x}) 
= (\Xi_{{\bf p},{\bf q}}(\frak a,u'),\psi(\frak a,u')).
$
On the other hand the commutativity of Diagram (\ref{diagram12-2}) 
implies 
$
({\rm Pr}_{\bf p}^{\rm source},{\rm Pr}_{\bf p}^{\rm map})({\bf x} \cup {}_{\bf q}\vec{\frak w}_{\bf x}) 
= 
(\frak a,u' \circ \phi).
$
The corollary follows.
\end{proof}
The injectivity of the differential of $(\Xi_{{\bf p},{\bf q}},\psi)$ 
on the tangent space of $V_{{\bf q}}$ is now a consequence of 
Corollary \ref{cor1014}
and unique continuation.
\par\smallskip
\noindent(Step 5):  The general case follows from Step 3 and Step 4 using Lemma \ref{lem79}.
\end{proof}
The proof of the fact that $\widehat{\varphi}_{{\bf p}{\bf q}}$ 
is a $C^n$ embedding of (orbi)bundles is similar.
Definition \ref{defKchart} (3)(4)(5) are clear from construction.

\subsection{Proof of Lemma \ref{lem1244}
}
\label{sub:kurasmstrlem}
To complete the proof of Proposition \ref{lem101010} and of $C^n$ analogue of Lemma \ref{lem7878}
it remains to prove Lemma \ref{lem1244}.
\begin{proof}[Proof of Lemma \ref{lem1244}]
Commutativity of the first factor (${\rm Pr}^{\rm source}$) is obvious.
The commutativity of the second factor (${\rm Pr}^{\rm map}$) is an issue.
Put 
\begin{equation}\label{form8622}
\aligned
&\mathcal V({{\bf p}^+};\vec{\epsilon}) = 
\prod_{a\in \mathcal A_{{\bf p}}^{\rm s} \cup \mathcal A_{{\bf p}}^{\rm d}}\mathcal V_a^{{\bf p}^+}
\times
\prod_{j=1}^{m_{\rm d}^{\bf p}} [0,{\epsilon_{j}^{\rm d}})
\times
\prod_{i=1}^{m_{\rm s}^{\bf p}} D^2_{\circ}({\epsilon_{i}^{\rm s}}), \\
&\mathcal V({{\bf q}^+};\vec{\epsilon'}) = 
\prod_{b\in \mathcal A_{{\bf q}}^{\rm s} \cup \mathcal A_{{\bf q}}^{\rm d}}\mathcal V_b^{{\bf q}^+}
\times
\prod_{j=1}^{m_{\rm d}^{\bf q}} [0,{\epsilon_{j}^{\prime \rm d}})
\times
\prod_{i=1}^{m_{\rm s}^{\bf q}} D^2_{\circ}({\epsilon_{i}^{\prime \rm s}}).
\endaligned
\end{equation}
An element ${\bf x} \cup \vec{\frak w}_{\bf x}$ in a neighborhood of 
its source curve (an element of $\mathcal M^{\rm d}_{k+1,\ell+\ell'}$) is written 
as 
$$
{\bf x} \cup \vec{\frak w}_{\bf x}= \Phi_{{\bf q}^+}(\frak x_{\bf q}) = \Phi_{{\bf p}^+}(\frak x_{\bf p}),
$$
where $\frak x_{\bf q} \in \mathcal V({{\bf q}^+};\vec{\epsilon'})$ 
and $\frak x_{\bf p} \in \mathcal V({{\bf p}^+};\vec{\epsilon})$.
Also there exists $\frak q_{\bf p} \in \mathcal V({{\bf p}^+};\vec{\epsilon})$
such that ${\bf q}^+ = \Phi_{{\bf p}^+}(\frak q_{\bf p})$.
$\Phi_{{\bf p}^+}$, $\Phi_{{\bf q}^+}$ are defined by Lemma \ref{lem34} using $\frak W_{\bf p}$,$\frak W_{\bf q}$,
respectively.
\begin{sublem}\label{sublem125}
We can choose $\{\varphi_{{\bf q},b,i}^{\rm s}\}, \{\varphi_{{\bf q},b,j}^{\rm d}\},
\{\phi_{{\bf q},b}\}$ so that the next diagram commutes.
\begin{equation}\label{diagram12-3}
\begin{CD}
\Sigma_{{\bf q}^+}(\vec{\epsilon'})
@ >{\hat{\Phi}_{{\bf q}^+,\frak x_{\bf q},\vec{\epsilon'}}}>>
\Sigma_{\bf x\cup \vec{\frak w}_{\bf x}}(\vec{\epsilon'}) \\
@ A{\hat{\Phi}_{{\bf p}^+,\frak q_{\bf p},\vec{\epsilon}}}AA 
@ AAA 
\\
\Sigma_{{\bf p}^+}(\vec{\epsilon})
@ >{\hat{\Phi}_{{\bf p}^+,\frak x_{\bf p},\vec{\epsilon}}}>>
\Sigma_{\bf x\cup \vec{\frak w}_{\bf x}}(\vec{\epsilon})
\end{CD}
\end{equation}
The right vertical arrow is an inclusion and other arrows are 
as in Lemma \ref{lem3838}.
\end{sublem}
It is easy to see from definition that Sublemma \ref{sublem125} implies 
Lemma \ref{lem1244}. In fact the smooth open embedding  
$\phi : \Sigma_{{\bf p} \cup \vec {\frak w}_{\bf p}}(\vec{\epsilon}) \to 
\Sigma_{{\bf q} \cup \vec {\frak w}_{\bf q}}(\vec{\epsilon'})$
mentioned in the statement of Lemma \ref{lem1244}
is the map $\hat{\Phi}_{{\bf p}^+,\frak q_{\bf p},\vec{\epsilon}}$ 
appearing in Diagram (\ref{diagram12-3}). 
\end{proof}
\begin{proof}[Proof of Sublemma \ref{sublem125}]
The proof is similar to the discussion of \cite[Section 23]{foootech}.
We first define $\varphi_{{\bf q},b,i}^{\rm s}$ and $\varphi_{{\bf q},b,i}^{\rm t}$, analytic families of coordinates 
at the nodal points of $\Sigma_{\bf q}$. An irreducible component $\Sigma_{{\bf q}^+}(b)$
of $\Sigma_{{\bf q}^+}$
is obtained by gluing several irreducible components 
$\{\Sigma_{{\bf p}^+}(a)\mid a \in \mathcal A(b)\}$ of $\Sigma_{{\bf p}^+}$.
Here $\mathcal A(b) \subset \mathcal A_{{\bf p}}^{\rm s} \cup \mathcal A_{{\bf p}}^{\rm d}$.
We may identify
\begin{equation}\label{form126}
\mathcal V_b^{{\bf q}^+} 
\subset \prod_{a\in \mathcal A(b)}\mathcal V_a^{{\bf p}^+}
\times \text{Some of the gluing parameters}.
\end{equation}
Here the second factor of the right hand side consists of the gluing parameters 
of the nodes ${\frak n}$ of $\Sigma_{{\bf p}^+}$ such that $\{\frak n\}
= \Sigma_{{\bf p}^+}(a) \cap \Sigma_{{\bf p}^+}(a')$ with $a,a' \in \mathcal A(b)$.
\begin{figure}[h]
\centering
\includegraphics[scale=0.7]{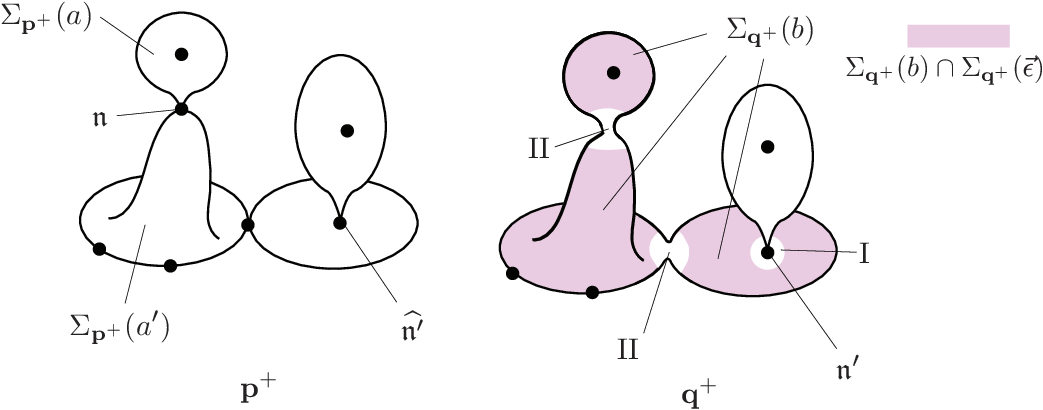}
\caption{${\bf p}^+$ and ${\bf q}^+$}
\label{zu2}
\end{figure}
We will define an analytic family of coordinates at a node $\frak n'$ in 
$\Sigma_{{\bf q}^+}(b)$. It is a family parametrized by 
$\mathcal V_b^{{\bf q}^+}$. There exists  $a \in \mathcal A(b)$
such that $\frak n'$ corresponds to a nodal point 
$\widehat{\frak n'}$ of $\Sigma_{{\bf p}^+}$ in $\Sigma_{{\bf p}^+}(a)$.
Then using $\varphi_{{\bf p},a,i}^{\rm s}$ or $\varphi_{{\bf p},a,i}^{\rm t}$ we can find a $\mathcal V_a^{{\bf p}^+}$ 
parametrized family of coordinates at this nodal point $\widehat{\frak n'}$.
We regard it as a $\mathcal V_b^{{\bf q}^+}$ parametrized family 
using the identification (\ref{form126}).
We thus obtain $\varphi_{{\bf q},b,i}^{\rm s}$ and $\varphi_{{\bf q},b,i}^{\rm t}$.
\par
We next define $\phi_{{\bf r},b}$.
This is a trivialization of the universal family of deformations of $\Sigma_{{\bf q}^+}(b)$
(equipped with marked or nodal points on it). The parameter space of this deformation is 
$\mathcal V_b^{{\bf r}^+}$.
In other words if $\Sigma_{\frak v}(b)$ together with marked points
is an object corresponding to $\frak v \in  \mathcal V_b^{{\bf q}^+}$,
the datum $\phi_{{\bf q},b}$ must be a diffeomorphism 
\begin{equation}\label{form127}
\Sigma_{{\bf q}^+}(b) \cong \Sigma_{\frak v}(b).
\end{equation}
\par
Note the data $\{\phi_{{\bf p},a}\mid a \in \mathcal A(b)\}$ 
and $\{\varphi_{{\bf p},a,i}^{\rm s}\}$, $\{\varphi_{{\bf p},a,i}^{\rm t}\}$
determine a 
smooth embedding 
$\Sigma_{{\bf q}^+}(b)\cap \Sigma_{{\bf q}^+}(\vec{\epsilon}) \to \Sigma_{\frak v}(b)$
uniquely such that Diagram (\ref{diagram12-3}) commutes there.
(This family of embeddings is parametrized by $\prod_{a\in \mathcal A(b)}\mathcal V_a^{{\bf p}^+}$.)
We extend the family to the required family  of diffeomorphisms (\ref{form127})
as follows.
\par
We remark that $\Sigma_{{\bf q}^+}(b) \setminus \Sigma_{{\bf q}^+}(\vec{\epsilon})$
is a union of the following two kinds of connected components.
(See Figure \ref{zu2}.)
\begin{enumerate}
\item[(I)] A neighborhood of a nodal point of $\Sigma_{{\bf q}^+}$ contained in $\Sigma_{{\bf q}^+}(b)$.
\item[(II)] A neck region in $\Sigma_{{\bf q}^+}(b)$. It corresponds to a nodal point of $\Sigma_{{\bf p}^+}$
which is resolved when we obtain $\Sigma_{{\bf q}^+}$ from $\Sigma_{{\bf p}^+}$.
\end{enumerate}
\par
To the part (I) we extend the embedding 
$\Sigma_{{\bf q}^+}(b)\cap \Sigma_{{\bf q}^+}(\vec{\epsilon}) \to \Sigma_{\frak v}(b)$ 
using the analytic families of coordinates $\varphi_{{\bf q},b,i}^{\rm s}$ or $\varphi_{{\bf q},b,i}^{\rm t}$
we produced above. In fact requiring Definition \ref{defn377} to be satisfied 
makes such an extension  unique.
\par
We extend it to the part (II) in an arbitrary way. 
We can prove the existence of such an extension by choosing $\mathcal V_b^{{\bf q}^+}$ small.
(The extension depends not only on the first factor but also on the second factor of (\ref{form126}).)
See \cite[Remark 23.5]{foootech} for example for detail.
\par
The commutativity of Diagram (\ref{diagram12-3}) is then immediate from construction.
In fact if 
$
\frak v_{\bf q} = 
\left(
(\frak v_b)_{b \in \mathcal A_{{\bf q}}^{\rm s} \cup \mathcal A_{{\bf q}}^{\rm d}};(0,\dots,0),
(0,\dots,0)
\right)
$
(namely if all the gluing parameters are $0$)
then this is the way how $\phi_{{\bf q},b}$ is 
chosen.
Then by the way how $\varphi_{{\bf q},b,i}^{\rm s}$ and $\varphi_{{\bf q},b,i}^{\rm t}$
are chosen Diagram (\ref{diagram12-3}) 
commutes  when gluing parameters are nonzero as well.
We thus proved Sublemma \ref{sublem125} and the $C^n$ analogue of Lemma \ref{lem7878}.
\end{proof}

\begin{rem}\label{lem1015}
In the proof of Lemma \ref{lem1244} and Sublemma \ref{sublem125} we never used the pseudo holomorphicity 
of $u_{\bf q}$. Therefore Lemma \ref{lem1244} and Sublemma \ref{sublem125} 
still hold if we replace ${\bf q}$
by ${\bf r} = ((\Sigma_{\bf r},\vec z_{\bf r},\vec{\frak z}_{\bf r}),u_{\bf r})$ such that
$u_{\bf r}$ is smooth and ${\bf r}$ is $\epsilon$-close to ${\bf p}$.
\end{rem}

\section{Existence of obstruction bundle data}
\label{sec:exiobst}

In this section we prove:
\begin{thm}\label{them111}
There exists an obstruction bundle data $\{E_{\bf p}(\bf x)\}$ of the 
moduli space ${\mathcal M}_{k+1,\ell}(X,L,J;\beta)$.
We may choose it so that Condition \ref{cond22} is satisfied.
\end{thm}

\subsection{Local construction of obstruction bundle data}
\label{sub:localconst}

Let ${\bf p} \in {\mathcal M}_{k+1,\ell}(X,L,J;\beta)$.
We will construct an obstruction bundle data $E_{\bf q;\bf p}(\bf x)$
at ${\bf q}$ when ${\bf q}$ is in a small neighborhood of ${\bf p}$.

\begin{lem}\label{lem1122}
There exists a finite dimensional subspace 
$E^0_{\bf p}(\bf p)$ of $C^{\infty}(\Sigma_{\bf p},u_{\bf p}^*TX \otimes \Lambda^{01})$ 
(the set of smooth sections)
such that:
\begin{enumerate}
\item The supports of elements of $E^0_{\bf p}(\bf p)$ are away from 
nodal points.
\item $E^0_{\bf p}(\bf p)$ satisfies the transversality condition 
in Definition \ref{defn54564}.
\item $E^0_{\bf p}(\bf p)$ is invariant under the ${\rm Aut}^+({\bf p})$ 
action in the sense of Condition \ref{conds55}.
\end{enumerate}
We may choose $E^0_{\bf p}(\bf p)$ so that it also satisfies Condition \ref{cond22}
and the second half of Condition \ref{conds55} holds.
\end{lem}
This is a standard consequence of the Fredholm property
of the operator (\ref{form523}) and unique continuation.
We omit the proof.
\par
We next take a strong stabilization data 
$(\frak W_{\bf p},\vec{\mathcal N}_{\bf p})$
as in Situation \ref{shitu82}.
\par
Let ${\bf x} \in {\mathcal X}_{k+1,\ell}(X,L,J;\beta)$
be $\epsilon$-close to ${\bf p}$.
Using Lemma \ref{lem832} and the definition of $\epsilon$-closeness, 
we can find ${}_{\bf p}\vec{\frak w}_{\bf x}$
such that 
${\bf x} \cup {}_{\bf p}\vec{\frak w}_{\bf x}$
is $o(\epsilon)$-close to 
${\bf p} \cup \vec{\frak w}_{\bf p}$
and 
$
u_{\bf x}({}_{\bf p}\vec{\frak w}_{{\bf x},i}) 
\in \mathcal N_{{\bf p},i}.
$
Moreover the choice of such ${}_{\bf p}\vec{\frak w}_{\bf x}$ 
is unique up to ${\rm Aut}({\bf x})$ action.
(We also remark that ${\rm Aut}({\bf x})$ is canonically embedded to ${\rm Aut}({\bf p})$.)
\par
Now we proceed in the same way as Subsection \ref{subsec:trivobst}.
We put ${\bf y} = {\bf x} \cup {}_{\bf p}\vec{\frak w}_{\bf x}$
and obtain $\frak y$ with 
$\Phi_{{\bf p}^+}(\frak y) = {\bf y}$.
We then obtain the map (\ref{form8000}).
\begin{equation}\label{form80002}
\mathcal P_{{\bf y}} : C^2(\Sigma_{\bf y}(\vec \epsilon);u_{{\bf y}}^*TX\otimes \Lambda^{01})
\to 
C^2(\Sigma_{\bf p};u_{{\bf p}}^*TX\otimes \Lambda^{01}).
\end{equation}
Note the image of $\mathcal P_{{\bf y}}$ is 
$C^2(\Sigma_{\bf p}(\vec \epsilon);u_{{\bf p}}^*TX\otimes \Lambda^{01})$.
We may take $\vec{\epsilon}$ so that $E^0_{\bf p}(\bf p)$ is contained 
in the image of (\ref{form80002}).
\begin{defn}\label{defn113}
We define 
\begin{equation}\label{defnofE}
E_{\bf p}^0({\bf x})
= 
\mathcal P_{{\bf y}}^{-1}(E^0_{\bf p}(\bf p)).
\end{equation}
Since the choice of ${}_{\bf p}\vec{\frak w}_{\bf x}$ is unique up to 
${\rm Aut}({\bf x}) \subseteq {\rm Aut}({\bf p})$ action, 
Lemma \ref{lem1122} (3) implies that the right hand side of 
(\ref{defnofE}) is independent of the choice of ${}_{\bf p}\vec{\frak w}_{\bf x}$.
\par
We also define
\begin{equation}\label{defnofE2}
E_{{\bf q};{\bf p}}({\bf x}) = E_{\bf p}^0({\bf x})
\end{equation}
if ${\bf q} \in {\mathcal M}_{k+1,\ell}(X,L,J;\beta)$ 
is sufficiently close to ${\bf p}$ and 
${\bf x}$ is $\epsilon$-close to ${\bf q}$.
\end{defn}

\begin{prop}\label{prop114}
If ${\bf q}$ is sufficiently close to ${\bf p}$
then $\{E_{{\bf q};{\bf p}}({\bf x})\}$ is an 
obstruction bundle data at ${\bf q}$.
\end{prop}
\begin{proof}
By Lemma \ref{lem413},
$E_{{\bf q};{\bf p}}({\bf x})$ is defined 
if ${\bf x}$ is $\delta$-close to ${\bf q}$
for some small $\delta$.
We will check Definition \ref{defn51} (1)(2)(3)(5).
(1) is obvious from definition.
(5) follows from Lemma \ref{lem1122} (3).
(3) holds if ${\bf q} = {\bf p}$ by Lemma \ref{lem1122} (2).
Then using Mrowka's Mayer-Vietoris principle 
it holds if ${\bf q}$ is sufficiently close to ${\bf p}$.
See \cite[Proposition 7.1.27]{fooobook2}.
We will prove (2) (smoothness of $E_{{\bf q};{\bf p}}({\bf x})$) in the next subsection.

\subsection{Smoothness of obstruction bundle data}
\label{sub:localsmooth}

In this subsection we prove that $E_{{\bf q};{\bf p}}({\bf x})$
(which is defined in Definition \ref{defn113})
is smooth in the sense of Definition \ref{defn8911}.
The proof is based on Proposition \ref{prop103}.
\par
Let ${\bf q} \in {\mathcal M}_{k+1,\ell}(X,L,J;\beta)$
and  take stabilization and trivialization data
$\frak W_{\bf q}$. In other words (together with 
the strong stabilization data at ${\bf p}$ which we have taken 
in the last subsection), we are in Situation \ref{soti1011}.
We use the notations of Subsection \ref{subsec:maintech}.
\par
We remark that the role of ${\bf p}$ in Definition \ref{defn8911}
is taken by ${\bf q}$ here.
\par
We take
$
\mathcal V({{\bf q}^+};\vec{\epsilon'}) 
$ as in (\ref{form862}) and $
\overline{\mathcal V}({{\bf q}^+};\vec{\epsilon'}) 
= \Phi_{{\bf q^+}}({\mathcal V}({{\bf q}^+};\vec{\epsilon'}))$.
For $\frak y \in \mathcal V({{\bf q}^+};\vec{\epsilon'})$ and 
$u' \in \mathscr L^{\bf q}_m$ we obtain
$
u'' = u' \circ \widehat{\Phi}_{{\bf q}^+,\frak y,\vec{\epsilon'}}^{-1} : \Sigma_{\frak y}(\vec{\epsilon'}) \to X
$ as in
(\ref{form1040400}).
We want to prove that 
the family of linear subspaces 
$
\mathcal P_{{\bf y}}(E_{{\bf q};{\bf p}}({\bf x}))
$
(see (\ref{form255}) and (\ref{form8000})) is $C^n$ family when we move  $\frak y,u'$.
Here  ${\bf x} \in \mathcal X_{k+1,\ell}(X,L,J;\beta)$ and  ${\bf y} = {\bf x} \cup {}_{\bf q}\vec{\frak w}_{\bf x}= ((\Sigma_{\frak y},\vec z_{\frak y},\vec{\frak z}_{\frak y}),u'')$.
We take ${}_{\bf p}\vec{\frak w}_{\bf x}$ so that Condition \ref{conds102} is satisfied.
In view of Definition \ref{defn113}, we will study the composition:
\begin{equation}
\aligned
\mathcal P_{{\bf x}\cup {}_{\bf q}\vec{\frak w}_{\bf x}} &\circ \mathcal P_{{\bf x}\cup {}_{\bf p}\vec{\frak w}_{\bf x}}^{-1} : 
C^{\infty}(\Sigma_{{\bf p}^+}(\vec{\epsilon});u_{{\bf p}}^*TX\otimes \Lambda^{01})
\\
&\to 
L^2_m(\Sigma_{\frak y}(\vec{\epsilon'});u_{{\bf x}}^*TX\otimes \Lambda^{01}) 
\to 
L^2_m(\Sigma_{{\bf q}^+}(\vec{\epsilon'});u_{{\bf q}}^*TX\otimes \Lambda^{01})
\endaligned
\end{equation}
and study $C^n$ dependence of the image of $E^0_{\bf p}({\bf p})$ by this map.
\begin{rem}
Note when we define $\mathcal P_{{\bf x}\cup {}_{\bf q}\vec{\frak w}_{\bf x}}$
we take $\bf q$ and ${\bf x}\cup {}_{\bf q}\vec{\frak w}_{\bf x}$ 
in place of $\bf p$ and ${\bf y}$ in (\ref{form8000}).
When we define $\mathcal P_{{\bf x}\cup {}_{\bf p}\vec{\frak w}_{\bf x}}$
we take $\bf p$ and ${\bf x}\cup {}_{\bf p}\vec{\frak w}_{\bf x}$ 
in place of $\bf p$ and ${\bf y}$ in (\ref{form8000}).
\end{rem}
We remark that Definition \ref{defn86} (1)(2) and Definition \ref{defn8911} (1) is 
obvious from construction. To prove Definition \ref{defn86} (3) (and so 
Definition \ref{defn8911} (2)) it suffices to prove the next lemma 
\footnote{In fact we can use partition of unity to reduce the case of general ${\bf e}$ supported 
away from node to the case of one with small support.}.
\begin{lem}
If ${\bf e} \in C^{\infty}(\Sigma_{\bf p};u_{{\bf p}}^*TX\otimes \Lambda^{01})$ is a 
smooth section which has a small compact support in $\Sigma_{{\bf p}^+}(\vec\epsilon)$, then
the map $({\frak y},u') \mapsto (\mathcal P_{{\bf x}\cup {}_{\bf q}\vec{\frak w}_{\bf x}} \circ \mathcal P_{{\bf x}\cup {}_{\bf p}\vec{\frak w}_{\bf x}}^{-1})({\bf e})$
is a $C^{n-1}$ map\footnote{Recall that in Proposition \ref{prop103}
we obtain a $C^n$ map 
$\Psi_{\bf p,\bf q}$.} to $L^2_m(\Sigma_{\bf q}(\vec{\epsilon'});u_{{\bf q}}^*TX\otimes \Lambda^{01})
$.
\end{lem}
\begin{proof}
Let $U^{\bf p}$ be a support of ${\bf e}$ and $z^{\bf p}$  its coordinate.
Let $\Psi_{{\bf p},{\bf q}}$ be as in Definition \ref{defn102}.
We may assume $\Psi_{{\bf p},{\bf q}}(\overline{\mathcal V}({{\bf q}^+};\vec{\epsilon'})\times \mathscr L^{\bf q}_m \times U^{\bf p})$
is contained in a single chart $U^{\bf q}$ and let $z^{\bf q}$ be its coordinate.
\par
We may also assume that $u'(U^{\bf q})$ for any $u'\in \mathscr L^{\bf q}_m$ and $u_{\bf p}(U^{\bf p})$ are 
contained in a single chart $U^{X}$ of $X$ and denote by 
${\frak e}_1,\dots,{\frak e}_d$, sections of the complex 
vector bundle $TX$ on $U^X$ which give a basis at all points.
\par
We put $\frak a = \Phi_{{\bf q}^+}(\frak y) \in \overline{\mathcal V}({{\bf q}^+};\vec{\epsilon'})$.
During the calculation below we write $\Psi_{{\bf p},{\bf q}}(z) = \Psi_{{\bf p},{\bf q}}(\frak a,u',z)$
for simplicity.
By definition the map $\mathcal P_{{\bf x}\cup {}_{\bf q}\vec{\frak w}_{\bf x}} \circ \mathcal P_{{\bf x}\cup {}_{\bf p}\vec{\frak w}_{\bf x}}^{-1}$ is induced by a bundle map
which is the tensor product of two bundle maps. One of them (See (\ref{form8181}))
is a composition of the parallel transports
\begin{equation}\label{form117}
T_{u_{\bf p}(\Psi_{{\bf p},{\bf q}}^{-1}(z^{\bf q}))}X \longrightarrow T_{u'(z^{\bf q})}X \longrightarrow T_{u_{\bf q}(z^{\bf q})}X.
\end{equation}
The other is (See (\ref{form8281}).)
\begin{equation}\label{form118}
\Lambda^{01}_{\Psi_{{\bf p},{\bf q}}^{-1}(z^{\bf q})}\Sigma_{\bf p} \longrightarrow 
\Lambda^{01}_{\hat{\Phi}_{{\bf q}^+,\frak y,\vec{\epsilon'}}(z^{\bf q})}\Sigma_{\bf x}
\longrightarrow
\Lambda^{01}_{z^{\bf q}}\Sigma_{\bf q}.
\end{equation}
The arrows in (\ref{form118}) are the complex linear parts of the derivatives of 
$\widehat{\Phi}_{{\bf q}^+,\frak y,\vec{\epsilon'}} \circ \Psi_{{\bf p},{\bf q}}$
and of $\hat{\Phi}_{{\bf q}^+,\frak y,\vec{\epsilon'}}^{-1}$.
In fact, by Definition \ref{defn102} and (\ref{10111}), 
$\widehat{\Phi}_{{\bf q}^+,\frak y,\vec{\epsilon'}} \circ \Psi_{{\bf p},{\bf q}}
= \widehat{\Phi}_{{\bf p}^+,\frak x_{\bf p},\vec{\epsilon}}$, 
where $\frak x_{\bf p}$ is defined by ${\bf x} \cup {}_{\bf p}\vec{\frak w}_{\bf x} = {\Phi}_{{\bf p}^+}(\frak x_{\bf p})$.
\par
If $h = \sum_i h_i \frak e_i$ is a section of $u_{\bf p}^*TX$ on $U^{\bf p}$ then (\ref{form117}) sends 
$h$ to the section
$$
z^{\bf q} \mapsto \sum_{i,j} h_i(\Psi_{{\bf p},{\bf q}}(\frak a,u',z^{\bf q})) G_{ij}(u_{\bf q}(z^{\bf q}),u'(z^{\bf q}),u_{\bf p}(\Psi_{{\bf p},{\bf q}}(\frak a,u',z^{\bf q}))) \frak e_j
$$
Here $G_{ij}$ is a smooth function on $U^X \times U^X \times U^X$.
\par
We take a $\frak y$ parametrized (smooth) family of complex structures $j_{\frak y}$ on $U^{\bf q}$,
which is a pull back of the complex structure on $\Sigma_{\bf x} = \Sigma_{\frak y}$ by $\widehat{\Phi}_{{\bf q}^+,\frak y,\vec{\epsilon'}}$.
Then (\ref{form118}) is a composition of 
$$
(D\Psi_{{\bf p},{\bf q}})^{01} : \Lambda^{01}_{\Psi_{{\bf p},{\bf q}}^{-1}(z^{\bf q})}U^{\bf p} \to 
\Lambda^{01}_{z^{\bf q}}(U^{\bf q},j_{\frak y})
\quad
\text{and}
\quad
({\rm id})^{01} :  \Lambda^{01}_{z^{\bf q}}(U^{\bf q},j_{\frak y})
\to \Lambda^{01}_{z^{\bf q}}(U^{\bf q},j_{0}).
$$
(Here $j_0$ is the complex structure of $U^{\bf q} \subset \Sigma_{\bf q}$.)
Therefore using Proposition  \ref{prop103} there exists a $C^{n-1}$ function 
$f(\frak y,u',z^{\bf q})$ such that (\ref{form118})  is written as 
$$
d\overline z^{\bf p}  \mapsto f(\frak y,u',z^{\bf q}) d\overline z^{\bf q}.
$$
We put ${\bf e} = \sum_i {\bf e}_i(z^{\bf p}) \frak e_i \otimes d\overline z^{\bf p}$.
Here ${\bf e}_i$ are smooth functions.
Then
$$
\aligned
&(\mathcal P_{{\bf x}\cup {}_{\bf q}\vec{\frak w}_{\bf x}} \circ \mathcal P_{{\bf x}\cup {}_{\bf p}\vec{\frak w}_{\bf x}}^{-1})({\bf e}) \\
&= 
f(\frak y,u',z^{\bf q})\sum_{i,j}
{\bf e}_i(\Psi_{{\bf p},{\bf q}}(\frak a,u',z^{\bf q})) G_{ij}(u_{\bf q}(z^{\bf q}),u'(z^{\bf q}),u_{\bf p}(\Psi_{{\bf p},{\bf q}}(\frak a,u',z^{\bf q}))) 
\frak e_j \otimes d\overline z^{\bf q}.
\endaligned
$$
Since ${\bf e}_i$, $u_{\bf p}$, $u_{\bf q}$, $G_{ij}$ are smooth, $f$ is of $C^{n-1}$ class, 
$\Psi_{{\bf p},{\bf q}}$ is of $C^n$ class, 
and $\frak a = \Phi_{{\bf q}^+}(\frak y)$ with $\Phi_{{\bf q}^+}$ smooth, this is a $C^{n-1}$ map of $(\frak y,u',z^{\bf 	q})$, as required.
\end{proof}
Definition \ref{defn8911} (2) is now proved.
We remark that we never used the fact that $u_{\bf q}$ is pseudo holomorphic in the above proof.
Therefore the proof of Definition \ref{defn8911} (3) is the same.\footnote{The proof 
of Lemma \ref{lem832} we gave in this article uses the fact that $u_{\bf p}$ there is pseudo holomorphic. 
We however never take local transversals to ${\bf q}$ in this subsection and so the pseudo-holomorphicity is 
not needed here.}
The proof of Proposition \ref{prop114} is complete.
\end{proof}

\subsection{Global construction of obstruction bundle data}
\label{sub:globalclnst}

In this and the next subsections we prove Theorem \ref{them111}.
The proof is based on the argument of \cite[page 1003-1004]{FO}.
For each ${\bf p} \in \mathcal M_{k+1,\ell}(X,L;\beta)$
we use Proposition \ref{prop114} to find its neighborhood $\frak U({\bf p})$ in $\mathcal M_{k+1,\ell}(X,L;\beta)$
such that if ${\bf q} \in \frak U({\bf p})$ then
$E_{{\bf q};{\bf p}}({\bf x})$ is an obstruction bundle data at ${\bf q}$.
We take a compact subset $K(\bf p)$ 
\index{00K1p@$K(\bf p)$} 
of $\frak U({\bf p})$ such that $K({\bf p})$ 
is the closure of an open neighborhood $K_{o}({\bf p})$ of ${\bf p}$.
\par
We note that during the construction of $E_{{\bf q};{\bf p}}({\bf x})$ 
we have chosen a linear subspace $E^0_{\bf p}(\bf p)$ 
(See Lemma \ref{lem1122}.) as well as 
strong stabilization data at ${\bf p}$.
\par
Since $\mathcal M_{k+1,\ell}(X,L;\beta)$ is compact, we can find a 
finite subset $\{{\bf p}_1,\dots,{\bf p}_{\mathscr P}\}$ of $\mathcal M_{k+1,\ell}(X,L;\beta)$
such that 
\begin{equation}\label{2199}
\bigcup_{i=1}^{\mathscr P} K_{o}({\bf p}_i) = \mathcal M_{k+1,\ell}(X,L;\beta).
\end{equation}
For ${\bf q} \in \mathcal M_{k+1,\ell}(X,L;\beta)$ we put
$
I({\bf q}) = \{ i \in \{1,\dots,\mathscr P\} \mid {\bf q} \in K({\bf p}_i) \}.
$\index{00I1qq@$I({\bf q})$}

\begin{lem}\label{lem1116}
We may perturb $E^0_{{\bf p}_i}({\bf p}_i)$ by an arbitrary small amount in $C^{2}$ norm 
so that the following holds.
For each ${\bf q} \in \mathcal M_{k+1,\ell}(X,L;\beta)$
the sum $\sum_{i \in I({\bf q})} E_{{\bf q},{\bf p}_i}({\bf q})$
of vector subspaces in $C^{\infty}(\Sigma_{\bf q}(\vec{\epsilon'});u_{{\bf q}}^*TX\otimes \Lambda^{01})$
is a direct sum
\begin{equation}
\bigoplus_{i \in I({\bf q})} E_{{\bf q},{\bf p}_i}({\bf q}).
\end{equation}
\end{lem}
Postponing the proof of Lemma \ref{lem1116} to the next subsection we continue the 
proof of Theorem \ref{them111}.
\par
Note we may choose $E^{\prime,0}_{{\bf p}_i}$, the perturbation of $E^{0}_{{\bf p}_i}$,
sufficiently close to $E^{0}_{{\bf p}_i}$,
so that for ${\bf q} \in K({\bf p}_i)$ the conclusion of Proposition \ref{prop114}
still holds after this perturbation.
\par
For each ${\bf x}$ sufficiently close to ${\bf q}$ we define
$
E_{\bf q}({\bf x}) = \bigoplus_{i \in I({\bf q})} E_{{\bf q},{\bf p}_i}({\bf x}).
$
(Since ${\bf x}$ is sufficiently close to ${\bf q}$ the right hand side 
is still a direct sum by Lemma \ref{lem1116}.)
\par
Now we prove that $\{E_{\bf q}({\bf x})\}$ satisfies 
Definition \ref{defn51} (1)-(5).
(1)(2)(5) are immediate consequences of the fact that 
$E_{{\bf q},{\bf p}_i}({\bf x})$ satisfies the same property.
(3) is a consequence of $I(\bf q) \ne \emptyset$ (which follows from (\ref{2199}))
and the fact that $E_{{\bf q},{\bf p}_i}({\bf x})$ satisfies (3).
\par
We check (4). Let ${\bf q}  \in \mathcal M_{k+1,\ell}(X,L;\beta)$.
Since $K({\bf p}_i)$ are all closed sets, there exists a neighborhood 
$U(\bf q)$ in $\mathcal M_{k+1,\ell}(X,L;\beta)$ such that  ${\bf q}' 
\in U(\bf q)$ implies $I({\bf q}') \subseteq I({\bf q})$.
Therefore
$
E_{{\bf q}'}({\bf x}) \subseteq E_{\bf q}({\bf x})
$
when both sides are defined. This implies (4).

\subsection{Proof of Lemma \ref{lem1116}}
\label{sub:transtechlem}

To complete the proof of Theorem \ref{them111} it remains to prove Lemma \ref{lem1116}.
The proof here is a copy of \cite[Section 27]{foootech}.
We begin with two simple lemmas.
(All the vector spaces in this subsection are complex vector spaces.)
\begin{shitu}\label{situ1117}
Let $V$ be a $D$ dimensional manifold, $V_i$ ($i=1,\dots,l$) open subsets of $V$ 
and $K_i \subset V_i$ compact subsets. 
$\pi_i : \mathcal E_i \to V_i$ is a $d_i$ dimensional vector bundle on $V_i$ 
and $E$ is a $d$ dimensional vector space.
Suppose 
$
F_i : \mathcal E_i \to E 
$
is a $C^1$ map which is linear on each fiber of $\mathcal E_i$.
Let
$Gr_a(E)$ be the Grassmannian manifold consisting of all $a$ dimensional subspaces 
of $E$.
\end{shitu}
\begin{lem}\label{lem1118}
In Situation \ref{situ1117} we assume
\begin{equation}
a+D\sum d_i < d.
\end{equation} 
Then the set of $E_0 \in Gr_a(E)$ satisfying the next condition is 
dense.
\begin{enumerate}
\item[(*)]
For any $v \in V$ we consider the sum of the linear subspaces 
$F_i(\pi_i^{-1}(v)) \subset E$ 
for $i$ with $v \in K_i$ and denote it by $F(v)$.
Then 
$$
F(v) \cap E_0 = \{0\}.
$$
\end{enumerate}
\end{lem}
\begin{proof}
The proof is by induction on $a$. Suppose  $E'_{0} \in Gr_{a-1}(E)$ 
satisfies (*). It suffices to prove that the set of $e \in E \setminus \{0\}$
such that $\C e \cap E'_0 = \{0\}$ and $E'_0 + \C e$ satisfies (*) is dense.
\par
Let $\mathcal L \subset \{1,\dots,l\}$ and 
$U_{\mathcal L} = \bigcap_{i\in \mathcal L}U_i$. 
Let $\mathcal E_{\mathcal L}$ be
the total space of the  Whitney sum 
bundle $\bigoplus_{i \in \mathcal L} \mathcal E_i$
on $U_{\mathcal L}$.
We define
$
\hat F_{\mathcal L} : \C \times E'_0 \times \mathcal E_{\mathcal L} \to E
$
as follows.  Let $w_i \in \pi_i^{-1}(v) \subset \mathcal E_i$, then 
$$
\hat F_{\mathcal L}(r,x,(w_i)_{i\in \mathcal L}) 
= r(x + \sum_{i\in \mathcal L} F_i(w_i)).
$$
This map is $C^1$ and the dimension of the domain is 
strictly smaller than $d$. Therefore the image of $\hat F_{\mathcal L}$ is nowhere dense.
On the other hand if $e$ is not contained in the union 
of the images of $\hat F_{\mathcal L}$ for various $\mathcal L$, then 
$E'_0 + \C e$ satisfies (*).
In fact suppose 
$\sum F_i(w_i) + v + c e = 0$ with  $w_i \in \pi_i^{-1}(v)$, $v \in E$ and $c\in \C$.
If $c=0$ the induction hypothesis implies $\sum_i F_i(w_i) = 0$, $v=0$.
If $c\ne 0$, then $e = \hat F_{\mathcal L}(-1/c,v,(w_i))$, a contradiction.
\end{proof}
We use the equivariant version of Lemma \ref{lem1118}.
\begin{shitu}\label{situ1118}
Let $\Gamma$ be a finite group of order $g$ and $d' >0$.
In Situation \ref{situ1117} we assume in addition that $E$ 
is a $\Gamma$ vector space such that 
any irreducible representation $W_{\sigma}$
of $\Gamma$ has its multiplicity in $E$ larger than $d'$.
Let ${\rm Rep}(\Gamma)$ be the set of 
all irreducible representations of $\Gamma$ over $\C$.
\par
Let
$
Gr_{(a_{\sigma}:\sigma \in {\rm Rep}(\Gamma))}(E)$
be the set of all $\Gamma$ invariant linear subspaces $E_0$ of $E$
such that $E_0$  is isomorphic to 
$\bigoplus_{\sigma\in {\rm Rep}(\Gamma)} W_{\sigma}^{a_{\sigma}}$ as $\Gamma$ 
vector spaces.
Let $a = \sup\{a_{\sigma}\}$.
\end{shitu}
\begin{lem}\label{lemma1110}
Suppose we are in Situation \ref{situ1118}.
We assume 
\begin{equation}\label{1114ddd}
a+ gD\sum d_i < d'.
\end{equation}
Then the set of all $E_0 \in Gr_{(a_{\sigma}:\sigma \in {\rm Rep}(\Gamma))}(E)$
satisfying the Condition (*) in Lemma \ref{lem1118}
is dense in $Gr_{(a_{\sigma}:\sigma \in {\rm Rep}(\Gamma))}(E)$.
\end{lem}
Note we do not assume $\Gamma$ equivariance of 
$V$, $\mathcal E_i$ or $F_i$ in Lemma \ref{lemma1110}.
\begin{proof}
Let $m_{\sigma} \ge d'$ be the multiplicity of $W_{\sigma}$ in $E$.
Then there exists an obvious diffeomorphism
\begin{equation}\label{form1115555}
Gr_{(a_{\sigma}:\sigma \in {\rm Rep}(\Gamma))}(E) 
\cong 
\prod_{\sigma} Gr_{a_{\sigma},m_{\sigma}}.
\end{equation}
Here $Gr_{a_{\sigma},m_{\sigma}}$ is the Grassmanian manifold 
of all $a_{\sigma}$ dimensional subspace of $\C^{m_{\sigma}}$.
Let $\C[\Gamma]$ be the group ring of the finite group $\Gamma$.
We put:
\begin{equation}\label{form1116}
\mathcal E^+_i = \mathcal E_i \otimes_{\C} \C[\Gamma] 
\cong \bigoplus_{\sigma\in {\rm Rep}(\Gamma)}\mathcal E_i \otimes_{\C}W_{\sigma}^{d_{\sigma}}.
\end{equation}
Here $d_{\sigma}$ is the multiplicity of $W_{\sigma}$ in $\C[\Gamma]$.
Note $d_{\sigma} \le g$. (\ref{form1116}) is an isomorphism of 
$\Gamma$ equivariant vector bundles.
$F_i$ induces a $\Gamma$ equivariant map $\mathcal E^+_i \to E_0$.
It then induces  $\Gamma$ equivariant fiberwise linear 
maps $F_{i,\sigma} : \mathcal E_i \otimes_{\C}W_{\sigma}^{d_{\sigma}} \to W_{\sigma}^{m_{\sigma}}$
by decomposing into irreducible representations.
The map $F_{i,\sigma}$ can be identified with a fiberwise
$\C$-linear map $\overline F_{i,\sigma} : \mathcal E_i \otimes_{\C}\C^{d_{\sigma}} \to \C^{m_{\sigma}}$
by Schur's lemma.
(\ref{1114ddd}) implies 
$
a_{\sigma} + D d_{\sigma} \sum d_i < m_{\sigma}.
$
Therefore using (\ref{form1115555}) we apply Lemma \ref{lem1118}
for each $\sigma$ and prove Lemma \ref{lemma1110}.
\end{proof}
\begin{proof}[Proof of Lemma \ref{lem1116}]
We take
$\{{\bf p}_1,\dots,{\bf p}_{\mathscr P}\}$,
$\frak U({\bf p}_i)$, $K({\bf p}_i)$
as in Subsection \ref{sub:globalclnst}.
We also have taken $E^0_{{\bf p}_i}({\bf p}_i)$
so that the conclusion of Lemma \ref{lem1122} 
is satisfied.
Let $\frak d$ be the supremum of the 
dimension of $E^0_{{\bf p}_i}({\bf p}_i)$
and $g$ the supremum of the order of 
${\rm Aut}^+({\bf p}_i)$.
For each $i$ we take $E^+({\bf p}_i)$
in $C^{\infty}(\Sigma_{{\bf p}_i},u_{{\bf p}_i}^*TX \otimes \Lambda^{01})$
such that:
\begin{enumerate}
\item
Lemma \ref{lem1122} (1)(3) are satisfied.
\item
$E^0_{{\bf p}_i}({\bf p}_i) \subset E^+({\bf p}_i)$.
\item
For any irreducible representation $W_{\sigma}$
of ${\rm Aut}^+({\bf p}_i)$
the multiplicity of $W_{\sigma}$ in $E^+({\bf p}_i)$
is larger than $d^+$. Here $d^+$ is determined later.
\end{enumerate}
We will prove, 
by induction on $I = 1,\dots,\mathscr{P}$, that we can perturb 
$E^0_{{\bf p}_i}({\bf p}_i)$
in $E^+({\bf p}_i)$ in an arbitrary small amount 
to obtain $E^1_{{\bf p}_i}({\bf p}_i)$
so that statement ($I$) holds.
\begin{enumerate}
\item[($I$)]
For any 
$J\subseteq \{1,\dots,I\}$ and ${\bf q} \in \bigcap_{i\in J} K({\bf p}_i)$,
the sum
$\sum_{i\in J} E'_{{\bf q};{\bf p}_i}({\bf q})$ is a 
direct sum.
Here 
we define $E'_{{\bf q};{\bf p}_i}({\bf q})$ in the same way as 
$E_{{\bf q};{\bf p}_i}({\bf q})$, 
using $E^1_{{\bf p}_i}({\bf p}_i)$
instead of $E^0_{{\bf p}_i}({\bf p}_i)$.
\end{enumerate}

Suppose ($I-1$) holds.
We choose a weak stabilization data  $\vec{\frak w}_{{\bf p}_I}$
and a neighborhood $\mathscr N({\bf p}^+_I) 
\subset \mathcal X_{k+1,\ell+\ell'}(X,L;\beta)$ of ${\bf p}^+_I = {\bf p}_I \cup \vec{\frak w}_{{\bf p}_I}$.
We define obstruction bundle data $\{E_{{\bf p}_I}({\bf x})\}$ at ${\bf p}_I$ by using 
$E^0_{{\bf p}_I}({\bf p}_I)$. 
Proposition \ref{prop114}
and Lemma \ref{lem91111} then imply that 
$$
U({\bf p}^+_I) : = \{{\bf x} \in \mathscr N({\bf p}^+_I)
\mid \overline{\partial}u_{\bf x} \in E_{{\bf p}_I}({\bf x}) \}
$$
has a structure of a finite dimensional orbifold $V({\bf p}^+_I)/{\rm Aut}({\bf p}^+_I)$.\footnote{We use $E^0_{{\bf p}_I}({\bf x})$ here only 
to cut down $\mathscr N({\bf p}^+_I)$ to a finite dimensional orbifold.
So we do not need to use any particular relation of this space  $E^0_{{\bf p}_I}({\bf x})$
to the obstruction bundle data $E_{\bf p}(\bf x)$ we finally obtain.}
Let $D$ be the supremum of the dimension of  $V({\bf p}^+_I)$ for $I=1,\dots,\mathscr P$.
\par
We apply Lemma \ref{lemma1110} as follows.
Put $V = V({\bf p}^+_I)$, $E = E^+({\bf p}_I)$, $\Gamma = {\rm Aut}({\bf p}^+_I)$.
For $i \le I-1$, we define $K_i$ as the inverse image of  $K_{{\bf p}_i} \cap K_{{\bf p}_I}
\cap U({\bf p}^+_I)$ in $V({\bf p}^+_I)$ and 
take a sufficiently small neighborhood $V_i$ of $K_i$.
For $\tilde{\bf y} \in V_i$ we take
its image ${\bf y} \in U({\bf p}^+_I)$.
Put $\mathcal E_i(\tilde{\bf y}) = E^1_{{\bf p}_i}(
\frak{forget}_{\ell+\ell',\ell}({\bf y}))$ and define a vector bundle 
$
\mathcal E_i = \bigcup_{\tilde{\bf y} \in V_i} \mathcal E_i(\tilde{\bf y}) \times \{\tilde{\bf y}\}
$
on $V_i$.
We  consider an isomorphism
$$
\mathcal P_{{\bf y}} : 
L^2_m(\Sigma_{\bf y}(\vec \epsilon);u_{{\bf y}}^*TX\otimes \Lambda^{01})
\to 
L^2_m(\Sigma_{{\bf p}_I}(\vec \epsilon);u_{{\bf p}_I}^*TX\otimes \Lambda^{01})
$$
for $\tilde{\bf y} \in V_i$ with ${\bf y} = [\tilde{\bf y}]\in U({\bf p}^+_I)$ and compose $\mathcal P_{{\bf y}}$ 
with the orthogonal projection
$
\Pi : L^2_m(\Sigma_{{\bf p}_I}(\vec \epsilon);u_{{\bf p}_I}^*TX\otimes \Lambda^{01})
\to E^+({\bf p}_I) = E
$.
The restrictions to $\mathcal E_i(\tilde{\bf y})$ of the composition $
\Pi\circ\mathcal P_{{\bf y}}$
for various $\tilde{\bf y} \in V_i$
defines a $C^1$ map $F_i : \mathcal E_i \to E$.
\par
We may take $d^+$ depending only on $\mathscr P$, $g$, $\frak d$, $D$
so that (\ref{1114ddd}) is satisfied.
By Lemma \ref{lemma1110}, we obtain 
$(I)$ as an immediate consequence of $(I-1)$ and Lemma \ref{lem1118} $(*)$.
\par
The claim $(I)$, in the case $I = \mathscr P$, implies Lemma \ref{lem1116}.
\end{proof}

\section{From $C^n$ to $C^{\infty}$}
\label{sec:Cmugen}

In Sections \ref{sec;Kuracharsmooth} and \ref{sec:changesmoo}
we have constructed a Kuranishi structure of $C^n$ class for arbitrary but 
fixed $n$.
In this section we provide the way to construct one of $C^{\infty}$ class.
The argument of this section is a copy of \cite[Section 26]{foootech}.
We use Condition (3) in Definition \ref{defn8911}  in this section.

\subsection{$C^{\infty}$ structure of Kuranishi chart}
\label{sub:kurasmstr}

We study the map $({\rm Pr}^{\rm source},{\rm Pr}^{\rm map})$
and also $V_{\bf p} \subset U_{{\bf p}^+}$.
($U_{{\bf p}^+}$ is defined in Definition \ref{defn72}.
$V_{\bf p}$ is defined in Definition \ref{defn910}.)
In Section \ref{sec;Kuracharsmooth} we proved that 
the image of $U_{{\bf p}^+}$ by the map $({\rm Pr}^{\rm source},{\rm Pr}^{\rm map})$ 
is a $C^n$ submanifold if $m$ is sufficiently larger than $n$
and also the image of $V_{\bf p}$ is a $C^n$ submanifold.
We first remark the following:
\begin{lem}\label{lem121}
The images of $U_{{\bf p}^+}$ and  $V_{\bf p}$ 
by the map $({\rm Pr}^{\rm source},{\rm Pr}^{\rm map})$ are $C^{\infty}$ submanifolds 
at ${\bf p}^+$.
\end{lem}
\begin{proof}
It suffices to show that they are of  $C^{n'}$ class for any $n'$.
It is obvious from the construction of Subsection \ref{subsec:smoproof} that for any $n'$ we can find $m'$ such that 
for an open neighborhood $U'_{{\bf p}^+}$ of ${\bf p}^+$ in $U_{{\bf p}^+}$
and $\vec{\epsilon'}$ the next diagram commutes. 
\begin{equation}\label{diagram12-1}
\begin{CD}
U'_{{\bf p}^+}
@ >{({\rm Pr}^{\rm source},{\rm Pr}^{\rm map})}>>
\!\!\!\!\!\!\!\!\!\!\!\!\!\!
\!\!\!\!\!\!\!\!\!\!\!\!\!\!\!\!\!\!\!\!\!\!\!\!\!\!\!\!\!\!\!\!\!\!\!\!\!\!\!\!\!\!\!\!\!\!\!\! \mathcal M^{\rm d,log}_{k+1,\ell+\ell'}
\atop 
\times L^2_{m'}(\Sigma_{{\bf p} \cup \vec {\frak w}_{\bf p}}(\vec{\epsilon'}),
\partial \Sigma_{{\bf p} \cup \vec {\frak w}_{\bf p}}(\vec{\epsilon'});X,L)\\
@ VVV @VVV 
\\
U_{{\bf p}^+}
@ >{({\rm Pr}^{\rm source},{\rm Pr}^{\rm map})}>>
\!\!\!\!\!\!\!\!\!\!\!\!\!\!
\!\!\!\!\!\!\!\!\!\!\!\!\!\!\!\!\!\!\!\!\!\!\!\!\!\!\!\!\!\!\!\!\!\!\!\!\!\!\!\!\!\!\!\!\!\!\!\!
\mathcal M^{\rm d,log}_{k+1,\ell+\ell'}
\atop
\times L^2_m(\Sigma_{{\bf p} \cup \vec {\frak w}_{\bf p}}(\vec{\epsilon}),
\partial \Sigma_{{\bf p} \cup \vec {\frak w}_{\bf p}}(\vec{\epsilon});X,L)
\end{CD}
\end{equation}
where vertical arrows are inclusions.
Moreover the first horizontal arrow is of $C^{n'}$ class 
and the image of $V_{\bf p}$ by the first horizontal 
arrow is also of $C^{n'}$ class.
(See Lemma \ref{lem108}.)\footnote
{
Note the map ${\rm Pr}^{\rm map}$ depends only on the embedding 
$\Sigma_{{\bf p} \cup \vec {\frak w}_{\bf p}}(\vec{\epsilon'}) \to \Sigma_{\bf x}$.
(Here ${\bf x} \in U_{{\bf p}^+}$.) Therefore it depends only on the 
stabilization and trivialization data ${\frak W}_{\bf p}$ we use and is independent of 
the Sobolev exponent $m$ in $L^2_m$.
}
Therefore the image of the second horizontal arrow 
is of $C^{n'}$ class at the image of ${\bf p}^+$ and  that the image of $V_{\bf p}$ by the second horizontal 
arrow is also of $C^{n'}$ class  at the image of ${\bf p}^+$.
The lemma follows.
\end{proof}
Note the size of the neighborhood $U'_{{\bf p}^+}$ may become smaller 
and converge to $0$
as $m'$ goes to infinity.\footnote{This is because the size of the
domain of the convergence of the Newton's iteration scheme 
we explained in Subsection \ref{subsec:glue} may go to $0$ as $m$ goes to $\infty$.}
So the above proof of Lemma \ref{lem121} can be used to 
show the smoothness of $U_{{\bf p}^+}$ and  $V_{\bf p}$ only 
at ${\bf p}^+$.
We fix $m$,  $U_{\bf p}^+$ and $V_{\bf p}$ and will prove:
\begin{prop}\label{prop122}
The image of $V_{\bf p}$ 
by the map $({\rm Pr}^{\rm source},{\rm Pr}^{\rm map})$ is a submanifold of $C^{\infty}$ class.
\end{prop}
\begin{cor}\label{cor123}
The Kuranishi chart $(U_{\bf p},\mathcal E_{\bf p},\frak s_{\bf p},\psi_{\bf p})$
we produced in Section \ref{sec;Kuracharsmooth} is of $C^{\infty}$ class.
\end{cor}
It is easy to see that Proposition \ref{prop122} 
implies Corollary \ref{cor123}.
\begin{proof}[Proof of Proposition \ref{prop122}]
Let ${\bf r}^+ \in V_{\bf p} \subset U_{{\bf p}^+}$.
We put 
$
{\bf r}^+ = {\bf r} \cup \vec{\frak w}_{\bf r}
$
where
$
{\bf r} = ((\Sigma_{\bf r},\vec z_{\bf r},\vec{\frak z}_{\bf r}),u_{\bf r}).
$
We apply Lemma \ref{lem1244} with ${\bf q}$ replaced by ${\bf r}$
(See Remark \ref{lem1015}.) and obtain $\frak W_{\bf r} = (\vec {\frak w}_{\bf r},\{\varphi_{{\bf r},a,i}^{\rm s}\}, \{\varphi_{{\bf r},a,j}^{\rm d}\},
\{\psi_{{\bf r},a}\})$. (See Remark \ref{rem83}.)
\par
We write ${\rm Pr}^{\rm map}_{\bf p}$ etc. in place of 
${\rm Pr}^{\rm map}$ hereafter.
We recall that the map ${\rm Pr}^{\rm map}_{\bf p}$ 
(resp. ${\rm Pr}^{\rm map}_{\bf r}$) depends 
not only on the weak stabilization data $\vec{\frak w}_{\bf p}$
(resp. $\vec{\frak w}_{\bf r}$)
but also on the stabilization and trivialization data, 
$\frak W_{\bf p} = (\vec {\frak w}_{\bf p},\{\varphi_{{\bf p},a,i}^{\rm s}\}, \{\varphi_{{\bf p},a,j}^{\rm d}\},
\{\phi_{{\bf p},a}\})$
(resp.
$\frak W_{\bf r} = (\vec {\frak w}_{\bf r},\{\varphi_{{\bf r},a,i}^{\rm s}\}, \{\varphi_{{\bf r},a,j}^{\rm d}\},
\{\psi_{{\bf r},a}\})$).
\par
We consider Diagram (\ref{diagram12-2})
with ${\bf q}$ replaced by ${\bf r}$.
For each $n'$ we can choose $m'$ such that the image of
the first horizontal arrow
is of $C^{n'}$ class.
(Here we use Definition \ref{defn8911} (3) for ${\bf r}$ 
and apply the $C^{n'}$ version of Lemma \ref{lem73} (which is proved in Subsection 
\ref{subsec:Cnkurastru}) at ${\bf r}$.)
The right vertical arrow is smooth.
In fact this map is $f \mapsto f\circ \phi$ for an open smooth 
embedding $\phi$ mentioned in Lemma \ref{lem1244}.
\par
Therefore the commutativity of Diagram (\ref{diagram12-2}) implies that the image of 
$({\rm Pr}^{\rm source},{\rm Pr}^{\rm map}_{\bf p})$ in Diagram (\ref{diagram12-2})
is of $C^{n'}$ class at ${\bf r}^+$.
We can also prove that the image of $V_{\bf p}$ under the map $({\rm Pr}^{\rm source},{\rm Pr}^{\rm map}_{\bf p})$ 
is of $C^{n'}$ class at ${\bf r}^+$.
Since $n'$ and ${\bf r}^+$ are arbitrary, Proposition \ref{prop122} follows.
%\end{proof}
The proof of Proposition \ref{prop122} 
and Lemma \ref{lem73} (the $C^{\infty}$ version) are now complete.
\end{proof}

\subsection{Smoothness of coordinate change}
\label{sub:smcochange}

In this subsection we prove that the coordinate change 
we produced in Section \ref{sec:changesmoo} 
is of $C^{\infty}$ class with respect to the 
$C^{\infty}$ structure we defined in the last 
subsection.
\par
We first consider the case ${\bf q} = {\bf p}$.
We take two strong stabilization data
$(\frak W_{\bf p}^{j},\vec{\mathcal N}_{\bf p}^j)$ $(j=1,2)$
at ${\bf p}$. Moreover we take two obstruction bundle data 
$\{E_{\bf p}^j(\bf x)\}$ $(j=1,2)$ at ${\bf p}$.
We obtain two Kuranishi charts $(U^j_{\bf p},\mathcal E^j_{\bf p},\frak s^j_{\bf p},\psi^j_{\bf p})$
for $j=1,2$ using them.
We assume $E_{\bf p}^1({\bf x}) \subseteq E_{{\bf p}}^2({\bf x})$.
\par
Then by the construction of Subsection \ref{subsec:cchangeconst}
we obtain a coordinate change $(\varphi_{21},\hat{\varphi}_{21})$,
where $\varphi_{21} : U^{21}_{\bf p} \to U^{2}_{\bf p}$ 
is an embedding of orbifolds from open subset $U^{21}_{\bf p}$ 
of $U^{1}_{\bf p}$. The map $\hat{\varphi}_{21}$ is an 
embedding of orbibundles $\mathcal E^1_{\bf p}\vert_{U^{21}_{\bf p}} \to \mathcal E^2_{\bf p}$.

\begin{lem}\label{lem126}
$\varphi_{21}$, $\hat{\varphi}_{21}$ are of $C^{\infty}$ class at ${\bf p}$.
\end{lem}
\begin{proof}
The proof is similar to the proof of Lemma \ref{lem121}.
For any $n'$ we can use $L^2_{m'}$ space 
to show that $\varphi_{21}$, $\hat{\varphi}_{21}$ are of $C^{n'}$ class
on $U^{21}_{\bf p} \cap U'_{\bf p}$, where $U'_{\bf p}$ is a neighborhood 
of ${\bf p}$ in  $U^{1}_{\bf p}$ and depends on $m'$,$n'$.
This is a consequence of Subsection \ref{subsec:smoproof}.
Since $n'$ is arbitrary $\varphi_{21}$ is of $C^{\infty}$ class at ${\bf p}$.
The proof for $\hat{\varphi}_{21}$ is the same.
\end{proof}

Suppose we are given obstruction bundle data $\{E_{\bf p}({\bf x}) \mid {\bf p},
{\bf x}\}$ and are in Situation \ref{soti1011}.
We also take local transversals 
$\vec{\mathcal N_{\bf p}}$, $\vec{\mathcal N_{\bf q}}$
such that $(\frak W_{\bf p},\vec{\mathcal N}_{\bf p})$,
$(\frak W_{\bf q},\vec{\mathcal N}_{\bf q})$ are
strong obstruction bundle data.
\par
We then obtain Kuranishi charts 
$(U_{\bf p},\mathcal E_{\bf p},\frak s_{\bf p},\psi_{\bf p})$
(resp. $(U_{\bf q},\mathcal E_{\bf q},\frak s_{\bf q},\psi_{\bf q})$)
at ${\bf p}$ (resp. ${\bf q}$).
We assume ${\bf q}$ is close to ${\bf p}$.
Then by the construction of Subsection \ref{subsec:cchangeconst}
we obtain a coordinate change $(\varphi_{{\bf p}{\bf q}},\hat{\varphi}_{{\bf p}{\bf q}})$
of $C^n$ class.
\begin{prop}\label{prop12777}
$(\varphi_{{\bf p}{\bf q}},\hat{\varphi}_{{\bf p}{\bf q}})$ is 
of $C^{\infty}$ class.
\end{prop}
\begin{proof}
Let ${\bf r} \in U_{{\bf p}{\bf q}}$.
We will prove $\varphi_{{\bf p}{\bf q}}$ is of $C^{\infty}$ class at ${\bf r}$.
\par
We define strong stabilization data $({}_{\bf p}\frak W_{\bf r},{}_{\bf p}\vec{\mathcal N}_{\bf r})$
at ${\bf r}$ as follows.
We first take $_{\bf p}\vec{\frak w}_{\bf r}$ so that Condition \ref{conds102}
(with ${\bf x}$ replaced by ${\bf r}$) is satisfied.
We then 
apply Sublemma \ref{sublem125} to obtain analytic families of coordinates 
and smooth local trivializations. We thus obtain 
${}_{\bf p}\frak W_{\bf r}$. We put ${}_{\bf p}{\mathcal N}_{{\bf r},i} = {\mathcal N}_{{\bf p},i}$.
We define $({}_{\bf q}\frak W_{\bf r},{}_{\bf q}\vec{\mathcal N}_{\bf r})$ in the same 
way replacing ${\bf p}$ by ${\bf q}$.
We also put ${}_{\bf p}E_{\bf r}({\bf x}) = E_{\bf p}({\bf x})$
and ${}_{\bf q}E_{\bf r}({\bf x}) = E_{\bf q}({\bf x})$.

\begin{lem}
Using $({}_{\bf p}\frak W_{\bf r},{}_{\bf p}\vec{\mathcal N}_{\bf r})$
and ${}_{\bf p}E_{\bf r}({\bf x}) = E_{\bf p}({\bf x})$
(resp. $({}_{\bf q}\frak W_{\bf r},{}_{\bf q}\vec{\mathcal N}_{\bf r})$
and ${}_{\bf q}E_{\bf r}({\bf x}) = E_{\bf q}({\bf x})$)
we can construct
$({}_{\bf p}U_{\bf r},{}_{\bf p}\mathcal E_{\bf r},{}_{\bf p}\frak s_{\bf r},{}_{\bf p}\psi_{\bf r})$
(resp. $({}_{\bf q}U_{\bf r},{}_{\bf q}\mathcal E_{\bf r},{}_{\bf q}\frak s_{\bf r},{}_{\bf q}\psi_{\bf r})$), 
which is a Kuranishi chart at ${\bf r}$.
\end{lem}
\begin{proof}
We have completed the construction of smooth Kuranishi charts in the last 
subsection. The only difference here is 
the fact that ${\bf r}$ may not be a point 
of the moduli space $\mathcal M_{k+1,\ell}(X,L;\beta)$.
In other words, $u_{\bf r}$ may not be pseudo holomorphic. On the other hand, it is smooth.
The construction of a Kuranishi chart goes through 
\footnote{We use Definition \ref{defn8911} (3) 
in place of Definition \ref{defn8911} (2).} in this case except the following point.
During the proof of Lemma \ref{lem91111} we use the
fact that $\overline\partial u_{\bf p} = 0$.
\par
However we can prove the conclusion of Lemma \ref{lem91111}
in our situation as follows.
\par
We used $\overline\partial u_{\bf p} = 0$ to prove 
that the map ${\bf x} \mapsto u_{\bf x}({\frak z}_{{\bf x},\ell+i})$
(See (\ref{formfor911})), which is a map from $U_{{\bf p}^+}$ to $X$, is 
transversal to $\mathcal N_{{\bf p},i}$ at ${\bf p}^+$.
We prove that the same map ${}_{\bf p}U_{{\bf r}^+} \to X$
is 
transversal to ${}_{\bf p}\mathcal N_{{\bf r},i}$ at ${\bf r}^+$ as 
follows.
\par
By our choice of ${}_{\bf p}\frak W_{\bf r}$
and ${}_{\bf p}E_{\bf r}(\bf x)$, the set 
${}_{\bf p}U_{{\bf r}^+}$ is a neighborhood of ${\bf r}^+ = {\bf r}\cup {}_{\bf p}\vec{\frak w}_{\bf r}$ 
in $U_{{\bf p}^+}$.
Moreover ${}_{\bf p}\mathcal N_{{\bf r},i} = \mathcal N_{{\bf p},i}$.
Therefore we use the fact that 
${\bf x} \mapsto u_{\bf x}({\frak z}_{{\bf x},\ell+i})$ is 
transversal to $\mathcal N_{{\bf p},i}$ at ${\bf p}^+$
and ${\bf r}^+$ can be chosen to be close to ${\bf p}^+$
to show the required transversality at ${\bf r}^+$.
The case when we replace ${\bf p}$ by ${\bf q}$ is the same.
\end{proof}
We now consider the following commutative 
diagram.
\begin{equation}\label{diagram12-4}
\begin{CD}
{}_{\bf q}U'_{\bf r}
@ >>>
U_{{\bf p}{\bf q}} \\
@ VVV 
@ VVV 
\\
{}_{\bf p}U_{\bf r}
@ >>>
U_{{\bf p}}
\end{CD}
\end{equation}
Here ${}_{\bf q}U'_{\bf r}$ is a small neighborhood 
of ${\bf r}$ in ${}_{\bf q}U_{\bf r}$.
All the arrows are coordinate changes.
The commutativity of the diagram follows from Lemma \ref{lem79}.
\par
We first observe that two horizontal arrows are 
smooth at ${\bf r}$. This is the consequence of our 
choice of $({}_{\bf p}\frak W_{\bf r},{}_{\bf p}\frak W_{\bf r})$ and 
$({}_{\bf q}\frak W_{\bf r},{}_{\bf q}\frak W_{\bf r})$
(and of ${}_{\bf p}E_{\bf r}$ and ${}_{\bf q}E_{\bf r}$).
In other words it is nothing but Proposition \ref{prop122} and its 
proof.
Moreover they are open embeddings.
\par
We next observe that the left vertical arrow is smooth at ${\bf r}^+$.
This is a variant of Lemma \ref{lem126}
where $u_{\bf r}$ (which corresponds to $u_{\bf p}$) may not 
be pseudo holomorphic.
The proof of this variant is the same as the proof of  Lemma \ref{lem126}.
\par
Therefore the right vertical arrow, which is 
nothing but the map $\varphi_{\bf pq}$, is smooth 
at ${\bf r}$. The proof of smoothness of 
$\widehat{\varphi}_{\bf pq}$ is the same.
The proof of Proposition \ref{prop12777} is complete.
\end{proof}

\section{Proof of Lemma \ref{lem413}}
\label{proofoflem413}

\begin{proof}
The proof is divided into 4 steps.
In the first two steps we 
consider the case ${\bf p} = {\bf q}$.
We write $\frak W_{\bf p}^o = (\vec{\frak w}^o,
\{\varphi_{a,i}^{o,\rm s}\}, \{\varphi_{a,j}^{o,\rm d}\},
\{\phi^o_a\})$ $(o=1,2)$.
We will prove 
\begin{equation}\label{form131}
B_{\delta}({\mathcal X}_{k+1,\ell}(X,L,J;\beta);{\bf p},\frak W^1)
\subset 
B_{\epsilon}({\mathcal X}_{k+1,\ell}(X,L,J;\beta);{\bf p},\frak W^2),
\end{equation}
for sufficiently small $\delta$.
\par\smallskip
\noindent(Step 1)
We assume
$\vec{\frak w}^1 \subseteq \vec{\frak w}^2$
or $\vec{\frak w}^2 \subseteq \vec{\frak w}^1$:
Suppose $\vec{\frak w}^1 \subseteq \vec{\frak w}^2$
and $\# \vec{\frak w}^1 = \ell'$,  $\# \vec{\frak w}^2 = \ell'+\ell''$.
We consider the next diagram:
$$
\begin{CD}
\mathcal V_{{\bf p}\cup \vec{\frak w}^2} \times \Sigma_{{\bf p}\cup \vec{\frak w}^2}(\vec \epsilon)
@ >{\widehat{\Phi}^2_{\vec{\epsilon}}}>>
\mathcal C^{\rm d}_{k+1,\ell+\ell'+\ell''}
@ >{\pi}>>
\mathcal M^{\rm d}_{k+1,\ell+\ell'+\ell''} \\
@ VVV 
@ VVV 
@ VV{\frak{forget}_{\ell+\ell'+\ell'',\ell+\ell'}}V
\\
\mathcal V_{{\bf p}\cup \vec{\frak w}^1} \times \Sigma_{{\bf p}\cup \vec{\frak w}^1}(\vec{\epsilon'})
@ >{\widehat{\Phi}^1_{\vec{\epsilon'}}}>>
\mathcal C^{\rm d}_{k+1,\ell+\ell'}
@ >{\pi}>>
\mathcal M^{\rm d}_{k+1,\ell+\ell'}
\end{CD}
$$
Here the right half of the diagram is one induced by the 
forgetful map in the obvious way.
We put
\begin{equation}\label{for132}
\mathcal V_{{\bf p}\cup \vec{\frak w}^o}
=
\prod_{a\in \mathcal A_{\bf p}^{\rm s} \cup \mathcal A_{\bf p}^{\rm d}} \mathcal V^o_a 
\times [0,c)^{m_{\rm d}} \times (D^2_{\circ}(c))^{m_{\rm s}}.
\end{equation}
See (\ref{form33}). The second and the third factors of the right hand 
side of (\ref{for132}) are the gluing parameters of the nodes.
$\mathcal V^o_a$ is the deformation parameter of the 
$a$-th irreducible component of the source curve of 
${\bf p}\cup \vec{\frak w}^o$.
\par
The map $\widehat{\Phi}^1_{\vec{\epsilon}}$ 
(resp. $\widehat{\Phi}^2_{\vec{\epsilon'}}$)  is the map  
(\ref{hatPhi39}) defined  using $\frak W^1$ (resp. $\frak W^2$).
\par
The left vertical arrow is defined as follows.
$\Sigma_{{\bf p}\cup \vec{\frak w}^2}(\vec \epsilon)
\to \Sigma_{{\bf p}\cup \vec{\frak w}^1}(\vec{\epsilon'})$
is the inclusion. (The inclusion exists if $\vec{\epsilon'}$ is sufficiently 
small compared to $\vec{\epsilon}$.)
The forgetful map of the marked points induces a map
$\mathcal V_{a}^2 \to \mathcal V_{a}^1$.
This map together with the identity map of the second and the third factors 
of (\ref{for132}) defines the map in the left vertical arrow.
\par
The maps appearing in the diagram are all smooth.
The right half of the diagram commutes.
Note the left half of the diagram does {\it not} commute since 
we use different stabilization and trivialization data to define 
$\widehat{\Phi}^1_{\vec{\epsilon}}$ and $\widehat{\Phi}^2_{\vec{\epsilon}}$.
\par
Note however that the left half of the diagram 
commutes if we restrict it to the fiber of 
${\bf p}\cup \vec{\frak w}^1$.
Therefore choosing 
$\mathcal V_{{\bf p}\cup \vec{\frak w}^2}$ small the left half 
of the diagram commutes modulo a term whose $C^2$ norm is 
sufficiently smaller than $\epsilon$.
(\ref{form131}) is an easy consequence of this fact.
The proof of the case $\vec{\frak w}^2 \subseteq \vec{\frak w}^1$ is the same.
\par\smallskip
\noindent(Step 2)
We consider the case
${\bf p} = {\bf q}$ in general, where $\vec{\frak w}^1$, $\vec{\frak w}^2$ have no inclusion relations.
Let  $\frak W_{\bf p}^o$ $(o=1,2)$ be as above.
We take a weak stabilization data $\vec{\frak w}^3$ such that 
$\vec{\frak w}^1 \cap \vec{\frak w}^3 = \vec{\frak w}^2 \cap \vec{\frak w}^3 = \emptyset$.
We take stabilization and trivialization data $\frak W_{\bf p}^3$, $\frak W_{\bf p}^{13}$, $\frak W_{\bf p}^{23}$,
so that their weak stabilization data are given by   
$\vec{\frak w}_{\bf p}^3$, $\vec{\frak w}_{\bf p}^1 \cup \vec{\frak w}_{\bf p}^3$, 
$\vec{\frak w}_{\bf p}^2 \cup \vec{\frak w}_{\bf p}^3$, respectively.
Now we can apply Step 1 to the following 4 situations.
$(\frak W_{\bf p}^1,\frak W_{\bf p}^{13})$,
$(\frak W_{\bf p}^{13},\frak W_{\bf p}^{3})$, 
$(\frak W_{\bf p}^3,\frak W_{\bf p}^{23})$,
$(\frak W_{\bf p}^{23},\frak W_{\bf p}^{2})$.
Combining them we obtain the conclusion (\ref{form131}) of our case 
$(\frak W_{\bf p}^1,\frak W_{\bf p}^{2})$.
\par\smallskip
\noindent(Step 3)
The case ${\bf p} \ne {\bf q}$.
In this step we prove that given $\frak W_{\bf p}$
we can find $\frak W_{\bf q}$
such that (\ref{lem413occc}) holds: 
We first use the fact that ${\bf q}$ is sufficiently close to ${\bf p}$
to find $\vec{\frak w}_{\bf q}$ such that 
 ${\bf q} \cup \vec{\frak w}_{\bf q}$ is sufficiently close to ${\bf p}\cup \vec{\frak w}_{\bf p}$.
We then apply Sublemma \ref{sublem125} to obtain 
$\{\varphi_{{\bf q},a,i}^{\rm s}\}, \{\varphi_{{\bf q},a,j}^{\rm d}\},
\{\phi_{{\bf q},a}\}$.
We put
$\frak W_{\bf q} = (\vec {\frak w}_{\bf q},\{\varphi_{{\bf q},a,i}^{\rm s}\}, \{\varphi_{{\bf q},a,j}^{\rm d}\},
\{\phi_{{\bf q},a}\})$.
Diagram (\ref{diagram12-3}) 
commutes. (\ref{lem413occc}) is its immediate consequence.
\par\smallskip
\noindent(Step 4) Now the general case follows by combining Step 2 and Step 3.
\end{proof}

\include{index}
\printindex

\bibliographystyle{amsalpha}

\end{document}